\documentclass[11pt,a4]{amsart}

\usepackage{amsfonts,amsmath,amssymb,bbm}
\usepackage{verbatim}
\usepackage{color}
\usepackage{hyperref}
\usepackage{etoolbox}
\makeatletter
\let\ams@starttoc\@starttoc
\makeatother
\usepackage[parfill]{parskip}
\makeatletter
\let\@starttoc\ams@starttoc
\patchcmd{\@starttoc}{\makeatletter}{\makeatletter\parskip\z@}{}{}
\makeatother

\usepackage[parfill]{parskip}

\usepackage{enumerate}
\usepackage{pdfsync}
\usepackage{color}

\numberwithin{equation}{section}

\theoremstyle{plain}
  \newtheorem{thm}{Theorem}[section]
  \newtheorem{prop}[thm]{Proposition}
  \newtheorem{lemma}[thm]{Lemma}
   \newtheorem{cor}[thm]{Corollary}
\theoremstyle{definition}
  \newtheorem{remark}[thm]{Remark}
  
\newtheorem{example}[thm]{Example}

\newcommand{\Ne}{\mathcal{N}}
\newcommand{\Nb}{\overline{\Ne}}
\newcommand{\Be}{\mathcal{B}}
\newcommand{\Bb}{\overline{\Be}}
\newcommand{\Gx}{G_x}
\newcommand{\Esc}{\mathcal{E}}
\newcommand{\tho}{\tau}

\newcommand{\R}{\mathbb{R}}
\newcommand{\Rpo}{\R^+_{\geq 0}}
\newcommand{\C}{\mathbb{C}}

\newcommand{\eign}{\lambda_n}

\newcommand{\PPP}{\overline{\overline{\Pi}}}
\newcommand{\PP}{{\overline{\Pi}}}
\newcommand{\CBI}{${\rm{CBI}}(\psi,\phi)\:$}
\newcommand{\CBIb}{${\rm{CBI}}(\overline{\psi},\overline{\phi})\:$}
\newcommand{\CB}{${\rm{CB}}(\psi)\:$}

\newcommand{\Si}{\underline{\kappa}}
\renewcommand{\star}{*}
\newcommand{\piv}{\Phi_{\nu}}
\newcommand{\biv}{\overline{\Phi}_{\nu}}
\renewcommand{\liminf}{\varliminf}
\renewcommand{\limsup}{\varlimsup}

\begin{document}
\title[Smoothness of CBI semigroups]{Smoothness of continuous state branching  with immigration semigroups}
\setcounter{tocdepth}{1}

\author{M. Chazal}
\address{School of Operations Research and Information Engineering, Cornell University, Ithaca, NY 14853.}
\email{mc969@cornell.edu}

\author{R. Loeffen}
\address{School of Mathematics,
University of Manchester, Manchester M13 9PL, UK}
\email{ronnie.loeffen@manchester.ac.uk}

\author{P. Patie} \thanks{The authors are grateful to an anonymous referee for the careful reading of the manuscript and the insightful comments that helped improving the quality of the paper. This work was partially supported by the Actions de Recherche Concert\'ees  IAPAS, a fund of the Communaut\'ee francaise de Belgique}
\address{School of Operations Research and Information Engineering, Cornell University, Ithaca, NY 14853.}
\email{pp396@cornell.edu}

\begin{abstract}
In this work we develop an original and thorough  analysis of the (non)-smoothness properties of the  semigroups, and their heat kernels, associated to a large class of continuous state branching processes with immigration.  Our approach is based on an in-depth   analysis of the regularity of the absolutely continuous part of the invariant measure combined with  a substantial refinement of Ogura's  spectral expansion  of the transition kernels. In particular, we find new representations for the eigenfunctions and eigenmeasures that allow us   to derive  delicate uniform bounds that are useful for establishing the uniform convergence of the spectral representation of the semigroup acting on  linear spaces that we identify.    We detail several  examples which illustrate the variety of smoothness properties that CBI transition kernels may enjoy and also reveal that our results are sharp. Finally, our technique enables us to provide the (eventually) strong Feller property as well as   the rate of convergence to equilibrium in the total variation norm.
 \end{abstract}


\maketitle

%\tableofcontents

%\listofchanges[style=summary]

%\clearpage

\bibliographystyle{plain}

%\tableofcontents
\section{Introduction and main results}
The objective of this paper is to develop an original approach to obtain detailed information regarding the representation and the regularity properties  of the solution to the parabolic evolution equation
\begin{equation} \label{eq:cauchy}
\frac{d}{dt}u_t - \left(\mathbf{D}+\mathbf{L}\right)u_t = 0, \: u_0 = f,\end{equation}
where, for a smooth function $f$ on $x>0$, \begin{equation} \label{eq:def_diff} \mathbf{D} f(x)=\sigma^2 x f''(x) + \left(\mathbf{b}-\mathbf{m}x\right) f'(x) -(qx+a)f(x)\end{equation} and
\begin{equation}
\mathbf{L} f(x) =  \int_0^{\infty}\left( f(x+y)-f(x)-yf'(x)\mathbb{I}_{\{y<1\}}\right) \mathbf{K}(x,dy),
\end{equation}
with  the parameters $\sigma^2,q,a\geq0$, $\mathbf{m},\mathbf{b} \in \R$ and the L\'evy kernel $ \mathbf{K}(x,dy) = x \mathbf{\Pi}(dy) + \mu(dy)$ defined in terms of the L\'evy measures $\mathbf{\Pi}$ and  $\mu$,      satisfying  some mild conditions that are detailed in \eqref{eq:def_psib}-\eqref{eq:defintion_theta} below.  It is already worth pointing out that our analysis includes the situations where $\overline{\mathbf{\Pi}}(0^+)=\int_0^{\infty}\mathbf{\Pi}(dy)=\infty$ and where $\mathbf{A}=\mathbf{D}+\mathbf{L}$ is purely non-local, that is $\mathbf{D}\equiv 0$.

The linear operator $\mathbf{A}$  turns out to be the generator of the Feller semigroup  of a continuous state branching process with immigration (for short CBI). There exists a rich and fascinating literature devoted to the study of fine analytical and stochastic properties of CBI-semigroups,  see e.g.~Handa \cite{CBI-Handa}, Stannat \cite{Stanna}, Li and Ma \cite{Li-Ma}, Foucart and Bravo \cite{Foucart-Bravo}, Duhalde et al.~\cite{Foucart_Ma}, Caballero et al.~\cite{Caballero-Lamperti}, Abraham and Delmas \cite{Delmas-CBI}, Lambert \cite{Lambert-07} and the monograph \cite{Li_Book}. %, as well as to their applications, see e.g.~Filipovic et al.~\cite{FCBI},
However,  little seems to be known regarding the regularity properties of (at least) their transition kernels, %This is rather surprising  as  regularity properties beyond their intrinsic theoretical interests are substantial  in applied problems such as for instance simulation techniques or in statistics, see \cite{FCBI} and \cite{}.
something which may %rather
  be attributed to the fact that (non)-regularity properties of general non-local Markov semigroups and their heat kernels are fragmentally understood due to the lack of a comprehensive theory from both functional and stochastic analysis.\\
  In the framework of diffusions, that is, semigroups associated to differential operators, many theories have been successively designed to study this type of question. We mention for instance the techniques based  on Malliavin Calculus,  see e.g.~Hairer \cite{Hairer} for a recent and general account on these approaches, and, also analytical techniques based on H\"older estimates, see the monograph \cite{Lady-Sol-Ura-68} for a thorough account.

  Due to the generic role played by non-local  operators in the class of operators satisfying the maximum principle, see Courr\`ege \cite{Courr} for more details, the recent years have  witnessed   a fast growing literature devoted to the study of smoothness properties of the heat kernels or of the solution to parabolic problems associated to non-local Markov generators. %  Carry over techniques from the diffusions
%theory have had a limited success, providing some evidence towards the fact that developing smoothness of associated to parabolic problem with non local operator has crucial differences, something which have been confirmed by recent studies.
The approaches can be split into three main categories.  \\
The first one is based on classical Fourier analysis which requires precise information regarding the asymptotic decays for large arguments of the Fourier transform of the semigroup in order to derive smoothness properties. This approach has been successfully developed by Hartman and Wintner \cite{Hartman-Wintner-Levy} for providing a sufficient condition  expressed in terms of their symbol for the continuity of the densities of L\'evy processes, see also  Knopova and Schilling \cite{Schilling_Smooth} and the references therein for more detailed and interesting results in this direction.
 It was also used recently in Filipovic et al.~\cite{FCBI}  to study smoothness properties of the transition distributions of affine processes. In particular for the CBI transition  kernels that we study,  the authors  obtain a smoothness property on the real line     in the case where there is a diffusion component, i.e.~$\sigma^2>0$ in \eqref{eq:def_diff} above.  This latter result reveals two noteworthy limitations of the use of Fourier analysis in the context of transition kernels whose supports are not the entire real line, namely the  positive real line for CBI kernels.  On the one hand, this technique does not provide the optimal regularity property on the support of kernels (or their derivatives) that do not vanish at one end of the support.  On the other hand, it also requires precise information regarding the asymptotic behavior of the (modulus) of the Fourier transform which are often difficult to get without restrictive assumptions. For these reasons  we are lead to  develop an original  approach that focuses  on the  smoothness property of the density (or its absolute continuous part) of the CBI kernels  within their supports, that is,  allowing possible explosions at $0$ for the kernels or their successive derivatives. \\
A second approach is  based on Malliavin calculus which has been extended under various conditions to study regularity properties of the  solutions to stochastic differential equations driven by processes with jumps. However, these developments,  in the context of Markov kernels, considered merely dynamics with either a L\'evy kernel that is  homogeneous in space or of finite intensity, and/or with a diffusion component, see e.g.~Picard \cite{Picard},  Cass \cite{Cass} or the original paper by Fournier \cite{Fournier} and the references therein. \\
Finally there are some substantial results of  analytical nature based on Harnack inequalities which have been obtained recently for getting H\"older continuity properties of the solution to parabolic equations involving integro-differential operators, see e.g.~Caffarelli and Silvestre  \cite{Silvestre_Caf}.
However all these techniques are not general enough to be
applied in our context. This is due to either a lack of symmetry and/or non-homogeneity
of the L\'evy kernel, or, unboundness of the drift coefficient, or, the possible absence of
diffusion part, or, simply, the non-local feature of the generators.

  We propose an alternative approach that stems on a combination of  an original methodology  that is developed to deal with the smoothness properties of distributions on $\R^+=(0,\infty)$  and  the spectral expansion of CBI-semigroups.  For the former, we show that the invariant measure of a CBI-semigroup is solution to a convolution equation that we study in depth to derive the regularity properties. We already emphasize that these developments could be used  in a larger context as, for instance, for the class of positive infinitely divisible laws. For the second part, we exploit and develop substantially a spectral expansion of the kernel of CBI-semigroups whose initial form was found by Ogura \cite{Ogura-70}. The novelty of this result lies on the lack of a spectral theorem for non-self-adjoint operator and Ogura overcomes this difficulty by first suggesting a candidate as an eigenvalues expansion for their transition kernel and by means of the Laplace transform techniques show that it indeed corresponds to the right one. These series expansions involve three spectral components, the eigenvalues which take in this case a simple form, the eigenfunctions and the eigenmeasures, that is, when these latter are absolutely continuous they correspond, in some sense, to the eigenfunctions of the adjoint semigroup. Unfortunately Ogura's representation of these spectral functions  does not allow one to study their regularity properties. Thus, we start by  providing new appropriate representations for the eigenfunctions as well as for  the eigenmeasures that are useful for deriving delicate uniform bounds and for investigating their smoothness properties.  This approach which  offers another view of Ogura's work also enables us to describe linear functional spaces for which the spectral expansion of the semigroup remains valid. This is  critical to determine regularity properties of the semigroup, that is formally the solution to the associated Cauchy problem \eqref{eq:cauchy}.

 Our new developments  on the spectral decomposition  of these semigoups also enable to apply some  transforms that are known to carry over  the spectral expansion. For instance, by means of a tensorization procedure, our results extend directly  to  any dimension $d \geq 2$. Furthermore this latter could be associated to the subordination   in the sense of Bochner in order to obtain similar results for a larger class of $d$-dimensional assymetric Markov processes with two-sided jumps. We shall not present the details of these standard arguments therein but we refer the interested reader to the excellent monograph \cite{Bakry_Book} for a description.

%We write $\R^+ = (0,\infty)$ and $\R^+_0 = [0,\infty)$. We denote by $B(\Bbb R^+_0) $ the set of Borel measurable functions on   $\Bbb R^+_0$,  $B_b(\R^+_0)$ its subset  of bounded  functions and  $C_b(\R^+_0)$  its subset  of bounded continuous functions. Their respective subsets of functions with compact support are denoted by  $B_{b,c}(\R^+_0)$, $C_{b,c}(\R^+_0)$. Finally  $C_0^0(\R^+_0)$ stands for  the set of continuous functions on $\R^+_0$ vanishing at infinity. Unless specifically mentioned otherwise, $f^{(n)}(x)$ stands for the $n$-th derivative of $f$ at the point $x$.
%For any Borel measure $\mu$ on $\R^+_0$ and any $f,g \in B(\R^+_0)$ we write
%$ \mu f = \int_{\Bbb R^+} f(y)\mu(dy)$ and $ \langle f, g\rangle =\int_{\Bbb R^+} f(y)g(y)dy $.

Now writing $\Rpo =[0,\infty)$, denoting by $B_b(\Rpo)$ the set of bounded Borel measurable functions on $\Rpo $ and by $C_0(\Rpo)$ the set of continuous functions on $\Rpo$ vanishing at infinity and  following Kawazu  and Watanabe \cite{Watanabe-CBI}, we say that $P = (P_t)_{t\geq 0}$ is a CBI-semigroup (resp. CB-semigroup)  if it is a non-negative strongly continuous contraction semigroup on  $C_0(\Rpo)$,
%if it is a non-negative  contraction semigroup on  $B_b(\Rpo)$,  that is strongly continuous on $C_0(\Rpo)$,  with the Feller property,
%\begin{equation*}
%P_t(C_0(\Rpo)) \subset C_0(\Rpo), \mbox{  for all } t\geq 0,
%\end{equation*}and
satisfying
\begin{equation}\label{eq:CBI-initial}
P_t(\Lambda_0) \subset \widetilde{\Lambda_0} \quad  \mbox{(resp.~} P_t(\Lambda_0) \subset \Lambda_0 \mbox{)}, \mbox{  for all } t\geq 0,
\end{equation}
where, writing $e_{\lambda}(\cdot)=e^{-\lambda \cdot}$, we have set
\begin{equation*}
\Lambda_0 =(e_{\lambda})_{\lambda \geq 0}, \quad
\widetilde{\Lambda_0} =(ce_{\lambda})_{c, \lambda \geq 0}.
%\quad \mbox{ and } \quad \Lambda =\textrm{Span}(e_{\lambda},\lambda > 0)
\end{equation*}
%where $\textrm{Span}$ of a set stands for its linear hull. Note that by the Stone-Weierstrass theorem $\Lambda$ is dense in $C_0(\R_0^+)$.
From \cite{Watanabe-CBI}, it is well known that any CBI-semigroup $P$ is characterized, via its Laplace transform, by a unique couple of functions $(\psi, \phi)$ that are respectively called the branching and immigration mechanisms and we denote it by ${\rm CBI}(\psi, \phi)$ and write  simply ${\rm CB}(\psi)={\rm CBI}(\psi, 0)$.

We point out that \emph{the main results of this paper hold under the general conditions,  stated in \eqref{eq:def_psib}-\eqref{eq:defintion_theta} below, on the mechanisms $(\overline{\psi}, \overline{\phi})$, defined in terms of the parameters $\mathbf{b},\mathbf{m}, q,a$ and $\mathbf{\Pi}$ that appear in the expression of the generator in \eqref{eq:def_diff}.} However, for sake of presentation, we prefer to state first  these results   in a slightly more restrictive setting that we now describe.  Note that Proposition \ref{prop:transf}\eqref{it:transf_theta} explains that the semigroups under the  two sets of conditions are related by a Doob h-transform which enable to  readily transfer the properties from one semigroup to the h-transformed one.

%From Kawazu  and Watanabe \cite{KawaWata71}, it is well known that there exists a bijection between the set of CBI-semigroups (resp.~CB-semigroup) and the set $ \Nb\times  \Bb$ (resp.~the set $\Nb $). More specifically, any ${\rm CBI}$ semigroup $P$ is characterized by a unique couple $(\psi, \phi) \in \Nb \times \Bb$ via its Laplace transform which is given by
 %\begin{equation}\label{CBI}
 %P_te_{\lambda}(x) = e^{-\int_0^t\phi(v_s(\lambda))ds -x v_t(\lambda)}, \quad t\geq 0, x\geq 0,\lambda \geq 0
 %\end{equation}
% where $v_t$ is the unique non-negative solution to
% \begin{equation}\label{Riccati}
% v_t(\lambda )+ \int_0^t \psi(v_s(\lambda))ds =\lambda, \quad t\geq 0,\lambda \geq 0.
% \end{equation}
%The functions $ \psi$ and $\phi$ are respectively called the branching and immigration mechanisms.
%\textcolor{magenta}{postpone : Since $C_0(\Bbb R^+)$ is dense in  $L^2(\Bbb R^+)$, the ${\rm CBI}(\psi,\phi)$ semigroup $P$ has a unique extension on $L^2(\Bbb R^+)$ satisfying $\Vert P_t f \Vert_2 \leq \Vert  f \Vert_2$.}
We will denote by  $\Ne$ the set of functions $\psi:\Rpo \mapsto \Rpo$ defined, for $u \geq 0$, by
 \begin{equation} \label{eq:def_psi}
\psi(u) =\sigma^2 u^2 + m u+ \int_0^\infty\left(e^{-u r}-1+ur \right)\Pi(dr),
 \end{equation}
 where $\sigma \geq 0$, $m>0$  and  $\Pi$  is a non-negative Borel measure  on
 $\R^+=(0,\infty)$
 , and, which satisfy  the following two conditions
\begin{eqnarray}
 \int^\infty \frac{du }{\psi(u)}<\infty \label{eq:standing_assumption}
 \mbox{ and } \psi \in \mathcal{H}(R_\psi) \textrm{ with } R_\psi>0,  %\label{Analicity_of_the_mechanisms}
\end{eqnarray}
where throughout $\mathcal{H}(R)$ (resp.~$\mathcal{H}_{(a,b)}$) is the set of functions holomorphic on the open disk $D(0,R)$  (resp.~on the strip $a<\Re(z)<b$), where we understand that $R$ is the radius of convergence of their Taylor series at $0$.  Note that the second requirement in \eqref{eq:standing_assumption} implies that  $\int_0^\infty (r \wedge r^2) \Pi(dr) < \infty$ and  that the integral in \eqref{eq:def_psi} is well-defined. Of course, here and below, we mean that $\psi$ as defined in \eqref{eq:def_psi} extends to a holomorphic function and we keep the same notation for its extension.
Note that, by Sato \cite[Theorem 25.17]{Sato-99}, $\psi \in \mathcal{H}(R_\psi)$ is equivalent to assume that $\psi$ is holomorphic on the half-plane $\Re(z) > -R_\psi$. This standard equivalence also holds for Laplace transforms.
Next,  we denote by $\Be$ the set of Bernstein functions on $\Rpo$, that is, the functions $\phi:\Rpo \mapsto \Rpo$ such that
 \begin{equation}\label{eq:def_Bernstein}
\phi(u) =  b u  + \int_0^\infty \left(1-e^{-u r} \right) \mu(dr)=u \left( b+ \int_0^{\infty}e^{-ur}\overline{\mu}(r)dr\right)
 \end{equation}
 where   $b \geq 0$ and  $\mu$  is a non-negative Borel measure  on
 $\R^+$ satisfying $ \int_0^\infty (1 \wedge r) \mu(dr) < \infty$ and such that
  \begin{eqnarray}
  \phi \in \mathcal{H}(R_\phi) \textrm{ with }  R_\phi>0.  \label{Analicity_of_the_mechanismsp}
 \end{eqnarray}
Here in \eqref{eq:def_Bernstein} we have set $\overline{\mu}(r)=\int_r^{\infty}\mu(dy)$, for all $r>0$. %We define by $R_\psi$ (resp.~$R_\phi$) the radius of convergence of the Taylor series at $0$ of $\psi$ (resp.~$\phi$). Then by  the assumptions \eqref{eq:standing_assumption} and \eqref{Analicity_of_the_mechanismsp}, we have $R_\psi,R_\phi>0$ and by definition $\psi\in \mathcal H(R_\psi)$, $\phi\in\mathcal H(R_\phi)$.
%\added[id=RL]{Note that I have redefined $R_\psi$ and $R_\phi$ such that they are uniquely defined numbers. This is needed so that $R_A$ and $\bar{\lambda}_t$ are also uniquely defined numbers.}
Note that for any $\psi \in \Ne$, we have
\begin{eqnarray} \label{eq:whl}
\psi(u) &=& u\phi_p(u) = u\left(\sigma^2 u + m + \int_0^{\infty}(1-e^{-ur})\overline{\Pi}(r)dr\right)\\
&=& u\left(\sigma^2 u + m + u \int_0^{\infty}e^{-ur}\PPP(r)dr\right), \nonumber
\end{eqnarray}
where $\PP(y)=\int_y^{\infty}\Pi(dr)$,  $\PPP(y)=\int_y^{\infty}\PP(r)dr$ and  $\phi_p \in \Be$ is the descending ladder height exponent.
Under these conditions, the  ${\rm CBI}(\psi, \phi)$ semigroup is conservative and subcritical with
 \begin{equation}\label{theta_positive}
 \psi(0)=0,   {\psi}^{(1)}(0)=m  >0,   \int_{\lambda }^{\infty} \frac{du}{ \psi(u)} < \infty \mbox { for all } \lambda >0 \mbox{ and }
\int_{0}^{\infty} \frac{du}{ \psi(u)} =\infty.
 \end{equation}
 where  we have used the notation $f^{(\mathfrak{p})}(x)=\frac{d^\mathfrak{p}}{dx^\mathfrak{p}} f(x)$
 for the $\mathfrak{p}$-th derivative of $f$, for some integer $\mathfrak{p}$.
 We refer to e.g.~\cite[Theorem 3.8]{Li_Book} for a proof of \eqref{theta_positive}. We also  point out that, in fact, it is easy to check that the condition ${\psi}^{(1)}(0)=m>0$ yields that $\int_{0}^{\infty} \frac{du}{ \psi(u)} =\infty$. From these considerations, we deduce that the mapping
 \begin{equation}\label{def_A}
 \lambda \mapsto A(\lambda) = \exp\left(-m\int_\lambda^\infty \frac{du }{\psi(u)}\right) %=e^{-m F(\lambda)}
 \end{equation}
 is an increasing bijection from  $\Rpo$ to $[0,1)$ with  inverse function denoted by $B$, i.e.~$B: [0, 1)\longmapsto \Rpo$ satisfies
\[
   \exp \left( -m\int_{B(z)}^{\infty} \frac{du }{ \psi(u)} \right) =z. \]
In addition, under the  assumptions \eqref{eq:standing_assumption} and \eqref{Analicity_of_the_mechanismsp},
and writing
\begin{equation} \label{eq:def_Phi}
 \piv(\lambda) = 	\int_0^{\lambda} \frac{\phi(u)}{\psi(u)}du,
 \end{equation}
we shall show in Lemma \ref{Analicity_of_A_B_F} below that, there exists some $0< R_0 \leq 1$ such that, for all $x\geq 0$,  the function
\begin{equation}\label{eq:def_Gx}
z \longmapsto \Gx(z) = e^{ -x B(z)+\piv(B(z))} \in \mathcal{H}(R_0).
\end{equation}
 %Note that $R_0$ is independent of $x>0$. %, see Lemma \ref{Analicity_of_A_B_F}.
  Then, with this notation, we have the following representation of the Laplace transform of the Feller semigroup $P$ which is due to Ogura \cite[Proposition 1.2]{Ogura-70} and valid for any $t,x,\lambda\geq 0$,
\begin{equation} \label{eq:lt_cbi}
	P_te_\lambda (x) = e^{-\piv(\lambda)}\Gx(A(\lambda)e^{-mt}).
\end{equation}
Next, we set
\begin{equation}\label{def:t0}
T_0=-\frac{\ln(R_0)}{m }
\end{equation}
and we denote by $W$ the so-called scale function associated to the spectrally negative L\'evy process whose law is determined by its Laplace exponent $\psi$. More precisely, the function $W: \Rpo \rightarrow \Rpo$ is characterized by its Laplace transform as follows, for any $u > 0$,
\begin{equation}\label{eq:def_W}
 \psi(u)\int_0^\infty e^{-u y} W(y)dy = 1.
 \end{equation}
We shall check, in  Lemma \ref{eq:propW} below, that the function $W$  satisfies $W(0)=0$ and is in $C^1(\R^+)$ with the derivative being positive on $\mathbb R^+$. With $b$ the drift parameter and $\mu$ the L\'evy measure associated to $\phi$, see \eqref{eq:def_Bernstein} and recalling $\overline{\mu}(r) =\int_r^\infty\mu(dy)$  we define, for all $y>0$,
\begin{equation} \label{eq:def_k}
\kappa(y) =
%\begin{cases}
%\sigma^2 W^{(1)}(y) +  \int_0^yW^{(1)}(y-r)r\Pi(dr) \quad &\textrm{ if } \phi\equiv 0,\\
bW^{(1)}(y) + \int_0^yW^{(1)}(y-r)\overline{\mu}(r)dr,  %\quad & \textrm{ otherwise},
%\end{cases}
\end{equation}
and, we set $\underline{\kappa}(0^+)=\liminf_{y\to 0} \kappa(y) $,   $\overline{\kappa}(0^+)=\limsup_{y\to 0} \kappa(y) $ and we define the integer $\Si$ by
\begin{equation}\label{def:kappa_bar}\Si=
\lceil \underline{\kappa}(0^+) \rceil -1,
    \end{equation}
where $\lceil .\rceil$ is the ceiling function and we understand that $\lceil\infty\rceil=\infty$.
%\begin{equation}\label{def:kappa_bar}\Si
%  = \begin{cases}
%  \infty & \text{if $\underline{\kappa}(0^+)=\infty$}, \\
%    \max\{l\in\mathbb N; l<\underline{\kappa}(0^+) \}    & \text{if $1<\underline\kappa(0^+)<\infty$,} \\
%    0 & \text{if $\underline{\kappa}(0^+)\leq 1$.}
%    \end{cases}
%\end{equation}
Note that when $\phi \equiv 0$, i.e.~$P$ is a CB($\psi$), obviously $\kappa \equiv 0$.
Next, we denote %, for some $\varepsilon >0$ and
for $t>T_0$,
\begin{equation}\label{eq:defBtbar_and_H}
\bar\lambda_t =
\begin{cases}
\min\left(- B \left( 2-e^{m(t-T_0)} \right), R_\phi\right) & \text{if $T_0<t<T_0+\frac 1m \ln \left( 2-A(-R_A) \right)$}, \\
R_A\wedge R_\phi & \text{if $t\geq T_0+\frac 1m \ln \left( 2-A(-R_A) \right)$},
\end{cases} % <0,
\end{equation}
where $R_A$ is the radius of convergence of the Taylor series at $0$ of $A$. Note that in Lemma \ref{Analicity_of_A_B_F} it is shown that $A \in \mathcal{H}(R_A)$ with $R_A>0$   and thus $B$, the inverse of $A$, is well-defined for $z>A(-R_A)$.
Writing  $f e_{\lambda} (x)  = f(x) e^{-\lambda x}, \: x\geq 0$, we set
\begin{equation}\label{eq:defBtbar_and_Ht}
 \mathcal{D}_t = \{ f:\Rpo\to\mathbb R \ \mbox{ measurable}; \: f e_{\lambda}  \in L^{\infty}(\R^+) \ \text{for some $\lambda<\bar{\lambda}_t$} \},
\end{equation}
that is, $f \in \mathcal{D}_t$ if there exists $C>0$ and $\lambda<\bar{\lambda}_t$ such that  $\vert f(x)\vert \leq C e^{ \lambda x}$, for almost every
(a.e.) $x\geq 0$. Plainly $\mathcal{D}_t$ is a linear space and since $B(0)=0$, $\mathcal D_t$ contains the set of bounded measurable functions if $t>T_0+\frac{\ln 2}{m}$.
We  write
\begin{equation*}
\begin{split}
L^1_{loc}(\R^+) = & \{f:\R^+ \to\mathbb R \mbox{ measurable}; \: \mbox{ for any } a >0, \: \int_0^a \vert f(y) \vert dy < \infty
\}, \\
L^1(\R^+) = & \{f:\R^+\to\mathbb R \mbox{ measurable}; \:  \int_0^\infty \vert f(y) \vert dy < \infty
\}.
\end{split}
\end{equation*}
Further, for $E\subseteq \R$,   $C(E)$,  respectively $C^p(E)$ for $p =1, 2, \cdots, \infty $   stand for the space of continuous, respectively $p$ times  continuously differentiable functions on $E$. Similarly, for any $E_i\subseteq \R$, $i=1,2,3$, $C^{\infty^2,k}(E_1\times E_2 \times E_3)$  %(resp.~simply $C(E)$)
 denote the space of infinitely continuously differentiable functions with respect to the two first variables and $k$ times with respect to third one on $E_1\times E_2 \times E_3$. % (resp.~on $E$).  Further, %let $L^1_{loc}(\Rpo)$ be the space of functions on $(0,\infty)$ which are (absolutely) integrable on bounded sets in $\R^+$.
We also denote by $C_b(\Rpo)$, the set of bounded continuous functions on $\Rpo$ and we set \begin{equation*}
\Lambda =\textrm{Span}(e_{\lambda})_{\lambda > 0},
\end{equation*}
where $\textrm{Span}$ of a set stands for its linear hull. Note that, by the Stone-Weierstrass theorem, $\Lambda$ is dense in $C_0(\Rpo)$.
Finally, for all $t,x\geq 0$, we denote by $P_t(x,dy)$ the transition kernel of the CBI-semigroup $P$ and by $\delta_a(d y)$ the dirac measure at $a$.
We are now ready to state the first main result of this paper.

\begin{thm}\label{thm:regularity}
  Let $P$ be a CBI$( \psi,   \phi)$ semigroup with $(\psi,\phi) \in \Ne \times \Be$.  Then, for all $f \in \mathcal{D}_t \cup \Lambda$, $z\mapsto P_zf \in \mathcal{H}_{(T_0,\infty)}$.
   Moreover,    for any $t>T_0$, the following hold.
 	 	\begin{enumerate}
      \item \label{it:Cinfty}	\label{it:sci} $P_tf \in C_b(\Rpo) \cap C^\infty(\Rpo)$ for all $f \in \mathcal{D}_t \cup \Lambda$.
      \item \label{it:CinftyC0} $P_tf \in C_0(\Rpo) \cap C^\infty(\Rpo)$ for all $f \in (\mathcal{D}_t \cap C_0(\Rpo))\cup \Lambda$.
     \item \label{it:sf} $ P_tf \in C_b(\Rpo)$ for all $f \in B_b(\Rpo)$, that is, $P$ is (eventually) strong Feller.
 \item \label{it:sh} For all $x\geq 0$, there exists  a function  $y\mapsto p_t(x,y)$ such that
 \begin{equation}\label{eq:deck}
 P_t(x, dy) = e^{-\biv }G_x(e^{-mt}) \delta_0(dy) + p_t(x,y)dy, \quad y\in\mathbb R.\end{equation}
 Note that one can take $p_t(x,y)=0$ for all $y<0$.
 Thus, for all $x \geq 0$, $ P_t(x, dy)$ is  absolutely continuous if and only if
  		\begin{equation} \label{absolute_continuity_condition0}
  	 \biv = \lim_{\lambda \to \infty} \piv(\lambda) = \infty,
  		\end{equation}
  which holds if $\Si>0$. Moreover, in any case,
 \begin{enumerate}[a)]
 \item  \label{it:4a} if  $\Si\geq 1$, then $(t, x, y)\longmapsto p_t(x,y)$ is $C^{\infty^2,\Si-1}((T_0, \infty )\times \mathbb R^+\times \boldsymbol{\mathbb{R}})$,
 \item \label{it:defq}  $(t, x, y)\longmapsto p_t(x,y) \in C^{\infty^2, \Si+\bar{\mathfrak{q}}}((T_0, \infty )\times \mathbb R^+\times \mathbb R^+)$ where
 \begin{enumerate}[b1)]
  \item\label{it:mainthm_qis0}  $\bar{\mathfrak{q}}=0$ if either $\Si\geq 1$ or $\Si=0$  and $\overline{\kappa}(0^+)<\infty$,
  \item  \label{it:b2} $\bar{\mathfrak{q}} = \sup\{q\geq 1;\: \kappa, W \in C^{{\mathfrak{q}}}(\R^+) \}$ if $\kappa \in C^{1}(\R^+)$, $\kappa^{(1)}\in L^1_{loc}(\R^+)$  and $\underline{\kappa}(0^+)=\overline{\kappa}(0^+)$.
 \end{enumerate}
 \end{enumerate}
 \end{enumerate}
 \end{thm}
 \begin{remark} \label{rem:main}
 Note that in  % \added[id=RL,remark={I do not like how the labels of the subitems are printed when referenced. I have tried to improve this, see just before `begin document' but does not look good yet. Will come back to this.}]{See footnote}
   Theorem \ref{thm:regularity}\eqref{it:b2} the condition $\kappa \in C^{1}(\R^+)$ ensures that in this case $\bar{\mathfrak{q}} \geq 1$ as we shall prove that, in our setting, $W \in C^{1}(\R^+)$, see Lemma \ref{eq:propW} below. \end{remark}
   \begin{remark}
   We also point out that in a recent paper Li and Ma \cite{Li-Ma}  have shown  by means of an elegant coupling argument the strong Feller property of CBI semigroups satisfying the first condition in \eqref{eq:standing_assumption} and having a linear immigration, i.e.~$\phi(u)=bu$.
 \end{remark}
Another interesting by-product of our analysis is the following precise estimate regarding the speed of convergence to stationarity in the total variation norm, which we recall to be defined for a signed measure $\mu$ on $\Rpo$ by $|| \mu ||_{{\rm{TV}}} = \sup_{E \in \mathcal{B}(\Rpo)} |\mu(E)| $, with $\mathcal{B}(\Rpo)$ the set of Borelians of $\Rpo$.

\begin{prop}\label{cor:exp-ergodicity} Let  $ P$ be a \CBI semigroup  with $(\psi,  \phi) \in \Ne \times \Be$, then  $P$ admits a unique  invariant probability measure $\mathcal V$ on $\Rpo$. Moreover $P$ is exponentially ergodic, in the sense that  there exist  $C>0$ and $\overline{B} >0$ such that,  for any $x\geq 0$ and $t>\underline{T}=T_0 + \frac{1}{m} \ln (2-A(-R_A))$, we have
\begin{equation}\label{eq:total-variation}
||  P^{\mathcal{V}}_t(x)||_{{\rm{TV}}} \leq Ce^{\overline{B} x} %e^{-\piv(-\lambda)}
\frac{e^{-m(t-\underline{T})}}{1- e^{-m(t-\underline{T})}}
%\frac{e^{-ms}}{1-e^{-s}}
\end{equation}
where we have set $P^{\mathcal{V}}_t(x)(.)= P_t(x, .) -{\mathcal V}(.)$.
%with $s = mt +\ln(R) - \ln(2- A(-\lambda))$ and
\end{prop}
\begin{remark}
We point out that  the exponential ergocity of CBI semigroups have been studied recently under various restrictive conditions. For instance, Li and Ma \cite{Li-Ma} (resp.~Jin and al.~\cite{Jin-Rudiger}) proved this fact by means of a coupling argument (resp.~a Forster-Lyapunov function argument) when the immigration mechanism is linear, i.e.~$\phi(u)=bu, b>0$ (resp.~the branching mechanism is quadratic, i.e.~$\psi(u)=\sigma^2 u^2 +mu$).
\end{remark}

 %Next, we introduce the following notation on  asymptotic behaviours that will remain in force throughout the  paper.
%\begin{eqnarray*}
%f &\asymp &  g \textrm{ means that }  \exists \:  c>0  \textrm{ such that }
%c \leq  \frac{f}{g} \leq c^{-1}, \\
%f &\stackrel{a}{\sim}& g  \textrm{ means that } \lim_{x \to a}\frac{f(x)}{g(x)}=1, \textrm{ for some }  a\in \R\cup\{\pm\infty\},\\
%f &\stackrel{a}{=}& {\rm{O}}\left(g\right)\textrm{ means that } \limsup_{x \to a } \left| \frac{f(x)}{g(x)}\right| < \infty.
%\end{eqnarray*}
We proceed by stating some sufficient conditions, expressed in terms of the characteristics of both mechanisms, for the mapping $\kappa$ defined in \eqref{eq:def_k} to satisfy the specific conditions appearing in the smoothness properties of the absolutely continuous part of the transition kernel.

 \begin{prop} \label{lem:kap} \label{lem:diff_kappa} We have the following.
 \begin{enumerate}[(1)]
 \item \label{it:prop121}
\begin{enumerate}[(i)]
\item\label{item:lem_kap_sigma} If $\sigma^2  + b >0$  then
\begin{equation}
\Si(0^+)=\overline{\kappa}(0^+)=\frac{b}{\sigma^2} \in [0,\infty].
\end{equation}
\item \label{item:lem_kap_regvar} If $\sigma^2  + b =0$ then
\begin{equation}
\underline{\kappa}(0^+) \geq \liminf_{y\to 0}\overline{\mu}(y) W(y).
\end{equation}
 In particular $\underline{\kappa}(0^+)=\infty$ if  $\overline{\Pi}(y) \stackrel{0}{=} {\rm{O}}(  y^{-\alpha})$ and $\frac{1}{\bar{\mu}(y)}\stackrel{0}{=} {\rm{O}}(  y^{\beta})$ with $1<\alpha<1+\beta<2$, where throughout
 $f \stackrel{a}{=} {\rm{O}}\left(g\right)$ for $a \in [-\infty,\infty]$ means that $\limsup_{x \to a } \left| \frac{f(x)}{g(x)}\right| < \infty$. Furthermore, $\overline{\kappa}(0^+)=0$  if $\bar{\mu}(0^+)<\infty$.
\end{enumerate}
\item\label{it:prop122} Assume  that $\underline\kappa(0^+)=\overline\kappa(0^+)<\infty$.  Then
$\kappa\in C^1(\R^+)$ and  $\kappa^{(1)}\in L^1_{loc}(\R^+)$   if one of the following holds.
\begin{enumerate}[(i)]
\item \label{it:sigma_nul_W}  $\sigma=0$, $W^{(1)},\bar{\mu}\in C^1(\R^+)$ and for some $\delta>0$, we have that $W^{(2)}$ is non-positive on $(0,\delta)$ and
\begin{equation*}
-\int_0^{\delta} \frac{y\bar{\mu}^{(1)}(y) dy}{\int_0^ y \overline{\overline \Pi}(r) d r } d y <\infty.
\end{equation*}
\item \label{it:mu_finite} $\sigma=0$, $\bar{\mu}(0^+)<\infty$,  $\bar{\mu}\in C^1(\R^+)$ and $\bar{\mu}^{(1)}\in L^1_{loc}(\R^+)$.
\item  \label{it:sigma_pos} $\sigma>0$ and $\bar{\mu} \in C(\R^+)$. Moreover, if in addition  $b=0$, $\bar{\mu}(0^+)<\infty$,  $\bar{\mu}\in C^1(\R^+)$ and $\bar{\mu}^{(1)}\in L^1_{loc}(\R^+)$ then  $\kappa\in C^2(\R^+)$.
\end{enumerate}
\end{enumerate}
%Moreover,   if $\sigma>0$,   $b=0$, $\bar{\mu}(0)<\infty$,  $\bar{\mu}\in C^1(0,\infty)$ and $\bar{\mu}^{(1)}\in L^1_{loc}(0,\infty)$, then $\kappa\in C^2(0,\infty)$.
	\end{prop}
\begin{remark}
Note that if $\psi(u)=u^2 + u$, then
$W^{(1)}(y) = e^{-y} $ and thus
\[ \kappa(y) = e^{-y} \left( b  + \int_0^y e^{-r} \overline{\mu}(r)dr\right), \]
which implies $\kappa(0^+)=b$ as in   Proposition \ref{lem:kap}\eqref{item:lem_kap_sigma}.
Assume further that $\mu(dr) = \delta_1(dr)$. Then $\overline{\mu}(r) = \mathbb{I}_{\{r\leq 1\}} \notin C(\R^+) $ and
\[ y \mapsto \kappa(y) =  e^{-y} \left( b  +\left(1-e^{-y\wedge 1} \right)\right) \notin C^1(\R^+), \]
 which shows that Proposition \ref{lem:kap}\eqref{it:sigma_pos} is sharp.
\end{remark}
\begin{remark}\label{remark_smoothness_kappa}
Proposition \ref{lem:diff_kappa} gives some conditions under which $\kappa$ is in $C^1(\R^+)$ and $\kappa^{(1)}$ is in $L_{loc}^1(\R^+)$. In   part \eqref{it:sigma_nul_W}, it is assumed that $W^{(2)}$ exists, is in $C(\R^+)$ and is further negative in a neighbourhood of zero. Unfortunately, not much is known about which L\'evy measures $\Pi$ imply these conditions on the scale function  (in the  situation where $\sigma=0$ and $\int_0^1 r\Pi(d r)=\infty$). It is known that if $\overline{\overline \Pi}$ is log-convex, then $W^{(1)}$ is non-increasing (but not necessarily in $C^1(\R^+)$), whereas if $\overline \Pi$ is completely monotone, then $W^{(1)}$ is completely monotone, see  \cite[Chap.~11]{SchillingSongVondracek10}. Recall that a non-negative function $f$  is completely monotone if it is in $C^\infty(\R^+)$ and $(-1)^n f^{(n)}(x)\geq 0$ for all $x>0$ and $n\in\mathbb N$. Higher order differentiability properties of $\kappa$ can be straightforwardly deduced from the expressions for $\kappa^{(1)}$ and $\kappa^{(2)}$, see e.g.~\eqref{eq:k2} below, given in Proposition \ref{lem:diff_kappa}  in combination with Lemma \ref{lem_conv_diff_split} below, upon imposing higher order continuous differentiability   on $W^{(1)}$ and $\bar{\mu}$. If $\sigma>0$, the problem of higher order (non-)differentiability   of $W^{(1)}$ is  studied  in Chan et al.~\cite{Chan-Kyp-Savov}. In particular, Theorem 2 in \cite{Chan-Kyp-Savov} says that if the Blumenthal-Getoor lower index $\inf\{\beta>0; \int_0^1 r^\beta \Pi(d r)<\infty\}<2$, then $W^{(2)}\in C^{n+1}(\R^+)$ if and only if $\overline \Pi\in  C^{n}(\R^+)$. When $\sigma=0$, again little is known about the existence of higher order derivatives except in the aforementioned case where $\overline \Pi$ is completely monotone.
\end{remark}
\noindent We emphasize that in fact the main results of this paper extend to the larger class of CBI$(\overline{\psi}, \overline{\phi})$
semigroups whose mechanisms $(\overline{\psi}, \overline{\phi})$ are in $\Nb \times \Bb$, which corresponds to the set of functions of the form
 \begin{equation} \label{eq:def_psib}
\overline{\mathbf{\psi}}(u) =\sigma^2 u^2 + \mathbf{m} u+ \int_0^\infty\left(e^{-u r}-1+ur\mathbb{I}_{\{|r|<1\}} \right)\mathbf{\Pi}(dr)-q,
 \end{equation}
 where $\sigma^2,q\geq 0, \mathbf{m}\in \R$ and $\mathbf{\Pi}$  is a L\'evy measure  satisfying $\int_0^\infty (1 \wedge r^2) \mathbf{\Pi}(dr) < \infty$
 %\replaced[id=RL, remark={\text{Do we need here $\int_1^\infty r\Pi(d r)<\infty$}?}]{$\int_0^\infty (1 \wedge r^2) \Pi(dr) < \infty$}{$\int_0^\infty ( r \wedge r^2) \Pi(dr) < \infty$}
 and
 \[ \overline{\phi}(u) = \phi(u) +a,\] for some $a\geq 0$ and  $\phi$ of the form  \eqref{eq:def_Bernstein}, that satisfy the following conditions %, which correspond for the first two of them to the standing conditions \eqref{eq:standing_assumption} and \eqref{Analicity_of_the_mechanismsp}  on $(\psi,\phi)$,
\begin{equation}\label{Analicity_of_the_mechanis}
	 \int^\infty \frac{du }{|\overline{\psi}(u)|}<\infty  \mbox{ and either }  \theta>0  \mbox{ or } \overline{\psi} ,\overline{\phi} \in \mathcal{H}(R) \textrm{ for some } R>0 \textrm{ and } \overline{\psi}^{(1)}(0) > 0,
\end{equation}
where %By convexity of $\overline{\psi}$,
 $\theta$ is the largest   root of the equation $\overline{\psi}(u)=0$, i.e.~\begin{equation}\label{eq:defintion_theta}
\theta =\sup\{u \geq 0; \: \overline{\psi}(u) =0\} \in [0,\infty).
\end{equation}
Note that since  $\int^\infty \frac{du }{|\overline{\psi}(u)|}<\infty$, we must have $\lim_{u\to\infty}\overline\psi(u)=\infty$ (see Lemma \ref{lem:mainco} and its proof below) and thus there exists at least one root of $\overline\psi$ as $\overline\psi(0)\leq 0$. %Also, recalling that a ${\rm CB(\overline{\psi})}$ semigroup is called: critical when $\overline{\psi}^{(1)}(0) = \overline{\psi}(0)=0$;  supercritical  if $\overline{\psi}^{(1)}(0) <0$ or $\overline{\psi}(0)<0$; and subcritical if $\overline{\psi}^{(1)}(0) >0$ and $\overline{\psi}(0)=0$, and noticing that $\overline{\psi}$ and $ \overline{\phi}$ are analytic on the positive half-plane ${\rm Re}(z) >0$, we can state that
%	$\theta >0$ if and only if the semigroup is supercritical and in that case $ \overline{\psi}^{(1)}(\theta) >0$.
%Thus $\overline{\psi},\overline{\phi} \in \mathcal{H}_{\theta}$  and $\overline{\psi}^{(1)}(\theta) > 0$  amounts to exclude the case where the ${\rm CBI}(\overline{\psi})$ semigroup is critical and to assume that $\overline{\psi}$ and $\overline{\phi}$ can  be analytically extended in a neigborhood of $0$ when it is subcritical.
Next, denote, for any $\eta \geq0$, $\Esc_{\eta}$ the $\eta$-Esscher transform, which is defined for a function $f:\Rpo\mapsto \R$, by $\Esc_{\eta}f(u)=f(u+\eta) - f(\eta)$. It is well-known and easy to prove that, with $\theta$ as in \eqref{eq:defintion_theta}, $\Esc_{\theta} \: \overline{\mathcal{N}} \subseteq  \Ne$  and $\Esc_{\theta} \:  \Bb \subseteq  \Be$, see e.g.~\cite[Example 33.14]{Sato-99}. Then,  we define the following transform
\begin{eqnarray}\label{def:phi0_and_psi0}
\Esc &:&  \Nb \times \Bb \rightarrow {\Ne} \times \mathcal B \nonumber  \\
& & (\overline{\psi}, \overline{\phi})  \mapsto (\psi, \phi) =\Esc(\overline{\psi}, \overline{\phi}) = (\Esc_{\theta}\overline{\psi}, \Esc_{\theta}\overline{\phi}). \end{eqnarray}
An interesting motivation underlying the introduction of $\Esc_\theta$ is the  two time-space Doob's transforms that  leave invariant the set of CBI-semigroups that are described in Proposition \ref{prop:transf} below.  The first transform seems to be original whereas the second one was proved by Roelly and Rouault in \cite{Rouault-Roelly}.   These transforms serve to simplify the notation and are useful to derive the smoothness properties of general CBI-semigroups in $\Nb \times \Bb$ from the one of CBI-semigroups in ${\Ne} \times \mathcal B$. They are proved in subsection \ref{sec:prooftrans}.
\begin{prop} \label{prop:transf}
	\begin{enumerate}
		\item \label{it:transf_theta}
Let $\overline{P}$ be a \CBIb semigroup where $(\overline{\psi}, \overline{\phi})	\in \Nb \times \Bb  $ and let $P$ be the  \CBI semigroup where $(\psi, \phi) =\Esc(\overline{\psi}, \overline{\phi})$. Writing $f_{\theta}(x)=e^{\theta x} f(x)$,  we have, for all $f\in  B_{b}(\R^+)$ and $t, x\geq 0$,
		\begin{equation}\label{eq:rel_sem}
		\overline{P}_t f (x) = e^{-\theta x}e^{-\overline{\phi}(\theta)t}P_t f_\theta(x).
		\end{equation}
		Consequently, by replacing $f$ by $f_{\theta}$ in the statements \eqref{it:sci}, \eqref{it:CinftyC0}, \eqref{it:sf} and \eqref{it:sh}, % with the same $\Si$,
Theorem \ref{thm:regularity} also holds for $\overline{P}$.
		\item Let $\overline{P}$ be a ${\rm CB}(\overline{\psi})$ semigroup.
		Then, there exists  a \CBI semigroup $P$ where $(\psi, \phi)=(\Esc_\theta\overline{\psi}, (\Esc_\theta\overline{\psi})^{(1)} -m )$ with $m=\psi^{(1)}(0)$ such that, for any  $t,x>0$ and $f \in B_b(\R^+)$
		\begin{equation} \label{eq:cb-cbi}
		\overline{P}_tf(x) = x e^{-\theta x}e^{-m  t}P_t \bar{f}_{\theta}(x).
		\end{equation}
		where $\bar{f}_{\theta}(x) = \frac{f_{\theta}(x)}{x}$.
	\end{enumerate}
\end{prop}
%\begin{remark}  The result in Proposition \ref{prop:transf} \eqref{it:transf_theta}   does not require assumptions  \eqref{Analicity_of_the_mechanisms},   \eqref{Analicity_of_the_mechanismsp} or   \eqref{Analicity_of_the_mechanis} and  allows $m \in \R$. Moreover the transform can be generalized in the following sense.
%If $\overline{P}$ is a ${\rm CBI}(\overline{\psi},\overline{\phi})$ semigroup.   Let $\eta,q\in\mathbb R$ and assume
%that $\psi \in \mathcal{H}_{(\eta,\infty)}$ with  $-\infty <\overline{\psi}(\eta)\leq0$ and $\overline{\phi} \in %\mathcal{H}_{(\eta,\infty)}$ with $0 \leq \overline{\phi}(\eta)+q < \infty$. Then
%$(\overline{\psi}(\cdot +\eta),\overline{\phi}(\cdot +\eta) +q) \in \Nb \times \Bb$ and denoting by $P^{\eta,q}$ its  associated  semigroup, we have, for any  $t,x\geq0$ and $f \in B_b(\R^+)$,
%\begin{equation} \label{eq:dens_t}
%$\overline{P}_tf(x) = e^{-\eta x}e^{qt}P^{\eta,q}_tf_{\eta}(x)$.
%\end{equation}	
%\end{remark}
The proof of Theorem \ref{thm:regularity} relies on a combination of  an in-depth analysis of the smoothness properties of the invariant measure  and a substantial refinement of the spectral decomposition of the transition kernels of CBI-semigroups which was originally studied by Ogura \cite{Ogura-70} and that we now state.  To this end, we need to introduce further notation. First, let
	$({\mathcal L}_n)_{n \geq 0}$ be the family of Sheffer polynomials whose generating function is $\Gx(z)$ given by \eqref{eq:def_Gx},  i.e.~for any $x\geq 0$,
	\begin{equation}\label{eq:def-Ln_statement}
	\Gx(z)=\sum_{n=0}^{\infty} {\mathcal L}_n(x) z^n, %:=
	\quad \vert z \vert < R_0.
	\end{equation}
%We denote by $W$ the so-called scale function associated to the spectrally negative L\'evy process whose law is determined by its Laplace exponent $\psi(\cdot +\theta)$. More precisely, the function $W$ is characterized by its Laplace transform as follows, for any $u > 0$,
%\[ \psi(u+\theta)\int_0^\infty e^{-u y} W(y)dy = 1.\]
%Note that under assumption \eqref{eq:standing_assumption}, the underlying L\'evy  process with Laplace exponent $\psi$ has path of unbounded variation and hence, its scale function $\bar W$ satisfies $\bar W(0)=0$ and belongs to $C^1((0,\infty))$, see e.g.~\cite{Chan-Kyp-Savov}. Since $W = e_\theta\bar  W$, it enjoys the same properties.
We let $\nu$, respectively $\omega$, be a  non-negative  integrable function   on $\mathbb R^+$ whose  Laplace transform takes the form
 	\begin{equation}\label{def_nu}
 	\int_0^\infty e^{-\lambda y}\nu(y) dy= e^{-\piv(\lambda)} - e^{-\biv}, \quad \lambda \geq 0,
 	\end{equation}
 	respectively,
 \begin{equation}\label{def_omega}
 \int_0^\infty e^{-\lambda y} \omega(y) d y = 1-A(\lambda), \quad \lambda\geq 0,
 \end{equation}	
where we recall  that $A$, $\piv$ and $\biv$ are defined in \eqref{def_A}, \eqref{eq:def_Phi} and \eqref{absolute_continuity_condition0}. It will be shown in Proposition \ref{prop:invariant_measure} and Corollary \ref{corol:omega} below that the functions $\nu$ and $\omega$ are well-defined.
We further set for $n\geq 1$,
\begin{equation}\label{eq:co-eigen_series}
\mathcal W_{n}(y) = \sum_{j=1}^n \binom n j (-1)^j \omega^{*j}(y),
\end{equation}
where  $\omega^{*1} = \omega$, and, for any $n \geq 2$,
\begin{equation*} %\label{def_omega_star}
\omega^{*n}(y)= \omega^{*n-1} * \omega(y),
\end{equation*}
where $*$ stands for the standard convolution, i.e.~$f * g (y) = \int_0^y f(y-x) g(x)dx$.
%
%\bigskip
%Denoting, for all $ y > 0$,
%\begin{equation}
%\omega(y) = \frac{W(y)}{y},
%\end{equation}
%we write   $\omega^{*1} = \omega$, and, for any $n \geq 2$,
%\begin{equation}\label{def_omega_star}
%\omega^{*n}(y)= \omega^{*n-1} * \omega(y),
%\end{equation}
%where $*$ stands for the standard convolution, i.e.~$f * g (y) = \int_0^y f(y-x) g(x)dx$. We  also write, for some reals $y,q$ and some function $f$,
%\begin{equation}\label{eq:co-eigen_series}
% 	\mathcal{W}_q(y) = \sum_{k=1}^\infty   \omega^{*k}(y) \frac{(-q )^k}{k!} \quad \mbox{ and } \quad 	\mathcal{W}_{q}\bar{\star} f (y)= f(y) + \sum_{k=1}^\infty   \omega^{*k}* f (y) \frac{(-q )^k}{k!},
%\end{equation}
% 	where it is part of the next statement to insure that these objects are well defined. In particular we use $\mathcal W_q$ with $q=\lambda_n$, where
Also, we set
 	\begin{equation*}
 	\lambda_n = m n,
 	\end{equation*}
 	where we recall that $m=\psi^{(1)}(0)$.
% 	Finally, throughout we use the notation, for some integer $\mathfrak{p}$,
% \[ f^{(\mathfrak{p})}(x)=\frac{d^\mathfrak{p}}{dx^\mathfrak{p}} f(x) \]
 %that is the $\mathfrak{p}$-th derivative of $f$.
\begin{thm}  \label{thm:spectral-expansion}  Let $ P$ be a \CBI semigroup with   $(\psi,  \phi) \in \Ne \times \Be$.  Then, for any $t>T_0$, $f\in  \mathcal{D}_t \cup \Lambda$, and,
for all integers $\mathfrak{m}, \mathfrak{p} \geq 0$, we have
\begin{equation}\label{eq:derivatives_of_the_semigroup}
\frac{d^{\mathfrak{m}} }{d t^{\mathfrak{m}}} \left( P_tf \right)^{(\mathfrak{p})} (x)
	=  \sum_{n=\mathfrak{p}}^{\infty}(-\eign )^{\mathfrak{m}}e^{-\eign  t} {\mathcal L}_n^{(\mathfrak{p})}(x) \mathcal{V}_n f, \quad x\geq 0,
\end{equation}
where
the series is locally uniformly convergent in $x,t$ and,  for any $n\geq 0$, $
 	\mathcal{V}_nf  = \int_0^\infty f(y)\mathcal{V}_n(dy)$ with
 	\begin{equation} \label{eq:def_Vn}
 	 \mathcal V_n(dy) = e^{-\biv}\delta_0(dy) + \nu_n(y)dy,
 	\end{equation}
 	where for $n\geq 1$,
 	\begin{equation} \label{eq:definition_nu_n}
 		\nu_n(y) =  e^{-\biv} \mathcal{W}_{n}(y) + \mathcal{W}_{n} {\star} \nu(y) + \nu(y), \quad y> 0,
 	\end{equation}
 	and $\nu_0=\nu$.
% 	 and,  	$\nu$ is a  non-negative  integrable function   on $\mathbb R^+$ whose  Laplace transform takes the form
% 	\begin{equation*}
% 	\int_0^\infty e^{-\lambda y}\nu(y) dy= e^{-\piv(\lambda)} - e^{-\biv}, \quad \lambda > 0.
% 	\end{equation*}
 	In particular, for all $t>T_0$, $x,y\geq 0$, $P_t(x,dy)=e^{-\biv }\Gx(e^{-m t})\delta_0(dy)
 		+  p_t(x,y)dy$ with
 		\begin{equation} \label{eq:transition_kernel}
 		p_t(x,y)=
 	 		\sum_{n=0}^{\infty}e^{-\eign  t}{\mathcal L}_n(x)\nu_n(y).
 	\end{equation}
 	Finally, for all integers $ \mathfrak{m}, \mathfrak{p}\geq 0$ and $0\leq \mathfrak{q} \leq \Si+\bar{\mathfrak{q}}$, with $\bar{\mathfrak{q}}$  as in Theorem \ref{thm:regularity}\eqref{it:defq}, we have
 	\begin{equation}\label{eq:derivatives_of_the_density_in_y}
\frac{d^{\mathfrak{m}} }{d t^{\mathfrak{m}}} p_t^{(\mathfrak{p},\mathfrak{q})}(x,y)
 	=  \sum_{n=\mathfrak{p}}^{\infty}(-\eign )^{\mathfrak{m}}e^{-\eign  t}  {\mathcal L}_n^{(\mathfrak{p})}(x) \nu_n^{(\mathfrak{q})}(y)
 	\end{equation}
 	where, when $0\leq \mathfrak{q} \leq \Si-1$, the ${\mathfrak{q}}$-th derivative of $\nu_n$ is given by
 	$
 	\nu_n^{({\mathfrak{q}})}(y) = \mathcal{W}_{ n}  \star \nu^{({\mathfrak{q}})}(y).
 	$
 	 \end{thm}
\begin{remark}
Note that from the Doob's transform \eqref{eq:rel_sem} in Proposition \ref{prop:transf} we get the following identity between the corresponding heat kernels
\begin{equation} \label{eq:dens_t}
 	\bar{P}_t(x,dy) = e^{-\theta(x-y)}e^{-\bar{\phi}(\theta)t}P_t(x,dy), \quad t,x,y\geq0.
 	\end{equation}
\end{remark}
\begin{remark} \label{rem:main}
 We point out that the phenomena that the linear functional space, here $\mathcal{D}_t$, for which the spectral representation is valid increases with respect to time, has been observed in recent works dealing with the spectral representation of non-self-adjoint (NSA) Markov semigroups, see e.g.~\cite{PS14}, \cite{Patie-Zhao} and \cite{Patie-Savov-GL}. This may be explained by the fact that in opposition to the self-adjoint case where, by the spectral theorem,  a resolution of the identity is available, the invariant subspaces of NSA operators do not form in general a basis of the Hilbert space yielding to convergent spectral expansion only on a subspace of the full Hilbert space.
 \end{remark}

\begin{remark}
In Proposition \ref{prop:eigen} below, we state that the set $(e^{-\eign t})_{n\geq0}$ is part of the point spectrum of the (unique continuous extension of the) CBI operator $P_t$ in the weighted Hilbert space
\begin{equation} \label{eq:l2nu} L^{2}(\mathcal{V})= \left\{f:\Rpo \to\mathbb R \mbox{ measurable}; \:  \int_0^\infty  f^2(y)  \mathcal{V}(dy) < \infty
\right\},\end{equation} where $\mathcal{V}=\mathcal{V}_0$ and the latter is defined in \eqref{eq:def_Vn}. The characterization  of the different components of the spectrum of $P_t$, that is  the point, continuous and residual spectrum, see e.g.~\cite{Dunford_II} for definition, seems to be  a delicate issue and goes beyond the scope of this work. We refer the interested readers to the recent paper by Patie and Savov \cite{PS14}  where an  approach based on the theory of Hilbert sequences has been developed to describe these different parts of the spectrum, including the algebraic and geometric multiplicities of the eigenvalues.% that is the discrete part of the point spectrum for linear operators that are related by an intertwining relationship.}
\end{remark}
\begin{remark}
  It is interesting to observe that  the condition $\psi, \phi \in \mathcal{H}(R)$ for some $R>0$, when $\theta =0$, ensures that the \CBI  semigroup  contains a countable set of isolated eigenvalues, that is, its (discrete) point spectrum is countable. Indeed,  under this condition, the expansion of the holomorphic mapping $\Gx$ enables us to define the eigenfunctions. We point out that Ogura \cite{Ogura-70} shows that when this condition is not satisfied and assuming some (restrictive) additional technical conditions on the mechanisms then the transition kernel of the corresponding CBI-semigroup admits an integral  representation. It would be interesting to relax Ogura's conditions in this situation and to study if the eigenvalues are part of the (continuous) point  or the continuous spectrum. Regarding the second assumption $\int^\infty \frac{du }{\psi(u)}<\infty$ in \eqref{eq:standing_assumption}, it ensures both the existence of eigenmeasures and the absolute convergence of the eigenvalues expansions.
  % and we refer to  \cite{Loeffen-Patie}  for an illustration of these facts.
   Finally, we remark that the analycity property of the mechanisms is, according to Lemma \ref{Analicity_of_A_B_F} below, equivalent to the existence of positive exponential moments of the associated L\'evy measures, that is about the behavior of the L\'evy measure at $\infty$ whereas, from   Lemma \ref{lem:mainco} below, the second condition $\int^\infty \frac{du }{\psi(u)}<\infty$ in \eqref{eq:standing_assumption}, when $\sigma^2=0$, is a condition on  their behaviors at $0$.
\end{remark}

\begin{remark}
The main improvement of our spectral representation %in Theorem \ref{thm:spectral-expansion} (and also in Proposition \ref{prop:Analicity_of_G})
in comparison to Ogura's one in \cite{Ogura-70} is our original characterization of both the eigenfunctions $\mathcal{L}_n$, see Section \ref{sec:polyn}, and of the eigenmeasures $\mathcal V_n$, which  allows us to study, in particular,   regularity properties of the CBI semigroup and its transition kernel.  Besides providing the Lebesgue decomposition of $\mathcal V_n$, \eqref{eq:def_Vn} and \eqref{eq:definition_nu_n} also lead to the following bound on the Laplace transform of $|\mathcal V_n|$, the total variation measure of the signed measure $\mathcal V_n$,
\begin{equation*}
\int_0^\infty e^{-\lambda y}|\mathcal V_n|(d y) \leq (2-A(\lambda))^n e^{-\Phi_\nu(\lambda)}, \quad \lambda>-(R_A\wedge R_\phi),
\end{equation*}
see Proposition \ref{prop:nu_n}  below.
This bound improves (since $A(\lambda)\in[0,1)$ for $\lambda\geq 0$) the one from Ogura, see (2.2) in  \cite{Ogura-70}, which reads
\begin{equation*}
\int_0^\infty e^{-\lambda y}|\mathcal V_n|(d y) \leq A(\lambda)^{-n}e^{-\Phi_\nu(\lambda)}, \quad \lambda>0.
\end{equation*}
This improved bound on $|\mathcal V_n|$ allows us in particular to show that \eqref{eq:derivatives_of_the_semigroup} holds for a wider class of functions $f$ than can be concluded from the results in \cite{Ogura-70}. In this vein, it is worth mentioning that when in $\eqref{eq:defBtbar_and_H}$ $\bar\lambda_t =  R_A \wedge R_\phi$  then, according to Lemma \ref{Analicity_of_A_B_F}, the functional spaces $\mathcal{D}_t$ and $ L^{2}(\mathcal{V})$  are comparable at least at infinity.

\end{remark}
%\section{Proof of Theorem \ref{thm:regularity}}

\begin{example}\label{example_specrep}
Though our main focus is on  studying smoothness properties, we now look at a short example to illustrate the ingredients of the spectral representation in Theorem \ref{thm:spectral-expansion}. Note that in \cite{CLP-F} we study how the eigenmeasures and eigenfunctions can be numerically computed. Examples that deal with the regularity of CBI-semigroups will follow in Section \ref{sec:ex}. We look at the CBI$( {\psi}, {\phi})$-semigroup with  mechanisms given, for any $u\geq0$, by
\begin{equation}\label{mech_selfsim}
\begin{split}
 {\psi}(u)   = & (u+1)^{1+\alpha} -(u+1), \\
 {\phi}(u)   = & (u+1)^{\alpha} - 1,
\end{split}
\end{equation}
with  $0<\alpha \leq 1$.  We point out that when $\alpha=1$,  the  CBI$( {\psi}, {\phi})$ boils down to a linear diffusion which is  the CIR process.  We see that $m=\psi^{(1)}(0)=\alpha$, $A(\lambda) = 1 - \left( \frac{1}{\lambda+1} \right)^{\alpha}$, which by \eqref{def_omega}, leads to $\omega(y) = \frac{1}{\Gamma(\alpha)} y^{\alpha-1} e^{- y}$, where $\Gamma$ is the gamma function. Therefore recalling the well-known fact that the convolution of two gamma distributions with the same scale parameter is again a gamma distribution with the same scale parameter but with shape parameter equal to the sum of the individual shape parameters, we have $\mathcal W_n(y)=e^{-y}\sum_{j=1}^n \binom n j (-1)^j \frac{y^{\alpha j-1}}{\Gamma(\alpha j)} $.
Further,  $\Phi_\nu(\lambda)=\ln(\lambda+1)$ and thus $\nu(y)=e^{-y}$. Then, from the expression \eqref{eq:definition_nu_n}, we get that for $n\geq 0$,
 \begin{equation*}
 \nu_n(y)=    e^{-y}\sum_{j=0}^n \binom n j (-1)^j \frac{y^{\alpha j}}{\Gamma(\alpha j +1)},\: y>0.
 \end{equation*}
Note that with $\alpha=1$, $\nu_n(y)=    e^{-y} \mathrm{L}_n(x)$ where the $\mathrm{L}_n$'s  are the classical Laguerre polynomials, see e.g.~\cite[4.17.2 ]{Lebedev-72}.

  We have $B$, the inverse of $A$, equals $B(z) =  (1-z)^{-1/ \alpha} - 1.$ Thus $B \in \mathcal{H}(1)$ %and we  have $R_0$ in \eqref{eq:def_Gx} equals $1$ and so
  and hence $T_0=0$ in \eqref{def:t0}, and, for all $n\geq 0$, $x\geq 0$, from \eqref{eq:def-Ln_statement}, we have
$ n! {\mathcal L}_n(x) = G^{(n)}_x(0)$ with
\[
G_x(z) = (B(z)+1)e^{-xB(z)}.
\]
Note that when $\alpha=1$, $G_x(z) =  (1-z)^{-1}e^{\frac{z}{1-z}x}$ which is the generating function of the classical Laguerre polynomials, that is, in this case, for all $n\geq 0$, ${\mathcal L}_n(x) =  \mathrm{L}_n(x)$.  Otherwise,
since for $j\geq 0$,
$
{(B(z) + 1)^{(j)}}_{\vert_{z=0}} =\frac{\Gamma(\frac{1}{\alpha} +j)}{\Gamma(\frac{1}{\alpha})}
$, we have, by Leibniz's formula, for  any $n\geq 0$,
\begin{equation*}
 {\mathcal L}_n(x) = \frac{1}{n!}\sum_{j=0}^n \binom n j \frac{\Gamma(\frac{1}{\alpha} +j)}{\Gamma(\frac{1}{\alpha} )} \left( e^{-x B(z)} \right)^{(n-j)}_{\vert_{z=0}},
\end{equation*}
with, by means of  the Faa Di Bruno's formula,
\[{\left(e^{-xB(z)}\right)^{(n-j)}}_{\vert_{z=0}} = \sum_{k=1}^{n-j} (-x)^k\mathcal{B}_{n-j,k}\left(\frac{\Gamma(\frac{1}{\alpha} +1)}{\Gamma(\frac{1}{\alpha})}, \frac{\Gamma(\frac{1}{\alpha} +2)}{\Gamma(\frac{1}{\alpha})}, \cdots, \frac{\Gamma(\frac{1}{\alpha} +n-j-k+1)}{\Gamma(\frac{1}{\alpha})}\right),\]
where the $\mathcal{B}_{n,k}$'s are the Bell polynomials.

%As $B \in \mathcal{H}(1)$, we have $R_0$ in \eqref{eq:def_Gx} equals $1$ and so $T_0$ in \eqref{def:t0} equals $0$.
Finally, by Theorem \ref{thm:spectral-expansion}, the CBI$(\psi,\phi)$-semigroup transition kernel is given by
\begin{equation*}
P_t(x,d y) =  \sum_{n=0}^{\infty} e^{-\alpha n t} \mathcal L_n(x) \nu_n(y) d y, \quad y>0,t>0,x\geq 0.
\end{equation*}
\end{example}

The remaining part of the paper is mainly devoted to  the proofs of the main statements,  namely  Theorem \ref{thm:regularity}, Proposition \ref{lem:kap}, Proposition \ref{prop:transf} and Theorem \ref{thm:spectral-expansion}. They are split into several parts where each of them may be of independent interest. Indeed, in the next Section, we review some useful preliminary results including different criteria for the main conditions and general results regarding smoothness property of solutions of convolution equations. We also provide therein the proof of Proposition \ref{lem:diff_kappa} and Proposition \ref{prop:transf}. Section \ref{sec:iv} contains a thorough study of smoothness properties of the absolutely continuous part of the invariant measure. We proceed by studying in detail the spectral components of the CBI-semigroups, and, in particular, we establish  some specific representations of each of them  which allow us to derive some analytical properties. Note, as  CBI-semigroups are in general non-self-adjoint operators, that the spectral components  include a set of eigenfunctions and  of eigenmeasures which when the latter are absolutely continuous may correspond to the sequence of eigenfunctions for the adjoint semigroups. More specifically, in Section \ref{sec:polyn}, we present an original study of the sequence of eigenfunctions by relating them to the Sheffer polynomials.  In Section \ref{sec:meas}   we provide explicit representations for the eigenmeasures which enable us to obtain both smoothness properties and uniform upper bounds for their absolutely continuous parts as well as for the successive derivatives of these latter, whenever they exist.  Section \ref{sec:proof_mr} includes  the last arguments required to prove Theorem \ref{thm:regularity} and Theorem \ref{thm:spectral-expansion}. Finally, the last Section contains several instances of CBI-semigroups which illustrate the variety of smoothness properties that this class may enjoy and also reveal that our results are sharp.

% whereas in Section \ref{sec:meas}   we provide a series representation for the absolute continuous part of the sequence of eigenmeasures which enables us to obtain both smoothness properties and uniform upper bounds for this sequence as well as for its successive derivatives, whenever they exist.  Section \ref{sec:proof_mr} includes  the last arguments required to prove Theorem \ref{thm:regularity} and Theorem \ref{thm:spectral-expansion}. Finally, the last Section contains several instances of CBI-semigroups which illustrate the variety of smoothness properties that this class may enjoy and also reveal that our results are sharp.

\section{Proof of Propositions \ref{lem:diff_kappa} and \ref{prop:transf}  and preliminary results} \label{sec:preli}
\subsection{Proof of Proposition \ref{prop:transf}}\label{sec:prooftrans}%: Time-space Doob's transform between CBI-semigroups}
 %The following result relates  the transition kernels of specific pairs of CBI-semigroups. Note that this result does not assume the conditions  \eqref{Analicity_of_the_mechanis} and allows $m \in \R$ in \eqref{eq:def_psi}.
% Let us introduce the group of transformations $(\T_\eta)_{\eta\in \R}$ acting on the set of real functions $f$ and defined by  $\T_\eta f(.)=f(.+\eta)$.

 %\begin{remark}
%A well-known necessary and sufficient condition for $\psi \in \mathcal{H}_{(\eta,\infty)}$ (resp.~$\phi \in \mathcal{H}_{(\eta,\infty)}$)  is  $\int_1^\infty e^{-\eta r}\Pi(dr)<\infty$ (resp.~ $\int_1^\infty e^{-\eta r}\mu(dr)<\infty$) where $\Pi(dr)$ (resp.~$\mu(dr)$) is the L\'evy measure associated respectively to $\psi$ (resp.~$\phi$) which both trivially hold whenever $\eta>0$.
 %\end{remark}

 An application of \cite[Proposition 1.2]{Ogura-70} shows that, for all $\lambda > \theta$,
 \[ \overline{P}_te_\lambda (x) = e^{-\overline{\phi}(\theta)t }
 \exp\left(\int_{\overline{B} \left(\overline{A}(\lambda) e^{-\overline{\psi}^{(1)}(\theta)t}\right)}^{\lambda}\frac{\overline{\phi} (u) -\overline{\phi}(\theta)}{\overline{\psi}(u)}du
  - x \overline{B} \left(\overline{A}(\lambda) e^{-\overline{\psi}^{(1)}(\theta)t}\right)\right), \]
 where $\overline{B}$ is the inverse of  $\overline{A}(\lambda) = \exp\left(-\overline{\psi}^{(1)}(\theta) \int_\lambda^\infty\frac{du}{\overline{\psi}(u)}\right)$.
 Now, recalling that, for all $u>\theta$, $\overline{\phi}(u) -\overline{\phi}(\theta) = \phi(u-\theta)$ and $\overline{\psi}(u) = \psi(u-\theta)$,  it is plain that $\overline{\psi}^{(1)}(\theta) =\psi^{(1)}(0)=m$ and  $\overline{A}(\lambda) = A(\lambda -\theta)$ and $\overline{B}(z) = B(z)+\theta$.
It then easy to check, by a change of variables and \eqref{eq:lt_cbi}, that
\begin{eqnarray*}
\overline{P}_te_\lambda (x) & = & e^{-\overline{\phi}(\theta)t }
\exp\left(\int_{B \left(A(\lambda-\theta) e^{-mt}\right)+\theta}^{\lambda}\frac{ \phi (u-\theta) }{\psi(u-\theta)}du
- x \left(B \left(A(\lambda-\theta) e^{-mt}\right)+\theta\right)\right) \\
 & = & e^{-\theta x} e^{-\overline{\phi}(\theta)t } G_x(A(\lambda -\theta)e^{-mt}) \\
 & = & e^{-\theta x} e^{-\overline{\phi}(\theta)t } P_t e_{\lambda-\theta}.
\end{eqnarray*}
This proves the first part of Proposition \ref{prop:transf}.
The last claim follows by first applying  the transformation \eqref{eq:rel_sem} and then the transformation \cite[(3.37)]{Li_Book} recalling that $m <\infty$ as it is imposed in this latter reference.

\subsection{Criteria for the main conditions}
The conditions  as well as the criteria used in the description of the main results are given in terms of the CBI mechanisms or of the function $\kappa$ defined in \eqref{eq:def_k}. In this part, we aim at providing equivalent criteria expressed directly in terms of the characteristic triplets of these mechanisms.
We recall the following notation on asymptotic behaviours that will remain in force throughout the  paper.
\begin{eqnarray*}
f &\asymp &  g \textrm{ means that }  \exists \:  c>0  \textrm{ such that }
c \leq  \frac{f}{g} \leq c^{-1}, \\
f &\stackrel{a}{\sim}& g  \textrm{ means that } \lim_{x \to a}\frac{f(x)}{g(x)}=1, \textrm{ for some }  a\in \R\cup\{\pm\infty\}.%\\
%f &\stackrel{a}{=}& {\rm{O}}\left(g\right)\textrm{ means that } \limsup_{x \to a } \left| \frac{f(x)}{g(x)}\right| < \infty.
\end{eqnarray*}
We start  with the following result dealing with the first condition in \eqref{eq:standing_assumption}.
\begin{lemma} \label{lem:mainco} Let $\psi(u) =\sigma^2 u^2 + m u+ \int_0^\infty\left(e^{-u r}-1+ur \right)\Pi(dr)$ with $\sigma\geq 0$, $m>0$ and $\int_0^\infty(r\wedge r^2)\Pi(d r)<\infty$. Then
%\replaced[id={RL},remark={Since $\psi\in \Ne$ implies by definition  $\int^{\infty}\frac{du}{\psi(u)}<\infty$, the old statement is a bit strange.}]{Let $\psi(u) =\sigma^2 u^2 + m u+ \int_0^\infty\left(e^{-u r}-1+ur \right)\Pi(dr)$ with $\sigma\geq 0$, $m>0$ and $\int_0^\infty(r\wedge r^2)\Pi(d r)<\infty$. Then}{For any $\psi \in \Ne$, we have}
 \begin{equation} \label{eq:int_test}
  \int^{\infty}\frac{du}{\psi(u)}<\infty \Longleftrightarrow \sigma>0  \textrm{ or } \int_0 \frac{dv}{\int_0^{v}\overline{\overline{\Pi}}(r)dr}<\infty.
  \end{equation}
If for some $\epsilon>0$, $\liminf_{y\downarrow 0}\PPP(y)y^{\epsilon}>0$ then \eqref{eq:int_test} holds. However, if $\PPP(0^+)<\infty$ and $\sigma=0$, then  \eqref{eq:int_test} fails. When $\sigma^2=0$, \eqref{eq:int_test} is not equivalent to  $\PPP(0^+)=\infty$ as $\psi(u)=u(\ln(u+1)+1) \in \Ne$ with $\PPP(0^+)=\infty$ and $\ \int^{\infty}\frac{du}{u\ln(u+1)} = \infty$. %Note that this latter integral test holds if the tail of the L\'evy measure
 % $\Pi(y,\infty)$ is regularly varying at zero of index $-2<\alpha<-1$. However, it fails if $ \overline{\Pi}(0^+)<\infty$ that is  when the underlying L\'evy process has paths of bounded variations,
\end{lemma}
\begin{remark}
Note that the conditions in \eqref{eq:int_test} ensure  that the class of CBI-processes have paths  of unbounded variation, as either $\sigma^2>0$ or $\PPP(0^+)=\infty$. %but it seems that the one whose L\'evy measure has a slowly varying tail at $0$ are excluded.
\end{remark}
\begin{proof}
 First, recall from \eqref{eq:whl} that
\begin{equation}\label{eq:wh}
 \psi(u)=u \phi_p(u)=u\left(\sigma^2 u+ u\int_0^{\infty}e^{-ur}(\PP(r)+m)dr \right)\end{equation} and so $\phi_p \in \Be$. Thus, \cite[Proposition III.1]{Bertoin-96} yields
 \begin{equation} \label{eq:est_ber}
 \psi(u) = u\phi_p(u) \asymp \sigma^2 u^2 + mu + u^2\int_0^{\frac{1}{u}}\overline{\overline{\Pi}}(r)dr.
  \end{equation}
 From this estimate, we easily get the first statement where for the integral test we have performed a change of variables.  Next, the condition $\liminf_{r\downarrow 0}\PPP(r)r^{\epsilon}>0$ for some $\epsilon>0$ implies that there exists $C>0$ such that for $r$ small enough, $\PPP(r)\geq Cyr^{-\epsilon}$ and as $\PPP$ is non-increasing, we get that    $\int_0^{v}\overline{\overline{\Pi}}(r)dr\geq v \PPP(v) \geq C v^{1-\epsilon}$, that is, from the preceding discussion,
 \begin{eqnarray}
  \int^{\infty}\frac{du}{\psi(u)} &\leq &  C^{-1}\int_0 v^{\epsilon-1}dv <\infty. \end{eqnarray}
  Next, since from \eqref{eq:wh} when $\sigma^2=0$, we get that  $\phi_p(\infty)=m+\PPP(0^+)$ and hence, when $\PPP(0^+)<\infty$, $\psi(u)\leq u\phi_p(\infty), u\geq 0,$ as $\phi_p$ is non-decreasing, which provides the statement in this case. Finally, observing that $\ln(u+1) = \int_0^{\infty}(1-e^{-ux})\frac{e^{-x}}{x}dx$, we easily deduce, by integration by parts, that  $u\ln(u+1)+u = u\int_0^{\infty}(1-e^{-ux})\frac{e^{-x}}{x}dx +u= \int_0^{\infty}(e^{-ux}-1+ux)e^{-x}\frac{(x+1)}{x^2}dx+u  \in \Ne$ with $\PPP(0^+)=\int_0^\infty \frac{e^{-x}}{x}d x=\infty$ and $ \int^{\infty}\frac{du}{u\ln(u+1)+u}=\infty$, which completes the proof of the Lemma.
\end{proof}

We proceed with these known and basic facts regarding the scale function $W$ defined in \eqref{eq:def_W}.
\begin{lemma} \label{eq:propW}
For any $\psi \in \Ne$, we have $W \in C^1(\R^+) \textrm{ with }  W(0)=0 \textrm{ and } W^{(1)}(y)> 0, \textrm{ for all } y> 0$. Moreover,  for all $u>0$,
 	\begin{eqnarray} \label{eq:pot}
		\frac{1}{\phi_p(u)}=\frac{u}{\psi(u)}&=&\int_0^{\infty}e^{-u y}W^{(1)}(y)dy,
	\end{eqnarray}
that is $W^{(1)}(y)dy$ is the potential measure of the subordinator whose Laplace exponent is $\phi_p$. Consequently,   $\lim_{y\to\infty}W(y)=1/m$ and $W^{(1)}\in L^1(\R^+)$.
\end{lemma}
\begin{proof}
Since, under the assumption $\int^\infty \frac{du }{\psi(u)}<\infty$, Lemma \ref{lem:mainco}  ensures that  the underlying L\'evy process has paths of unbounded variation, $W(0)=0$ and $W \in C^1(\R^+)$ follows from p.254 and Proposition 5.1 in   \cite{Lambert}. %{\color{red}{is the right reference from Lambert?}} and from the well-known fact that $W$ is a strictly increasing function on $\Rpo$.
The identities in \eqref{eq:pot} follow from \eqref{eq:wh} and an integration by parts respectively.  Hence $W^{(1)}(y)dy$ is the potential measure of the subordinator, denoted by $(S_t)_{t\geq 0}$, whose Laplace exponent is $\phi_p$, i.e. $W^{(1)}(y)d y = \int_0^\infty \mathbb P(S_t\in d y) d t$. It follows that $W^{(1)}(y)\geq 0$ for all $y>0$. To show that actually $W^{(1)}(y)>0$ for all $y>0$, suppose instead that $W^{(1)}(y_0)=0$ for some $y_0>0$. Since $W^{(1)}(y)/W(y)$ is a non-increasing function for $y>0$ (see (8.26) in \cite{Kyprianou-14}), it follows that $W^{(1)}(y)=0$ for all $y\geq y_0$. But then $\mathbb P(S_t\geq y_0)=0$ for a.e. $t>0$ which is absurd. For the last two claims, by \eqref{eq:pot} and the monotone convergence theorem,
\begin{equation*}
\frac{1}{m} = \lim_{u\downarrow 0}	\frac{1}{\phi_p(u)} = \lim_{u\downarrow 0}\int_0^{\infty}e^{-u y}W^{(1)}(y)dy = \lim_{y\to\infty}W(y) - W(0) = \lim_{y\to\infty}W(y)
\end{equation*}
and so $W^{(1)}\in L^1(\R^+)$ since it is a positive function.
\end{proof}

\subsection{Derivatives  and smoothness of convolutions}\label{sec_conv}
In this section we present two lemmas on differentiability of convolutions, which will be used later on.
\begin{lemma}\label{lem_conv_diff_onederiv}
Let $f:\R^+ \to\mathbb R$ be absolutely continuous on $\R^+$ and $g\in L^1_{loc}(\R^+)$. Assume that $f'\in L^1_{loc}(\R^+)$, where $f'$ denotes a version of the density of $f$ and further assume that $f(0^+)=\lim_{y\downarrow 0}f(y) \in \R$. Then the convolution
\begin{equation*}
h(y) = \int_0^y f(y-r)g(r)  d r
\end{equation*}
has a density on $\R^+$ and a version of it is given, for any $y>0$, by
\begin{equation*}
h'(y) = \int_0^y f'(y-r)g(r) d r + f(0^+)g(y).
\end{equation*}
Moreover, $h \in C^1(\R^+)$  with derivative given by $h^{(1)}=h'$ if $g \in  C(\R^+)$  and either
 \begin{enumerate}[(i)]
  \item $f \in C^1(\R^+)$ or % $f^{(1)}$ is in $C(\R^+)$ or
  \item  $g\stackrel{0}{=}{\rm{O}}(1)$. % is     bounded  in a neighbourhood of $0$.
   \end{enumerate}
\end{lemma}
\begin{proof} %[\textbf{Proof}]
We have, for any $y>0$,
\begin{eqnarray*}
\int_0^y \int_0^r |f'(r-v)g(v)|  d v  d r &= &
 \int_0^y \int_r^y  |f'(v-r)|  d v |g(r)|  d r
 =   \int_0^y \int_0^{y-r} |f'(v)|  d v |g(r)|  d r \\
 &\leq & \int_0^y |g(r)|  d r \int_0^{y} |f'(r)|  d r  <  \infty,
\end{eqnarray*}
where we used the fact that $f'$ and $g$ are in $L^1_{loc}(\R^+)$. Hence  Fubini Theorem yields
\begin{eqnarray*}
\int_0^y  \int_0^r  f'(r-v)g(v)   d v  + f(0^+)g(r)  d r
&= &  \int_0^y \int_v^y   f'(r-v)   d r  g(v)   d v +  f(0^+)\int_0^y  g(r)   d r  \\
&= & \int_0^y (f(y-r)-f(0^+)) g(r)   d r  +  f(0^+)\int_0^x  g(r)   d r   \\
&= &  h(y).
\end{eqnarray*}
Hence the function $h$ has a density on $\R^+$ which is given by $h'$ as stated in the Lemma. Now assume that  $g$ is in $C(\R^+)$ and  let $y>0$. If  $f\in C^1(\R^+)$, then $f^{(1)}$ and $g$ are bounded on sets of the form $[a,b]\subset\R^+$ and therefore by the dominated convergence theorem,
\begin{eqnarray*}
\lim_{\delta\to 0} h'(y+\delta) & = & \lim_{\delta\to 0}\int_0^{\frac{y+\delta}{2}}  f^{(1)}(y+\delta-r)g(r)   dr   +  \lim_{\delta\to 0} \int_0^{\frac{y+\delta}{2}}  g(y+\delta-r)f^{(1)}(r)   dr + f(0^+)g(y) \\
&= & \int_0^{\frac{y}{2}}  f^{(1)}(y-r)g(r)   dr   +   \int_0^{\frac{y}{2} }  g(y -r)f^{(1)}(r)   dr + f(0^+)g(y) \\
&= & h'(y).
\end{eqnarray*}
If instead  $g$ is bounded  in a neighbourhood of zero, then
by the dominated convergence theorem,
\begin{eqnarray*}
\lim_{\delta\to 0} h'(y+\delta) & = &   \lim_{\delta\to 0} \int_0^{y+\delta }  g(y+\delta-y)f'(y)   d y + f(0^+)g(y)\\
&=&    \int_0^{y}  g(y -y)f'(y)   d y + f(0^+)g(y) \\
&= & h'(y).
\end{eqnarray*}
 Hence, in both cases, $h'$ is in $C(\R^+)$, which implies by the fundamental theorem of calculus that $h$ is (continuously) differentiable  on $\R^+$ with derivative given by $h^{(1)}=h'$.
\end{proof}

\begin{lemma}\label{lem_conv_diff_split}
Let $f,g\in C^{p-1}(\R^+)\cap L^1_{loc}(\R^+)$ for some $p\geq1$ and assume   that the $p$-th derivatives $f^{(p)}$ and $g^{(p)}$ exist on $\R^+$ and are bounded on compact (with respect to $\R$) subsets of $\R^+$.%$[a,b]\subset (0,\infty)$.
Then $h \in C^p(\R^+)$ where, for any $y>0$, \begin{equation*}
h(y) = \int_0^y f(y-r)g(r)  d r
\end{equation*}
and
\begin{eqnarray*}
h^{(p)}(y) &=& \int_0^{\frac{y}{2}}  f^{(p)}(y-r)g(r) d r + \int_0^{\frac{y}{2}}  g^{(p)}(y-r)f(r) d r \\
&+& \frac 12\sum_{j=0}^{p-1} \left(  f^{(j)}(\tfrac{y}{2}) g(\tfrac{y}{2}) +  f(\tfrac{y}{2}) g^{(j)}(\tfrac{y}{2}) \right)^{(p-1-j)}.
\end{eqnarray*}
\end{lemma}
\begin{proof} %[\textbf{Proof}]
First we prove the claims for $p=1$. Let $y>0$. We can write for $\delta>0$,
\begin{eqnarray}\label{deriv_conv_splitting}
\frac{h(y+\delta)-h(y)}{\delta} &= &  \int_0^{\frac{y}{2} }  \frac{ f(y+\delta-r) -f(y-r)}{\delta} g(r) d r  +  \int_{\frac{y}{2} }^{\frac{y+\delta}{2}}  \frac{ f(y+\delta-r)}{\delta} g(r) d r  \nonumber \\
& + & \int_0^{\frac{y}{2} }  \frac{ g(y+\delta-r) -g(y-r)}{\delta} f(r) d r  +  \int_{\frac{y}{2}}^{\frac{y+\delta}{2}}  \frac{ g(y+\delta-r)}{\delta} f(r) d r.
\end{eqnarray}
By the mean value theorem, we get, for all $r\in[0,\tfrac{y}{2}]$ and  $\delta\in[0,\tfrac{y}{2}]$,
\begin{equation*}
\left| \frac{ f(y+\delta-r) -f(y-r)}{\delta} \right| \leq \sup_{r\in \left[\frac{y}{2}, \frac{3y}{2}\right]} f^{(1)}(r).
\end{equation*}
Thus, by the assumption that $f^{(1)}$ is bounded on sets of the form $[a,b]\subset \R^+$ and the dominated convergence theorem, we obtain that
\begin{equation*}
\lim_{\delta\downarrow 0} \int_0^{\frac{y}{2}}  \frac{ f(y+\delta-r) -f(y-r)}{\delta} g(r) d r
=   \int_0^{\frac{y}{2}} f^{(1)}(y-r)g(r) d r.
\end{equation*}
Due to the continuity of $f$ and $g$, we have by the mean-value theorem again that for each $\delta>0$ sufficiently small   there exists   $r_\delta\in \left[ \frac{y}{2}, \frac{y+\delta}{2} \right]$ such that
\begin{eqnarray*}
\lim_{\delta\downarrow 0} \int_{\frac{y}{2}}^{\frac{y+\delta}{2}}  \frac{ f(y+\delta-r)}{\delta} g(r) d r
=\lim_{\delta\downarrow 0}  f( y+\delta - r_\delta)g(r_\delta) \frac{\frac{y+\delta}{2}-\frac{y}{2}}{\delta}
=\tfrac 12 f(\tfrac{y}{2}) g(\tfrac{y}{2}).
\end{eqnarray*}
Similarly, we can treat the other two terms on the right-hand side of \eqref{deriv_conv_splitting}, which leads to
\begin{equation*}
\lim_{\delta\downarrow 0} \frac{h(y+\delta)-h(y)}{\delta} = \int_0^{\frac{y}{2}}  f^{(1)}(y-r)g(r) d r + \int_0^{\frac{y}{2}}  g^{(1)}(y-r)f(r) d r + f(\tfrac{y}{2}) g(\tfrac{y}{2}).
\end{equation*}
Similarly, one   shows that
\begin{equation*}
\lim_{\delta\uparrow 0} \frac{h(y+\delta)-h(y)}{\delta} = \int_0^{\frac{y}{2}}  f^{(1)}(y-r)g(r) d r + \int_0^{\frac{y}{2}}  g^{(1)}(y-r)f(r) d r + f(\tfrac{y}{2}) g(\tfrac{y}{2}).
\end{equation*}
This proves the claims for $p=1$. The results for any $p\geq 1$ then follows by a straightforward induction argument using the same steps as for the $p=1$ case.
\end{proof}

\subsection{Proof of Proposition \ref{lem:diff_kappa}}
\subsubsection{Proof of Proposition \ref{lem:diff_kappa}\eqref{it:prop121}}
First note that \eqref{eq:pot} and the estimate in \cite[Proposition III.1]{Bertoin-96} yield that, for any $y>0$,
\begin{equation} \label{est:W}
W(y) {\asymp} \frac{1}{\phi_p\left(\frac{1}{y} \right)}{\asymp} \frac{y}{\sigma^2 + my+ \int_0^y  \overline{\overline \Pi}(r) d r}.\end{equation}
 Thus, $\liminf_{y \to 0} \frac{W(y)}{y} \geq c \liminf_{y \to 0} \frac{1}{y\phi_p\left(\frac{1}{y} \right)}$ for some $c>0$.  As when $\sigma^2=0$, we have $\phi_p\left(u\right)\stackrel{\infty}{=}o(u)$, see also \cite[Proposition III.1]{Bertoin-96}, we easily deduce that in this case $W^{(1)}(0^+)=\infty$. Otherwise, combining \eqref{eq:pot} with \cite[Theorem III.2.5]{Bertoin-96}, we get that  $W^{(1)}(0^+)=\sigma^{-2}$. Next, let us  consider the case $\sigma^2 + b >0$. When $\sigma^2=0$ and $b>0$ the statement  follows from the preceding discussion as $\lim_{y \to 0 } \kappa(y) \geq \lim_{y \to 0 }
bW^{(1)}(y)=\infty$. Otherwise if $\sigma^2>0$ we have, as $W$ is increasing,
   \[ 0 \leq \lim_{y \to 0 }\int_0^yW^{(1)}(y-r)\overline{\mu}(r)dr \leq \sup_{r \in [0,y]}W^{(1)}(r) \lim_{y \to 0 } \int_0^y\overline{\mu}(r)dr =0 \]
   as, when	$\sigma>0$, $W^{(1)}(0^+)=\sigma^{-2}$, and, by  \cite[Theorem 1]{Chan-Kyp-Savov}, $W^{(1)}\in C^1(\R^+)$. Hence, for all $b\geq0 $,
   \[\lim_{y \to 0 } \kappa(y)= \lim_{y \to 0 }
bW^{(1)}(y) + \int_0^yW^{(1)}(y-r)\overline{\mu}(r)dr = \frac{b}{\sigma^2}, \]
which completes the proof of the statement for the case $\sigma^2 + b >0$.
Next, since $\overline{\mu}$ is non-increasing, we have, for all $y>0$,  \[\kappa(y)=
bW^{(1)}(y) + \int_0^yW^{(1)}(y-r)\overline{\mu}(r)dr \geq  bW^{(1)}(y) + \overline{\mu}(y) \int_0^yW^{(1)}(r)dr =  bW^{(1)}(y) + \overline{\mu}(y) W(y).\]
Thus, if $\sigma^2+b=0$,
 \[ \underline{\kappa}(0^+) = \liminf_{y \to 0} \kappa(y) \geq  \liminf_{y\downarrow 0}\overline{\mu}(y) W(y).\]
 %Clearly, from \eqref{eq:def_Bernstein} with $b=0$, we have $\overline{\mu}(y) \stackrel{0}{\sim} \phi\left(\frac{1}{y} \right)$ and  as $W(0)=0$, we obtain from \eqref{est:W} that $W(y) \stackrel{0}{\sim} \frac{1}{\phi_p\left(\frac{1}{y} \right)}$ which provides  the (equivalent) lower bounds in this case. That is we have
%  \begin{equation*}
%\underline{\kappa}(0^+) \geq \liminf_{y\to 0}\overline{\mu}(y) W(y)=\liminf_{u\to \infty}\frac{\phi(u)}{\phi_p(u)}=\liminf_{r \rightarrow 0}\frac{ \int_0^{r}\overline{\mu}(y)dy}{\int_0^{r}\overline{\overline{\Pi}}(y)dy}
%\end{equation*}
%  where for the last identity we used again the estimate \cite[Proposition III.1]{Bertoin-96}, which already appeared in \eqref{eq:est_ber}.
  Thus if   $\overline{\Pi}(y) \stackrel{0}{=} {\rm{O}}(  y^{-\alpha})$ and $\frac{1}{\bar{\mu}(y)}\stackrel{0}{=}{\rm{O}}(  y^{\beta})$,  we observe from \eqref{est:W} that for some $C_{\alpha,\beta}>0$,
   \begin{equation}
 \underline{\kappa}(0^+)    \geq  \liminf_{y\downarrow 0}\overline{\mu}(y) W(y)
 \geq \liminf_{y\downarrow 0}  \frac{ C_{\alpha,\beta} \bar{\mu}(y) }{ m + \frac 1y \int_0^y  \overline{\overline \Pi}(r) d r}
= \liminf_{y\downarrow 0} \frac{ C_{\alpha,\beta} y^{-\beta}}{m+\frac 1{\alpha-1} \frac 1{2-\alpha}  y^{1-\alpha} },
    \end{equation}
%  \begin{equation}
%  \frac{ \int_0^{r}\overline{\mu}(y)dy}{\int_0^{r}\overline{\overline{\Pi}}(y)dy} \geq C_{\alpha,\beta} \: r^{\alpha-\beta-1},
%  \end{equation}
  which shows that $\underline{\kappa}(0^+)=\infty$ if  $1<\alpha<1+\beta<2$.
Finally if  $\overline{\mu}(0^+)<0$,  since $\overline{\mu}$ is non-increasing, we have, for all $y>0$,  \[
 \overline{\mu}(y) W(y) \leq \kappa(y) \leq  \overline{\mu}(0^+) W(y), \]
which gives the last claim  as  $W(0)=0$.
\subsubsection{Proof of Proposition \ref{lem:diff_kappa}\eqref{it:prop122}}
Assume first that $\sigma=0$, then necessarily $b=0$ as otherwise $\Si=\infty$. Under the assumption that $W^{(1)},\bar{\mu}\in C^1(\R^+)$, we can use Lemma \ref{lem_conv_diff_split} to deduce that $\kappa\in C^1(\R^+)$, and, for any  $y>0$,
\begin{equation*} %\label{deriv_kappa_2}
\kappa^{(1)}(y) = \int_0^{\frac{y}{2}}  W^{(2)}(y-r)\bar{\mu}(r) d r + \int_0^{\frac{y}{2}}  \bar{\mu}^{(1)}(y-r)W^{(1)}(r)  d r + W^{(1)}(\tfrac{y}{2}) \bar{\mu}(\tfrac{y}{2}).
\end{equation*}
Now we show that $\kappa^{(1)}\in L^1_{loc}(\R^+)$.
First, by integration by parts and the fact that $W$ is a non-decreasing  function,
\begin{eqnarray}\label{int_diffscale_tailm}
 \frac{\int_0^x W^{(1)}(\tfrac{y}{2}) \bar{\mu}(\tfrac{y}{2}) d y}{2} & = & \lim_{\epsilon\downarrow 0}\int_{\epsilon}^{\frac{x}{2}}  W^{(1)}(r) \bar{\mu}(r) d r \nonumber \\
 %= \lim_{\epsilon\downarrow 0} \left( 2 W(r)\bar{\mu}(r)|_{r=\epsilon}^{\frac12 x} - 2 \int_\epsilon^{\frac12 x}  W(u)\bar{\mu}^{(1)}(r) d r \right) \\
 &\leq &  W(\tfrac{x}{2})\bar{\mu}(\tfrac{x}{2}) - \int_0^{\frac{x}{2}}  W(r)\bar{\mu}^{(1)}(r) d r
 <  \infty,
 \end{eqnarray}
where the last line follows by the integral assumption and the fact that, in this case, $W(y)\asymp \frac{y}{m+\int_0^y  \overline{\overline \Pi}(r) d r}$, see  \eqref{est:W}.
Next, fix $x>2\delta$. By the assumptions we have $|W^{(2)}(y)|\leq C$ with $0<C<\infty$ for all $y\in[\delta,x]$  and $|W^{(2)}(y)|=-W^{(2)}(y)$ for all $0<y<\delta$. Therefore by Fubini Theorem,
\begin{eqnarray*}
\int_0^x \left| \int_0^{\frac{y}{2}}  W^{(2)}(y-r)\bar{\mu}(r) d r \right|  d y & \leq &
  \int_0^{\frac{x}{2}} \left( \int_{2r}^x \left| W^{(2)}(y-r) \right| d y \right) \bar{\mu}(r) d r \\
  &= &  \int_0^{\frac{x}{2}} \left( \int_{r}^{x-r} \left| W^{(2)}(v) \right| d v \right) \bar{\mu}(r) d r \\
  &\leq & \int_0^{\frac{x}{2}} \left( \int_{r}^{x-r} \left( C-W^{(2)}(v)\mathbb{I}_{\{v<\delta\}} \right) d v \right) \bar{\mu}(r) d r \\
  &= & C \int_0^{\frac{x}{2}} (x-2r) \bar{\mu}(r) d r - \int_0^{\delta} (W^{(1)}(\delta) - W^{(1)}(r) ) \bar{\mu}(r) d r\\
  &<&  \infty,
\end{eqnarray*}
where in the last line we used \eqref{int_diffscale_tailm}. Similarly, one shows that \[\int_0^x \left| \int_0^{\frac{y}{2}}  \bar{\mu}^{(1)}(y-r)W^{(1)}(r) d r \right|  d y<\infty\] and we conclude that  $\kappa^{(1)}\in L^1_{loc}(\R^+)$ which completes the proof of \eqref{it:sigma_nul_W}. Under the assumptions of \eqref{it:mu_finite}  we have by Lemma \ref{lem_conv_diff_onederiv}, for any $y>0$,
\begin{equation*} %\label{deriv_kappa_3}
\kappa^{(1)}(y) =   \int_0^y \bar{\mu}^{(1)}(y-r)W^{(1)}(r) d r + \bar{\mu}(0) W^{(1)}(y),
\end{equation*}
and, thus $\kappa\in C^1(\R^+)$  and $\kappa^{(1)}\in L^1_{loc}(\R^+)$, that is, \eqref{it:mu_finite}   holds.
Assume now that $\sigma >0$ and $\bar{\mu} \in C(\R^+)$. Then, by  \cite[Theorem 1]{Chan-Kyp-Savov}, $W^{(2)}\in C(\R^+)$. Moreover, recalling, from \eqref{eq:pot}, that $W$ can be seen as the integrated potential measure of a subordinator whose L\'evy measure   has a density given by ${\overline{\Pi}}$, it follows by \cite[Corollary 9]{Chan-Kyp-Savov} that, for any $y>0$,
\begin{equation}\label{scale_2ndderiv_sigma}
W^{(2)}(y) = \frac{1}{\sigma^2}\sum_{n=1}^\infty  \left(- \frac 1{\sigma^2} \right)^n \overline{\overline{\Pi}}{}^{*n}(y).
\end{equation}
Since $\int_0^y\overline{\overline{\Pi}}{}^{*n}(r) d r = \left( \int_0^y \overline{\overline{\Pi}}(r) d r \right)^n$
and  $ \lim_{y \downarrow 0}\int_0^y \overline{\overline{\Pi}}(r) d r =0$, it follows that $\int_0^y |W^{(2)}(r)| d r<\infty$ for $y>0$ small enough and thus for all $y>0$ since $W^{(2)}$ is locally bounded. Hence $W^{(2)}\in L^1_{loc}(\R^+)$. As $W^{(1)}(0^+)=\sigma^{-2}$, we can use Lemma \ref{lem_conv_diff_onederiv} to deduce that for any $y>0$,
\begin{equation}\label{deriv_kappa_1}
\kappa^{(1)}(y) = b W^{(2)}(y) + \int_0^y W^{(2)}(y-r)\bar{\mu}(r) d r + \sigma^{-2} \bar{\mu}(y).
\end{equation}
Therefore $\kappa^{(1)} \in L^1_{loc}(\R^+)$ as  $W^{(2)}\in L^1_{loc}(\R^+)$ and, by Lemma \ref{lem_conv_diff_onederiv} again, $\kappa\in C^1(\R^+)$ as  $W^{(2)}\in L^1_{loc}(\R^+)\cap C(\R^+)$  and by assumption $\bar{\mu}\in C(\R^+)$. This proves the first claim of \eqref{it:sigma_pos}.
In order to prove the last claim, assume further that   $b=0$, $\bar{\mu}(0)<\infty$,  $\bar{\mu}\in C^1(\R^+)$ and $\bar{\mu}^{(1)}\in L^1_{loc}(\R^+)$.  Then we can use Lemma \ref{lem_conv_diff_onederiv} on \eqref{deriv_kappa_1} to deduce that $\kappa^{(1)}\in C^1(\R^+)$ with, for $y>0$,
\begin{equation} \label{eq:k2}
\kappa^{(2)}(y) =  \int_0^y W^{(2)}(y-r)\bar{\mu}^{(1)}(r) d r + \frac1{\sigma^2} \bar{\mu}^{(1)}(y) + \bar{\mu}(0)W^{(2)}(y).
\end{equation}

\section{Smoothness   of the invariant measure} \label{sec:iv} %$\overline{\phi}(\theta)$-stationary
In this section we investigate fine distributional properties of the density of the absolutely continuous part of the invariant measure of the CBI-semigroups. We already point out that although we restraint our analysis to the framework of this paper, our results extend modulo mild modifications to the most general case. For sake of completeness, we start by stating and providing a short and  original  proof of some basic properties of this  invariant measure whose study  traces back to the work of Pinsky \cite{Pinsky}, see also  \cite{Keller} and \cite{Li_Book} for more recent references. %  whose main results are included in the next statement.
\begin{prop}\label{prop:invariant_measure} Let $( \psi, \phi)\in {\Ne} \times \Be$. % with $\phi \not\equiv 0 $.
\begin{enumerate}
\item 	
	The \CBI semigroup $P$ admits a unique  invariant probability measure $\mathcal V$ on $\Rpo$, in the sense that, for any $f \in C_0(\Rpo)$,
\begin{equation}\label{eq:inv}
{\mathcal V}P_tf  = {\mathcal V}f.
\end{equation}
The  measure $\mathcal V$ is  infinitely divisible and its Laplace exponent is the function $\piv$, see \eqref{eq:def_Phi}, which is the following Bernstein function
\begin{equation} \label{eq:lt_nup}
	\piv(\lambda) = \int_0^{\infty} (1- e^{-\lambda r}) \frac{\kappa(r)}{r}dr.
	\end{equation}
where  $\kappa$,  defined in \eqref{eq:def_k}, satisfies $\int_0^{\infty}\kappa(r)dr<\infty$.
\item   There exists  $\nu \in L^1(\R^+)$  such that %$\nu>0$ a.e.~and  we have
\begin{equation}\label{decomposition_of_Nu}
\mathcal V(dy) = e^{-\biv}\delta_0(dy) + \nu(y)dy,
\end{equation}
where we recall that $\biv = \lim_{\lambda \to \infty} \piv(\lambda)$.
Moreover, if $\phi\equiv 0$, then $\nu= 0$ a.e.~and  $\nu>0$ a.e.~otherwise.
\end{enumerate}
\end{prop}

\begin{remark}
If $(\overline{\psi}, \overline{\phi}) \in \Nb \times \Bb$ with $(\psi, \psi )= \mathcal E(\overline{\psi}, \overline{\phi})$ then  the measure $\mathcal{V}_{\theta}(dx) = e^{\theta x} {\mathcal V}(dx)$ is a $\overline{\phi}(\theta)$-stationary measure for the  \CBIb semigroup $\overline{P}$, in the sense that, for any $f \in C_0(\Rpo)$,
	\[
	{\mathcal V}_{\theta} \overline{P}_tf = e^{-\overline{\phi}(\theta)t}{\mathcal V}_{\theta} f.
	\]
\end{remark}

\begin{proof}  We shall prove that $\piv$ is a Bernstein function.
	By \eqref{eq:def_Bernstein} we have, for all $u>0$,
	\begin{equation}\label{eq:phi_over_psi}
	\frac{\phi(u)}{\psi(u)} = \frac{u}{\psi(u)}\left(\int_0^\infty e^{-ur}\left( b\delta_0(dr)+\overline{\mu}(r)dr\right)\right).
	\end{equation}
Thus, using the relation \eqref{eq:pot}, by convolution and injectivity of the Laplace transform, we get
	\begin{equation}\label{eq:phi_over_psi}
	\frac{\phi(u)}{\psi(u)} = \int_0^{\infty}e^{-u r}\kappa(r)dr
	\end{equation}
	where $\kappa$ is defined in \eqref{eq:def_k}. It is clear that $\kappa (r) \geq 0$, for all $r>0$ and since $\phi, \psi \in \mathcal{H}(R)$ for some $R>0$, with $\phi(0)=\psi(0)=0$, $m >0$ one has $\lim_{u \downarrow 0} \frac{\phi(u)}{\psi(u)} =\frac{\phi^{(1)}(0)}{m } <\infty$ and hence $\int_0^\infty \kappa(r)dr <\infty$.
Now, integrating \eqref{eq:phi_over_psi} yields, for all $\lambda\geq 0$,
	\begin{eqnarray*}
		\piv(\lambda)= \int_{0}^{\lambda} \frac{\phi(u)}{\psi(u)}du=  \int_0^{\infty}(1-e^{-\lambda r}) \frac{\kappa(r)}{r}dr,
	\end{eqnarray*}
and,  one can check that   $
\int_0^\infty (1 \wedge r)\frac{\kappa(r)}{r} dr \leq \int_0^\infty \kappa(r)dr <\infty$. Thus,  %$\piv$ is a Bernstein function so that
there exists an infinitely divisible measure $\mathcal V$ on $\Rpo$ such that, for all $\lambda \geq 0$,
\begin{equation} \label{eq:lt_nu}
\mathcal{V} e_{\lambda} = e^{-\piv(\lambda)}.
\end{equation}
 Since %by \ref{Analicity_of_A_B_F}
$\piv(0)=0$, the measure $\mathcal V$ is a probability measure.
Now, from \eqref{eq:lt_cbi}, we have for all $t,\lambda \geq 0$,
\[
{\mathcal V}P_te_{\lambda} = e^{-\piv(\lambda) }e^{\piv(B(A(\lambda) e^{-mt})) } {\mathcal V} e_{B(A(\lambda) e^{-mt})} =	
	 {\mathcal V}e_{\lambda}.
\]
Since, by the Stone-Weistrass Theorem, $\Lambda={\rm Span}(e_\lambda)_{\lambda>0}$ is dense in $C_0(\Rpo)$, this proves \eqref{eq:inv}, i.e.~that $\mathcal{V}$ is an invariant probability measure for $P$. One obtains that it is unique by showing, from \eqref{eq:inv}, that its Laplace transform is the unique solution to  some ordinary differential equation with initial condition and invoking injectivity property of the Laplace transform. We refer the reader to Ogura \cite[Proposition 1.1]{Ogura-70} for more details.
Finally we see, from Sato \cite[Theorem 27.7]{Sato-99},  that
a sufficient condition for  $\mathcal V$ to be absolutely continuous is $\int_0^{\infty}\frac{\kappa(r)}{r}dr=\infty $,
that is, %, %by the monotone convergence theorem,
$\biv =\piv(\infty)=\infty$. Moreover, in this case, we have from \cite[Theorem 1]{Hudson} that a.e.~$\nu>0$.
 Assume now that $\int_0^{\infty}\frac{\kappa(r)}{r}dr  < \infty $. Then $\mathcal V$ is a compound Poisson distribution associated to the L\'evy measure $\frac{\kappa(r)}{r}dr$. If $\phi\equiv 0$, then $\kappa\equiv 0$, which implies $\mathcal V(d y)=\delta_0(d y)$ and so $\nu=0$ a.e. Otherwise, when $b>0$ or $0<\bar{\mu}(0^+)\leq \infty$, then from    \eqref{eq:def_k} combined with  $W^{(1)}>0$ (see Lemma \ref{eq:propW}), we note that for all $r>0$, $\frac{\kappa(r)}{r} > 0$.   Then, by means of  \cite[Remark 27.3]{Sato-99} (with $t=1$ and recalling that therein the notation $\nu^{*0}= \delta_0$ is used), and using the fact that in our case the L\'evy measure has a density, we conclude that there exists some positive density $\nu$ on $\R^+$ such that $\mathcal V(dy) = e^{-\biv}\delta_0(dy) + \nu(y)dy$.  %The last statement follows readily from the previous one combined with Proposition \ref{prop:transf}.
\end{proof}
We proceed by giving  (a necessary and) sufficient conditions for the absolute continuity of the invariant measure, that is, conditions  for $\biv=\infty$ in \eqref{decomposition_of_Nu}.

\begin{lemma} \label{elm:abs}
If  $\phi\equiv 0$, i.e.~$P$ is a \CB semigroup, then $\biv=0$. Otherwise, we have   $\biv<\infty$ (resp.~$\biv=\infty$)  if and only if   $\int_0^1\frac{\kappa(r)}{r}dr <\infty$ (resp.~$\int_0^1\frac{\kappa(r)}{r}dr =\infty$).  A sufficient condition  for $\biv=\infty$ is $\underline{\kappa}(0^+)=\liminf_{r\downarrow 0} \kappa(r)>0$. %, which, from Proposition \ref{lem:kap}, holds if  $\liminf_{u \rightarrow \infty}\frac{\phi(u)}{\phi_p(u)} >0.$
 %In particular, this is true if $b>0$ or $\liminf_{r \rightarrow 0}\frac{\int_0^{r}\overline{\mu}(y)dy}{\int_0^{r}\overline{\overline{\Pi}}(y)dy}>0$.
 \end{lemma}
 \begin{proof}
 The first claim is obvious and the necessary and sufficient condition can be easily deduced from the proof of Proposition \ref{prop:invariant_measure}. Next, assume that $\underline{\kappa}(0^+)>0$, then there exists $C,\epsilon>0$ such that for small $0<r<\epsilon$, $\kappa(r)>C$ which implies that $\int_0^{1}\frac{\kappa(r)}{r}dr \geq C\int_0^{\epsilon}\frac{dr}{r} =\infty$ and completes the proof of the Lemma.
 \end{proof}

\begin{cor}\label{corol:omega}
Let $\psi\in\mathcal N$. Then there exists a proper probability density function $\omega$ on $\mathbb R^+$ such that
\begin{equation*}
\int_0^\infty e^{-\lambda y} \omega(y) d y = 1-A(\lambda), \quad \lambda\geq 0,
\end{equation*}
where we recall that $A(\lambda)=\exp \left( -m\int_{\lambda}^{\infty}\frac{d u}{\psi(u)} \right)$.
\end{cor}
\begin{proof}
The function $A$ satisfies the differential equation $\psi(\lambda)A^{(1)}(\lambda)=m A(\lambda)$. Differentiating on both sides and rearranging gives $A^{(2)}(\lambda)=-\frac{\psi^{(1)}(\lambda)-m}{\psi(\lambda)}A^{(1)}(\lambda)$, which leads to
\begin{equation*}
A^{(1)}(\lambda)=A^{(1)}(0)  \exp \left( - \int_0^\lambda \frac{\psi^{(1)}(u)-m}{\psi(u)} d u  \right).
\end{equation*}
Since $(\psi,\psi^{(1)}-m)\in {\Ne} \times \Be$,  Proposition \ref{prop:invariant_measure}  yields that there exists a probability measure $ \underline{\mathcal{V}}$ on $\Rpo$ such that
\begin{equation}\label{eq:Al}
\int_0^{\infty}e^{-\lambda y}\underline{\mathcal{V}}(d y ) = \frac{A^{(1)}(\lambda)}{A^{(1)}(0)}.
 \end{equation}
 Moreover, by assumption, $\lim_{u\to\infty}\psi(u)=\infty$ and $\int_1^\infty \frac{du}{\psi(u)}  <\infty$, which yields
\begin{equation*}
 \int_0^\infty \frac{\psi^{(1)}(u)-m}{\psi(u)} d u = \int_0^1 \frac{\psi^{(1)}(u)-m}{\psi(u)} d u + \left[\log \psi(u)\right]_{u=1}^\infty - m \int_1^\infty \frac{du}{\psi(u)}  = \infty.
\end{equation*}
This  implies by Proposition \ref{prop:invariant_measure} that $\underline{\mathcal{V}}$ is absolutely continuous on $\Rpo$ and we denote its probability density function by $\underline{\nu}$.
Since, from \eqref{theta_positive}, $\lim_{\lambda \to 0}\int_{\lambda}^{\infty}\frac{d u}{\psi(u)}=\infty$, we get $A(0)=0$ and thus it follows by Tonelli Theorem,
\begin{equation*}
\begin{split}
1- A(\lambda) = 1- \int_0^\lambda A^{(1)}(u)d u = & 1- A^{(1)}(0)\int_0^\lambda \left( \int_0^\infty e^{-u y} \underline{\nu}( y ) d y \right) d u \\
 = & 1- A^{(1)}(0) \int_0^\infty \left( 1- e^{-\lambda y} \right) \frac{\underline{\nu}(y)}{y} d y.
\end{split}
\end{equation*}
Because $\lim_{\lambda\to\infty}A(\lambda)=1$, we therefore must have
\begin{equation} \label{eq:Ao}
A^{(1)}(0) \int_0^\infty \frac{\underline{\nu}(y)}{y} d y =1
\end{equation} and hence the corollary has been proved with
\begin{equation*}
\omega(y) = \frac{ \underline{\nu}(y)}{ y \int_0^\infty \frac{\underline{\nu}(r)}{r} d r }.
\end{equation*}
\end{proof}

\begin{lemma} \label{Analicity_of_A_B_F} First, we have  $\psi \in \mathcal{H}(R)$ (resp.~$\phi \in \mathcal{H}(R)$) if and only if $\int^\infty e^{u r}\Pi(dr)<\infty$ (resp.~$\int^\infty e^{u r}\mu(dr)<\infty$) for all $u<R$.  Next, let $(\psi, \phi) \in \Ne\times \Be$. Then $R_{\psi}, R_\phi >0$ and $\piv \in \mathcal{H}(R_{\piv})$ with $\piv(0) =0$ and $R_{\piv}=R_{\psi}\wedge R_{\phi}\wedge \underline \theta$  where  $ \underline \theta =\inf\{u > 0; \: \psi(-u) =0\}\in (0,\infty]$.  %$R_{\piv}=R_{\psi}\wedge R_{\phi}\wedge \overline{\theta}$  where with $\theta =\inf\{u > 0; \: \psi(-u) =0\}$ we have set   $\overline{\theta}=\infty\mathbb{I}_{\{\theta=0\}}+\theta$.
Similarly $A \in \mathcal{H}(R_A)$ with $R_A=R_\psi\wedge\underline \theta$, $A(0)=0$ and  $A^{(1)}(\lambda)>0$ for any $\lambda>-R_A$. As a by-product,  $B \in \mathcal{H}(R_B)$  with $0<R_B\leq 1$  and $B(0) =0$. Finally, there exists some $0<R_0\leq 1$, such that, for all $x\geq 0$, $\Gx \in \mathcal{H}(R_0)$.
\end{lemma}
\begin{proof}
 The first claim follows readily from \cite[Theorem 25.17]{Sato-99}. Next, under the assumptions \eqref{eq:standing_assumption},  we have $\psi \in \mathcal{H}(R_{\psi})$ with $ \psi(0)= \phi(0)=0$ and $ \psi^{(1)}(0)=m >0$. Moreover, as from Lemma \ref{eq:propW}, $u\mapsto \frac{u}{\psi(u)}$ is the Laplace transform of a positive measure, by a standard result on the Laplace transform, its first singularity, as a function of the complex variable, occurs on the (negative) real line. Since $\psi(0)=0$, $\psi^{(1)}(0)=m>0$ and $\psi \in \mathcal{H}(R_\psi)$,  we deduce that the first singularity of $z\mapsto \frac{z}{\psi(z)}$  in the disc $\{z\in \C; |z|<R_{\psi}\}$ can  be only, if it exists, a  zero $-\underline \theta<0$ of $\psi$, which is isolated from $0$. Thus,  $\frac{z}{\psi(z)} \in\mathcal{H}(R_{\psi} \wedge \underline \theta)$, and, since   $\phi \in \mathcal{H}(R_{\phi})$, we conclude that $\piv \in \mathcal{H}(R_{\piv})$.
%$\frac{\psi(u)}{u} = m  (1+f_1(u))$, $\frac{\phi(u)}{u} = \phi^{(1)}(0) + g(u)$  where $f_1$ and $g$ belong to $\mathcal{H}_0$ and $f_1(0)=g(0)=0$.
%Therefore, for $u$ in a neighborhood of $0$,
%\[
%\frac{\phi(u)}{\psi(u)} = \frac{\phi^{(1)}(0) + g(u)}{m  } (1+f(u))
% \]
%where $f$ belongs to $\mathcal{H}_0$ and $f(0)=0$, from which one deduces the above stated properties of $\piv$.
  From the proof of Corollary \ref{corol:omega}, we easily get, by combining \eqref{eq:Ao} with \eqref{eq:Al}, that $A^{(1)}\in\mathcal{H}(R_A)$ with  $R_A=R_\psi\wedge\underline \theta>0$ and $ A^{(1)}(\lambda) > 0$ for any $\lambda>-R_A$. Therefore  $A\in\mathcal{H}(R_\psi\wedge\underline \theta)$ as well.	Then by the Lagrange inversion theorem, the inverse function $B$ of $A$ belongs to $\mathcal{H}(R_B)$ with $R_B>0$ and satisfies $B(0) = 0$ and $B^{(1)}(0) \neq 0$. Finally, $R_B\leq 1$ as $\lim_{\lambda \to \infty}A(\lambda)=\lim_{\lambda \to \infty}e^{-m\int_\lambda^\infty \frac{du }{\psi(u)}}=1$. The last statement follows readily from the previous ones.
  \end{proof}

\begin{remark}\label{remark_determ_T0}
Our results provide the smoothness of the transition kernel for $t>T_0$. Looking at \eqref{def:t0} and \eqref{eq:def_Gx}, we see that in order to determine $T_0$, we need to know  $R_B$, the radius of convergence of the Taylor series at $0$ of $B$. Though we know by Lemma \ref{Analicity_of_A_B_F} that $0<R_B\leq 1$, it would be interesting to find the precise value of $R_B$ when $A$ can not be inverted explicitly.   %Consequently, it seems hard to determine $T_0$ unless one can get an explicit expression for $B$  by inverting the function $A$ as for instance done in Example \ref{example_specrep}, but the number of examples for which this is possible seems to be small.
 In order to get an idea of the value of $T_0=-\ln(R_0)/m$ in specific examples, one can instead proceed by numerically computing $R_0$, the radius of convergence of the power series \eqref{eq:def-Ln_statement}. Note that in \cite{CLP-F} we have given an algorithm   for computing the  eigenfunctions $\mathcal L_n(x)$.
\end{remark}

With the aim of studying regularity properties of the  heat kernel of  CBI-semigroups, we  need to deepen significantly the analysis concerning  fine distributional properties of the invariant measure and in particular derive smoothness properties of  its absolutely continuous part. We provide both a smoothness result on $\R$ as well as on $\R^+$. The question of how smooth infinitely divisible distributions are on $\mathbb R$  has been well-studied, see Section 28 of \cite{Sato-99} for an overview.  However for distributions with support on $\mathbb R^+$, the approaches developed in the literature are limited  to the case when the density (or its derivatives) vanishes at $0$
%%We shall in fact provide a smoothness result on $\R$  which is based on classical Fourier analysis NOT EXACTLY TRUE.  However, the main drawback of this approach for investigating   the smoothness properties of positive variable is limited, for obvious reasons, to the case when the density (or its derivatives) vanishes at $0$.
and to the best of our knowledge, we are not aware of any available techniques in the literature to deal with the smoothness on $\R^+$. To this purpose, we shall %exploit the infinitely divisible structure of the invariant measure to
 derive a convolution equation that the absolutely continuous part of the invariant measure satisfies and then apply the two lemmas in Section \ref{sec_conv} to establish its degree of smoothness. The same technique will be used  again in Lemma \ref{lem_smooth_mathcalW} to derive smoothness properties of the function $\mathcal W_n$.   The next results %whose main focus is the smoothness properties of its density
  can be seen as a significant complement of the previous works on the study of invariant measures of CBI-semigroups.

\begin{prop}\label{thm:regularity_nu}
Recall from \eqref{decomposition_of_Nu} the notation $\mathcal V(d y) = e^{-\biv}\delta_0(d y) +\nu(y) d y, y>0$. We extend $\nu:\mathbb R^+\to\mathbb R$ to $\mathbb R$ by setting $\nu(y)=0$ for $y\leq 0$. Then, the density $\nu$ can be chosen such that it satisfies the following   smoothness properties.
\begin{enumerate}[(a)]
    \item \label{it:rR}If $\Si\in [1,\infty]$, then $\nu \in C^{\Si-1}(\mathbb R)$ where we recall that $\Si$ is defined in \eqref{def:kappa_bar}.
    \item  Assume $\underline{\kappa}<\infty$.
    \begin{enumerate}[(b1)]
    \item \label{it:rp} If $\Si\geq 1$, then $\nu\in C^{\Si}(\R^+)$ with $\nu^{(\Si)}\in L^1(\R^+)$  and if $\Si=0$  and $\overline{\kappa}(0^+)<\infty$, then $\nu \in C(\R^+) \cap L^1(\R^+)$.

    \item \label{it:r2} If  $\underline{\kappa}(0^+)=\overline{\kappa}(0^+)$, $\kappa\in C^q(\R^+)$ for some $q\geq 1$  and $\kappa^{(1)}\in L^1_{loc}(\R^+)$, then $\nu\in C^{\Si+q}(\R^+)$.
    \end{enumerate}
        \end{enumerate}
\end{prop}
\begin{proof}
The first claim  follows from  \cite[Theorem 6]{Wolfe-Cont}, see also \cite[Theorem 28.4]{Sato-99}.  %Note that strictly speaking the case $\Si=\infty$ is not contained in that theorem, but this special case follows easily from the proofs in \cite{Wolfe-Cont}.
Next assume that $\underline{\kappa}<\infty$.  We first show that $\kappa \in C(\R^+)$. Since $\bar{\mu}$ is non-increasing  on $\R^+$, it is continuous almost everywhere.  Hence, by the dominated convergence theorem,   one obtains, from \eqref{eq:def_k}, that for every $y>0$,
 \begin{eqnarray*}
 \lim_{\delta\to 0} \kappa(y+\delta)&  = & \lim_{\delta\to 0} \left( b  W^{(1)}(y+\delta)+ \int_0^{\frac{y+\delta}{2}}  W^{(1)}(y+\delta-r)\bar{\mu}(r) + \bar{\mu}(y+\delta-r)W^{(1)}(r)   d r  \right) \\
& = & b  W^{(1)}(y)+\int_0^{\frac{y}{2}}  W^{(1)}(y-r)\bar{\mu}(r)   d r   +   \int_0^{\frac{y}{2}}  \bar{\mu}(y - r) W^{(1)}(r)   d r   \\
 &= & \kappa(y)
 \end{eqnarray*}
 where we used, for the second identity, the fact that $W^{(1)}$ is continuous, see Lemma \ref{eq:propW}.  Next, differentiating \eqref{eq:lt_nu}, we get, for $\lambda>0$,
 \[ \int_{0}^{\infty} e^{-\lambda y} y \nu(y)dy = \int_0^{\infty}e^{-\lambda y} \kappa(y)dy \: e^{-\piv(\lambda)}, \]
 that is the density $\nu$ is (can be chosen as) the solution to the convolution equation, for any $y>0$,
 \begin{equation}\label{conv_eq_posinfdiv}
 y\nu(y)=\int_0^{y}\nu(y-r)\kappa(r) dr + e^{-\biv}\kappa(y) = \int_0^{y}\kappa(y-r) \nu(r)d r + e^{-\biv}\kappa(y).
 \end{equation}
 %Note that since by assumption  $\underline\kappa(0^+)=\overline\kappa(0^+)<\infty$ NOT ASSUMED AT THIS STAGE and   further $\kappa$ is continuous  on $(0,\infty)$, we get by an application of the dominated convergence theorem that $\nu$ is  continuous on $(0,\infty)$.
If  $\Si\geq 1$, then, from Lemma \ref{elm:abs}, $\biv=\infty$ and, from item \eqref{it:rR}, $\nu\in C^{\Si-1}(\mathbb R)$. Thus, we can use Lemma \ref{lem_conv_diff_onederiv} (repeatedly if necessary) to deduce,
 \begin{equation*}
(y\nu(y))^{(\Si-1)} =  \int_0^{y}\kappa(y-r)\nu^{(\Si-1)}(r) dr.
 \end{equation*}
 By    \cite[Theorem 6]{Wolfe-Cont}, $\nu^{(\Si-1)}$ has a  derivative $\nu^{(\Si)}$ that lies in $L^1(\R^+)$. Then since $\kappa \in C(\R^+)$ and $\nu^{(\kappa-1)}(0)=0$, Lemma \ref{lem_conv_diff_onederiv} yields that $\nu\in C^{\Si}(\R^+)$ and
 \begin{equation}\label{conv_eq_posinfdiv_pthorder}
(y\nu(y))^{(\Si)} =  \int_0^{y}\nu^{(\Si)}(y-r)\kappa(r) dr = \int_0^{y}\kappa(y-r)\nu^{(\Si)}(r) dr.
 \end{equation}
If $\Si=0$  and $\overline{\kappa}(0^+)<\infty$, then since $\kappa \in C(\R^+)$, we get by an application of the dominated convergence theorem and \eqref{conv_eq_posinfdiv} that $\nu \in C(\R^+)$. This proves \eqref{it:rp}. For the last claim, assume  that $\underline{\kappa}(0^+)=\overline{\kappa}(0^+)$, $\kappa\in C^1(\R^+)$    and $\kappa^{(1)}\in L^1_{loc}(\R^+)$.  Then by applying Lemma \ref{lem_conv_diff_onederiv} to \eqref{conv_eq_posinfdiv} in the case where  $\underline{\kappa}=0$ and to \eqref{conv_eq_posinfdiv_pthorder} in the case where  $\underline{\kappa}\geq 1$, we get
  \begin{equation*}
 (y\nu(y))^{(\Si+1)} =  \int_0^{y}\kappa^{(1)}(y-r)\nu^{(\Si)}(r) dr + \Si(0^+) \nu^{(\Si)}(y) + e^{-\biv}\kappa^{(1)}(y).
  \end{equation*}
 Hence $\nu\in C^{\Si+1}(\R^+)$.  If further $\kappa\in C^q(\R^+)$ for some $q\geq 2$, then one can use Lemma \ref{lem_conv_diff_split} and an induction argument to deduce that  $\nu\in C^{\Si+q}(\R^+)$.
\end{proof}

\section{Eigenfunctions: existence, properties and uniform bounds} \label{sec:polyn}   %and Sheffer polynomials}
In this part we investigate analytical properties of the  sequence of eigenfunctions for CBI-semigroups.
\begin{prop} \label{prop:Analicity_of_G}
For any $n =0,1,\ldots,$ and $x,t\geq0$, we have
\begin{equation} \label{eq:pol_eig}
P_t \mathcal{L}_n(x) = e^{-\lambda_n t} \mathcal{L}_n(x)
\end{equation}
where we recall that  $({\mathcal L}_n)_{n \geq 0}$ is the family of Sheffer polynomials whose generating function is $\Gx$, and, for all $\mathfrak{p}\geq 0$,
	\begin{equation}\label{eq:uniform_convergence}
	\frac{d^{\mathfrak{p}}}{d x^{\mathfrak{p}}}\Gx(z)  =  \sum_{n=\mathfrak{p}}^{\infty} {\mathcal L}^{(\mathfrak{p})}_n(x) z^n,  \quad \vert z \vert < R_0,
	\end{equation}
	where the series is locally uniformly convergent in $x$. Moreover,  for any $R \in (0,R_0)$, there exist $C=C(R)>0$ and $\overline{B}=\max_{|z|= R}|B(z)|\in \R_+$
 such that, for any $x \geq 0$, $n\geq0$,
 \begin{eqnarray} \label{eq:bpol}
 \left|\mathcal{L}_n(x)\right| &\leq& C \frac{e^{\overline{B} x}}{R^{n+1}}.
\end{eqnarray}
  \end{prop}

\begin{remark}
If $(\overline{\psi}, \overline{\phi}) \in \Nb \times \Bb$ with $(\psi, \psi )= \mathcal E(\overline{\psi}, \overline{\phi})$, then,  writing  ${\mathcal L}_n^{\theta}= e_\theta{\mathcal L}_n$, we have, for all $t,x \geq 0$, $n\geq 0$,
 		\[
 		\overline P_t{\mathcal L}^{\theta}_n(x) = e^{-\lambda_n t}{\mathcal L}^{\theta}_n(x).
 		\] 	
\end{remark}

\begin{proof}
Since, from Lemma \ref{Analicity_of_A_B_F}, we have for all $x\geq 0$, $\Gx(z)  =  e^{\piv(B(z))-xB(z)} \in  \mathcal{H}(R_0), \: R_0>0$, an application of the Cauchy's formula yields
that
\begin{equation}  \label{eq:pol_cont}
 \mathcal{L}_n(x) = \frac{1}{2 \pi i}\oint \frac{\Gx(z)}{z^{n+1}}dz
\end{equation}
where the contour is a circle centered at $0$ and of radius $R<R_0$. Since the functions  $B$ and $ \piv \circ B  \in \mathcal{H}(R_0)$, they are bounded on this contour and we get that
\begin{eqnarray} \label{eq:est_cont}
 \left|\mathcal{L}_n(x)\right| &\leq& \frac{1}{2 \pi }\oint \left|\frac{\Gx(z)}{z^{n+1}}\right| dz \leq   \frac{e^{\overline{B}x}}{2 \pi}\oint\left| \frac{e^{\piv(B(z))}}{z^{n+1}}\right|dz \leq C R^{-n-1} e^{\overline{B}x}
\end{eqnarray}
where $C>0$. % and we used the fact that $\piv \circ B  \in \mathcal{H}(R_0)$.
This proves \eqref{eq:bpol}.
Next, since $B \in \mathcal{H}(R_0)$ with $B(0) =0$, we can choose $R$ such that $\overline{B}=\max_{|z|=R }|B(z)|<R_0$.   From  \eqref{eq:lt_cbi}, we have, for any $\Re(z) >0$,
\begin{equation*}
P_te_z (x) =  e^{-\piv(z)}\Gx(A(z)e^{-mt}).
\end{equation*}
Since $z \mapsto e^{-\piv(z)}\Gx(A(z)e^{-mt}) \in \mathcal{H}(R_0)$, by the principle of analytical continuation, for all $x,t\geq 0$,  $z \longmapsto P_te_z(x) \in \mathcal{H}(R_0) $ and thus from \eqref{eq:bpol} we get that $P_t | \mathcal{L}_n|(x)\leq C R^{-n-1} P_te_{-\overline{B}}(x) <\infty$ as  $\overline{B}<R_0$.
Thus, by Fubini Theorem, one gets, for any $x\geq0$,
\begin{eqnarray*}
 P_t \mathcal{L}_n(x) &=& \frac{1}{2 \pi i}\oint \frac{P_t(G_.(z))(x)}{z^{n+1}}dz = \frac{1}{2 \pi i}\oint \frac{\Gx(ze^{-mt})}{z^{n+1}}dz \\ &=& e^{-m n t}\frac{1}{2 \pi i}\oint_t \frac{\Gx(\bar{z})}{\bar{z}^{n+1}}d\bar{z} \\
&=& e^{-\eign t}\mathcal{L}_n(x)
\end{eqnarray*}
where we performed an obvious change of variable and we used the following identities, with the obvious notation,
\begin{equation}
P_t(G_.(z))(x) = e^{\piv(B(z))}P_te_{B(z)}(x)= e^{\piv(B(z))} e^{-\piv(B(z))}\Gx(ze^{-m t})= \Gx(ze^{-m t}). \end{equation}
 This completes the proof of the first statement with $\mathfrak{p}=0$. The case $\mathfrak{p}\geq 1$  follows again from Cauchy's formula after observing that the mapping $z \mapsto \frac{d^{\mathfrak{p}} }{d x^{\mathfrak{p}}}\Gx(z) = (-1)^\mathfrak{p}B^{\mathfrak{p}}(z) \Gx(z) \in \mathcal{H}(R_0)$.
\end{proof}

We proceed by providing some additional properties regarding the CBI-semigroups and the Scheffer polynomials considered in this paper.
 	\begin{prop} \label{prop:eigen} Assume that $\phi\neq 0$. Then  $P$ extends to a strongly continuous contraction semigroup on the Hilbert space $L^2(\mathcal V)$, defined in \eqref{eq:l2nu}, which is still denoted by $P$. %, that is with the obvious notation $||P_t f||_{\mathcal{V}}\leq ||f||_{\mathcal{V}}$ for every $t>0$ and $f\in L^2(\mathcal V)$.
 Moreover, for any $t >0$,  \[(e^{-\lambda_n t})_{n \geq 0} \subseteq {\rm S}_p(P_t)=\{z \in \C;\: P_t -z\mathbf{I} \textrm{ is not one-to-one in } L^2(\mathcal V) \},\] that is the point spectrum of $P_t$,   and $( {\mathcal L}_n)_{n \geq 0}$ is a complete sequence of eigenfunctions of $P_t$ in $L^2(\mathcal V)$.
However, the sequence  $({\mathcal L}_n)_{n \geq 0}$ is formed of orthogonal polynomials in some weighted $L^2$ space if and only if $({\mathcal L}_n)_{n \geq 0}$ is the sequence of  Laguerre polynomials,  i.e.~$\psi(u) = \sigma^2u + mu, \sigma^2 > 0, m \geq 0$, and $\phi(u)=bu$, $b>0$. Hence, beyond this case, $P_t$ is a non-self-adjoint contraction semigroup in $L^2(\mathcal V)$. Finally, the algebra of polynomials $\mathcal{A}$ is a core of its generator $\mathbf{L}$.
 	\end{prop}

 	\begin{remark}
Although our proof for the non-self-adjointness property  of CBI-semigroups is rather straightforward, we mention that, in a recent and interesting paper,  Handa \cite{CBI-Handa} has shown that the invariant measure of the so-called CIR semigroup, that is the Gamma distribution, is the only reversible probability measure within the entire class of CBI invariant probability measures, which means that beyond the diffusion case they are non-self-adjoint.
 \end{remark}

 	\begin{proof}
Since from Proposition \ref{prop:invariant_measure}, $\mathcal{V}$ is an invariant measure and $P$ is a Feller semigroup, we deduce from Theorem 5.8 in \cite{Daprato-06} the first claim, that is, $P$  admits a unique contraction extension in $L^2(\mathcal V)$. Next, since, for any $\lambda>0$,  ${\mathcal V}e_\lambda= e^{-\piv(\lambda)}$ with, from Lemma \ref{Analicity_of_A_B_F}, $\piv \in \mathcal{H}(R_{\piv})$, we get that for any  $0<\epsilon<R_{\piv}$, $\int_0^{\infty}e^{\epsilon y} \mathcal V(dy)<\infty$. Hence, the probability measure $\mathcal V$ is moment determinate and according to \cite[Corollary 2.3]{Akhiezer-65} the sequence $(\mathcal{L}_n)_{n \geq 0}$ is complete in $L^2(\mathcal V)$.  Since for all $n\geq0$, $\mathcal{L}_n \in L^2(\mathcal V)$, the property of the point spectrum is a direct consequence of the relation \eqref{eq:pol_eig}.
 The statement regarding the orthogonality property of the polynomials follows from a result of Chahira \cite{Chihara-68}, see also \cite{Meixner}, stating that the only sequence of orthogonal polynomials on a (weighted) Hilbert space $L^2$ with support on $\R_+$ and whose generating function is of the form $B(z)e^{xA(z)}$, are the Laguerre polynomials, i.e.~when $\psi(u)=\sigma^2 u + m u, \sigma^2>0$ and $\phi(u)=bu$. This implies that, beyond these cases, the CBI-semigroups are non-self-adjoint in $L^2(\mathcal{V})$.  Using \eqref{eq:pol_eig}, we get, in the  $L^2(\mathcal V)$ topology,  that, for any $n\geq0$,
 \[\mathbf{L} \mathcal{L}_n(x) = \lim_{t\downarrow  0} \frac{P_t\mathcal{L}_n(x)-\mathcal{L}_n(x) }{t} = \lim_{t\downarrow 0} \frac{e^{-\eign t}-1 }{t}\mathcal{L}_n(x) = - \eign \mathcal{L}_n(x), \]
 which completes the proof.
 	\end{proof}

\section{Eigenmeasures: characterization, properties and uniform bounds} \label{sec:meas}
The next result provides the existence as well as explicit representations of the eigenmeasures of the CBI-semigroup, see \eqref{eq:eigenmeasure_P} below for definition. As a by-product, we derive sufficient conditions for  these eigenmeasures to be  absolutely continuous with a smooth density. We also establish a uniform  upper bound which will be needed later for obtaining smoothness properties of the transition densities.

\begin{prop}\label{prop:nu_n} Let $(\psi, \phi)\in {\Ne}\times \Be$ and  $n\geq 0$.
The following bound
\begin{equation}\label{bound_lapl_abs_Vn}
\vert \mathcal{V}_n \vert  e_{\lambda}
\leq e^{-\piv(\lambda)}(2-A(\lambda))^{n}
\end{equation}
holds for any $\lambda>-(R_A\wedge R_\phi)$ where $\vert \mathcal{V}_n \vert$ stands for the total variation of  the (signed) measure  $\mathcal{V}_n$, which we recall   is defined in \eqref{eq:def_Vn}.

Moreover,  $\mathcal{V}_n$ 	 is an eigenmeasure of the \CBI semigroup $P$, in the sense that, for any $f\in {\mathcal D}_{R_A\wedge R_\phi}= \{ f:\Rpo\to\mathbb R \ \mbox{ measurable}; \: f e_{\lambda}  \in L^{\infty}(\R_+) \ \text{for some $\lambda<R_A\wedge R_\phi$} \},$  and $t\geq 0$,
	\begin{equation}\label{eq:eigenmeasure_P}
	\mathcal V_n P_t f = e^{-\lambda_nt}\mathcal V_n f.
	\end{equation}
Next,
	%The measure  $\mathcal{V}_n$  is absolutely continuous if and only if $\biv = \infty$. In this case,
	recall from \eqref{eq:def_Vn} that  $\mathcal V_n(dy) = e^{-\biv}\delta_0(dy) + \nu_n(y)dy$. We have the following smoothness properties of $\nu_n$.
	\begin{enumerate}[a)]
		\item \label{it:s1}If $\Si \in [1,\infty]$,
		then
		$\nu_n\in C^{\Si-1}(\mathbb R)$
		with, for any integer $0 \leq \mathfrak{q}\leq \Si-1$ and $y\geq 0$,
	 \begin{equation*}
		\nu_n^{(\mathfrak{q})}(y) =\mathcal{W}_{n}  *\nu^{(\mathfrak{q})}(y) + \nu^{(\mathfrak{q})}(y), \quad n\geq 1.
	 \end{equation*}
		%where the series converges  locally uniformly on  $\Bbb R^+$.
		Further,  for all $0\leq  \mathfrak{q} \leq \Si-1$, $K>0$, $\lambda>0$ there exists $C=C_{K,\lambda}(\mathfrak{q})>0$ such that, for all $n\geq 1$,
		 \begin{equation*} %\label{bound_mathcalW_pthderiv}
		 \sup_{y\in   \left[0,K\right]} \left| \nu_n^{(\mathfrak{q})}  (y) \right|
		 \leq  C(2-A(\lambda))^n.
		 \end{equation*}
		 \item \label{it:si} Assume $\underline{\kappa}<\infty$.
		    \begin{enumerate}[b1)]
		    \item \label{it:si1} If $\Si\geq 1$, then $\nu_n\in C^{\Si}(\R_+)$  and if $\Si=0$  and $\overline{\kappa}(0^+)<\infty$, then $\nu_n \in C(\R_+)$. In both cases we have that for all $K>0$, $\lambda>0$ there exists $C=C_{K,\lambda}(\Si)>0$ such that, for all $n\geq 1$,
		    		 \begin{equation} \label{bound_mathcalW_pthderiv2}
		    		 \sup_{y\in  \left[K^{-1},K\right]} \left| \nu_n^{(\Si)}  (y) \right|
		    		  \leq   C n (2-A(\lambda))^n.
		    		 \end{equation}
%If  $\underline{\kappa}(0^+)=\overline{\kappa}(0^+)$, $\kappa,W\in C^q(\R_+)$ for some $q\geq 1$ and $\kappa'\in L^1_{loc}(\R_0^+)$, then		
		    \item\label{it:siq} If $\kappa \in C^{1}(\R_+)$, $\kappa'\in L^1_{loc}(\R_+)$  and $\underline{\kappa}(0^+)=\overline{\kappa}(0^+)$, then $\nu_n\in C^{\Si+\bar{\mathfrak{q}}}(\R_+)$ with $\bar{\mathfrak{q}}$  as in   Theorem \ref{thm:regularity}\eqref{it:b2}. Moreover,  for all $1 \leq  l \leq   \bar{\mathfrak{q}}$, $K>0$, $\lambda>0$ there exists $C=C_{K,\lambda}(\Si+l)>0$ such that, for all $n\geq 1$,
		    		 \begin{equation} \label{bound_mathcalW_pthderiv3}
		    		 \sup_{y\in   \left[K^{-1},K\right]} \left| \nu_n^{(\Si +l)}  (y) \right|
		    		 \leq   C n^{l+1} (2-A(\lambda))^n.
		    		 \end{equation}
		    \end{enumerate}
\end{enumerate}
%If $\biv < \infty$ and $\Si \in [0,\infty]$ then $\nu_n \in C(0, \infty)$.
\end{prop}
\begin{remark} Note that if $(\overline{\psi}, \overline{\phi}) \in \Nb \times \Bb$  with $(\psi, \phi )= \mathcal E(\overline{\psi}, \overline{\phi})$ then  the measure $\overline{\mathcal V}_n(dy) = e^{\theta y} {\mathcal V}_n(dy)$ is an eigenmeasure for the  \CBIb semigroup $\overline{P}$, in the sense that
	$
	\overline{\mathcal V}_n\overline{P}_tf = e^{-(\overline{\phi}(\theta)+\eign) t}\overline{\mathcal V}_n f.
	$
\end{remark}
\subsection{Proof of Proposition \ref{prop:nu_n}}
We split the proof into several steps.
Note that for  $n=0$ we have $\lambda_n =0$ and $\mathcal{V}(dy) =\mathcal{V}_0(dy) = e^{-\biv} \delta_0(dy) + \nu_0(y) dy$ with $\nu_0(y) =\nu(y)$. Hence this case corresponds to the study of the invariant measure $\mathcal{V}$ which was addressed in Proposition \ref{prop:invariant_measure}.
\subsubsection{Proof of  \eqref{eq:eigenmeasure_P}}
By the definitions \eqref{eq:co-eigen_series} and \eqref{def_omega}, a classical property of the Laplace transform of a convolution and the binomial theorem,
\begin{eqnarray}
\int_0^\infty e^{-\lambda y} \mathcal{W}_{n}(y) d y & = & \sum_{j=1}^n \binom n j (-1)^j \left( \int_0^\infty e^{-\lambda y} \omega(y) d y \right)^j \nonumber \\
& = & -1 + \sum_{j=0}^n \binom n j (-1)^j (1-A(\lambda))^j \label{lapl_mathcalWn} \\
& = & A(\lambda)^n -1,\nonumber
\end{eqnarray}
where $\lambda >-R_A$. By the triangle inequality $|\mathcal{W}_{n}(y) |\leq \sum_{j=1}^n \binom n j \omega^{*j}(y)$ and so by the same arguments that led  to \eqref{lapl_mathcalWn},
\begin{equation}\label{lapl_absmathcalWn}
\int_0^\infty e^{-\lambda y} |\mathcal{W}_{n}(y) | d y \leq   (2-A(\lambda))^n -1.
\end{equation}
Thus, combining \eqref{def_nu}, \eqref{eq:def_Vn}, \eqref{eq:definition_nu_n}, Lemma \ref{Analicity_of_A_B_F}, \eqref{lapl_mathcalWn}  and \eqref{lapl_absmathcalWn}  we get, for any $\lambda >-(R_A\wedge R_\phi)$,  that
\begin{equation}\label{eq:lt_nu_n_proof}
\mathcal{V}_n e_\lambda =  \left( e^{-\biv} + \int_0^\infty e^{-\lambda y} \nu(y)   d y \right) \left( 1 + \int_0^\infty e^{-\lambda y} \mathcal{W}_{n}(y) d y \right)
  =  e^{-\piv(\lambda)}A(\lambda)^n.
  \end{equation}
  and   the inequality \eqref{bound_lapl_abs_Vn}.
%\begin{equation}\label{bound_lapl_abs_Vn}
%\vert \mathcal{V}_n \vert  e_{\lambda}
%\leq e^{-\piv(\lambda)}(2-A(\lambda))^{n}.
%\end{equation}
On the other hand, from the expression  \eqref{eq:lt_cbi} of the Laplace transform of the semigroup, we obtain that, for any $t,\lambda>0$ and $n\geq0$,
\begin{eqnarray}
\mathcal{V}_n P_t e_\lambda
		& = & e^{-\piv(\lambda)}  \int_0^\infty  \Gx(A(\lambda)e^{-mt})\mathcal{V}_n(dx) \nonumber \\
		& = & e^{-\piv(\lambda)} e^{\piv(B(A(\lambda)e^{-mt}))}  \mathcal{V}_n e_{B(A(\lambda)e^{-mt})}  \nonumber \\
		&= & e^{-\piv(\lambda)} e^{\piv(B(A(\lambda)e^{-mt}))} e^{-\piv(B(A(\lambda)e^{-mt}))}A(\lambda)^ne^{-\eign t} \nonumber \\
	&= & e^{-\piv(\lambda)} A(\lambda)^ne^{-\eign t}=e^{-\eign t} \mathcal{V}_n e_\lambda \label{eq:coe}
\end{eqnarray}
where the third equality is obtained by means of the identity \eqref{eq:lt_nu_n_proof} and recalling that $B$ is the inverse of $A$. We complete the proof of \eqref{eq:eigenmeasure_P} by combining the identities \eqref{eq:lt_nu_n_proof} and \eqref{eq:coe} and invoking the injectivity property of the Laplace transform.

\subsubsection{A key lemma}
Before we continue with the proof of Proposition \ref{prop:nu_n}, we need the following lemma on the function $\mathcal W_n$ defined in \eqref{eq:co-eigen_series}.
% % % % % % % % % % %5
\begin{lemma}\label{lem_smooth_mathcalW}
 For any $n\geq 1$, $\mathcal W_{n}\in C^1(\R_+)$. Further, if $W\in C^{p}(\R_+)$ for some $p\geq 2$, then  for any $n\geq 1$, $\mathcal W_{n}\in C^p(\R_+)$  and
  for all $0\leq l\leq p$, $K>0$, $\lambda>0$ there exists $0<C=C_{K,\lambda}(l)<\infty$ such that, for all $n\geq 1$,
 \begin{equation}\label{bound_mathcalW_pthderiv}
 \sup_{y\in  \left[K^{-1},K\right]} \left|  \mathcal W^{(l)}_{ n}(y)   \right| \leq C   n^{l+1}  (2-A(\lambda))^n.
 \end{equation}
\end{lemma}
\begin{proof}
Recalling that  $A(\lambda)= e^{-m \int_{\lambda}^\infty \frac{du}{\psi(u)}}$  and  $\lambda_n =m n$, we have, for any $n \geq 0$ and $\lambda>0$,  the identity
\begin{equation*}
\frac{d}{d\lambda}A^n(\lambda) = \frac{\lambda_n}{\psi(\lambda)} A^n(\lambda).
\end{equation*}
%Moreover, from the definition of $\mathcal{W}_q$ in \eqref{eq:co-eigen_series}, on gets that, for any
% $\lambda >0$,
%\begin{equation}\label{lapalce_mathcalW}
%\int_0^\infty e^{-\lambda y}\mathcal W_{\eign} (y) d y = A^{n}(\lambda)-1,
%\end{equation}
Then the Laplace transform inversion %in \eqref{lapalce_mathcalW}
combined  with \eqref{eq:def_W} and \eqref{lapl_mathcalWn} yield, for any $y>0$,
\begin{equation}\label{int_eq_mathcalW}
-y \mathcal W_{n} (y) =  \lambda_n \left( \int_0^y W(y-r) \mathcal W_{n} (r) d r + W(y) \right).
\end{equation}
Since $W \in C(\mathbb R^+)$, we get by an application of the dominated convergence theorem to \eqref{int_eq_mathcalW} that $\mathcal W_{n}\in C(\mathbb R^+)$. Moreover, since $W^{(1)}\in L^1(\R_+)$ by Lemma \ref{eq:propW}  and $\mathcal W_n\in L^1(\R_+)$ by \eqref{lapl_absmathcalWn},   Lemma \ref{lem_conv_diff_onederiv} in combination with \eqref{int_eq_mathcalW} yields that  $\mathcal W_{n}\in C^1(\R_+)$ with
\begin{equation}\label{int_eq_mathcalWprime}
\left(- y \mathcal W_{n} (y) \right)^{(1)} = \lambda_n \left( \int_0^y W^{(1)}(y-r) \mathcal W_{n} (r) d r + W^{(1)}(y)\right).
\end{equation}
Assume now that $W\in C^p(\R_+)$ for some $p\geq 2$. Then  by induction and Lemma \ref{lem_conv_diff_split}  combined with Leibniz's formula, we deduce that $\mathcal W_{\lambda_n}\in C^p(\R_+)$,  and, for any $n>0$,
\begin{equation}\label{pderiv_mathcalW}
\begin{split}
-\frac{y\mathcal W^{(p)}_{ n} (y) +p \mathcal W_{ n}^{(p-1)}(y)}{\lambda_n}  = & -\frac{\left(y \mathcal W_{ n} (y) \right)^{(p)}}{\lambda_n}   \\
 = &   \int_0^{\frac{y}{2}}   W^{(p)}(y-r)\mathcal W_{ n}(r)  d r + \int_0^{\frac{y}{2}}  \mathcal W_{ n}^{(p-1)} (y-r)  W'(r) d r  + W^{(p)}(y)   \\
  & + \tfrac 12\sum_{j=0}^{p-2} \left(  W^{(j+1)}(\tfrac{y}{2}) \mathcal W_{ n} (\tfrac{y}{2}) +    W'(\tfrac{y}{2})  \mathcal W_{ n}^{(j)} (\tfrac{y}{2})  \right)^{(p-2-j)}    \\
 = &  \int_0^{\frac{y}{2}}   W^{(p)}(y-r)\mathcal W_{ n}(r)  d r + \int_0^{\frac{y}{2}}  \mathcal W_{ n}^{(p-1)} (y-r)  W'(r) d r  + W^{(p)}(y)     \\
  + &   \tfrac 12\sum_{j=0}^{p_0} (\tfrac 12)^{p_j} \sum_{k=0}^{p_j} \binom{p_j}{k} \left(  W^{(p-1-k)}(\tfrac{y}{2}) \mathcal W_{ n}^{(k)} (\tfrac{y}{2}) +    W^{(1+k)}(\tfrac{y}{2})  \mathcal W_{ n}^{(p_k)} (\tfrac{y}{2})  \right),
\end{split}
\end{equation}
where for $k=0,1,\ldots,$ we have set $p_k=p-2-k$. Next, we prove the uniform bound \eqref{bound_mathcalW_pthderiv} by induction in $l$. First assume that $l=0$. Then  the identity \eqref{int_eq_mathcalW} combined with the fact that $W$ is non-negative  and increasing entail, writing for $j=1,2$, $I_j = [\tfrac{1}{jK},K]$ here and below, that
\begin{eqnarray}\label{bound_convW_0thderiv}
\sup_{y\in I_1} \left|  y \mathcal W_{ n}(y)  \right|
&\leq&  \lambda_n W(K)  \left( \int_0^{ K} |  \mathcal W_{ n}(r) | d r + 1 \right) \nonumber  \\
&\leq&  \lambda_n W(K) \left(  e^{\lambda K} \int_0^\infty e^{-\lambda r} |  \mathcal W_{ n}(r) | d r  + 1 \right) \\
&\leq & m W(K)     n e^{\lambda K} (2-A(\lambda))^n,\nonumber
\end{eqnarray}
where  we used the inequality \eqref{lapl_absmathcalWn} in the last line.
Hence \eqref{bound_mathcalW_pthderiv} holds for $l=0$. Now assume that $l=1$. Then starting from the relation \eqref{int_eq_mathcalWprime} and using similar arguments as for the previous case plus \eqref{bound_convW_0thderiv}, we get
 \begin{eqnarray*}%\label{bound_convW_0thderiv}
\sup_{y\in I_1 } \left|  y  \mathcal W^{(1)}_{ n}(y)  \right|
&\leq &  \sup_{y\in I_1}   \left| \mathcal W_{ n}(y) \right|  \\ &+& \lambda_n \sup_{y\in I_1}  \Bigg| \int_0^{\frac{y}{2}} W^{(1)}(y-r) \mathcal W_{ n} (r) +\mathcal W_{ n}(y-r)   W^{(1)}(r) d r  + W^{(1)}(y) \Bigg|  \\
&\leq &  C n  (2-A(\lambda))^n + \lambda_n \left(  e^{\lambda K} (2-A(\lambda))^n \sup_{y\in I_2} W^{(1)}(y)   +   W(K) \sup_{y\in I_2} \mathcal W_{ n}(y)   \right) \\
&\leq &   n (2-A(\lambda))^n  \left( C + m e^{\lambda K} \sup_{y\in I_2} W^{(1)}(y)  +  \lambda_n C  W(K)    \right),
\end{eqnarray*}
where $C>0$ is a generic constant.
Hence \eqref{bound_mathcalW_pthderiv} holds for $l=1$.
 Now assume that \eqref{bound_mathcalW_pthderiv} is true for $l=0,1,\ldots,k-1$ with $2\leq k\leq p$. Then  by \eqref{pderiv_mathcalW} and the induction hypothesis, we get, writing $k_j=k-2-j$, for  $j=0,1\ldots$,
\begin{eqnarray*} %\label{bound_convW_0thderiv}
\sup_{y\in I_1} \left| y \mathcal W_{ n}^{(k)}(y)  \right|
&\leq &  \sup_{y\in I_1}   \left| k\mathcal W_{ n}^{(k-1)}(y) \right|  + \lambda_n \left(\int_0^{\frac{K}{2}}  \left| \mathcal W_{ n}(r) \right| d r
\sup_{y\in I_2} \left|W^{(k)}(y)\right|   \right. \\
& + & \sup_{y\in I_2} \left| \mathcal W_{ n}^{(k-1)} (y)  \right| \int_0^{\frac{K}{2}} W^{(1)}(r)  dr  + \sup_{y\in I_1} \left|W^{(k)}(y)\right| \\
& + & \tfrac 12\sum_{j=0}^{k_0} (\tfrac 12)^{k_j} \sum_{i=0}^{k_j} \binom{k_j}{i}  \sup_{y\in I_2} \left| W^{(k_i)}(y) \mathcal W_{ n}^{(i)} (y) \right| \left. +  \sup_{y\in I_2} \left| W^{(1+i)}(y)  \mathcal W_{ n}^{(k_i)} (y) \right| \right)     \\
&\leq &   n (2-A(\lambda))^n \Bigg(  n^{k-1} k C(k) +  m\sup_{y\in I_2} \left|W^{(k)}(y)\right|  e^{\lambda K}     +  n^k m C(k) W(K) \\
& +&   mk_0!\sum_{j=0}^{k_0}  \sum_{i=0}^{k_j}   n^{i+1} C(i)  \sup_{y\in I_2} \left| W^{(k_i )}(y)\right| + n^{k_i+1} C(k_i)  \sup_{y\in I_2} \left| W^{(1+i)}(y)\right|   \Bigg),
\end{eqnarray*}
where the $C(k)$'s are generic positive constants and we have used that $(2-A(\lambda))^n$ is increasing in $n\geq1$   because $A(\lambda)\in(0,1)$ for $0<\lambda<\infty$. It follows that \eqref{bound_mathcalW_pthderiv} holds for $l=k$, which completes the proof.
\end{proof}

\subsubsection{Proof of  Proposition \ref{prop:nu_n}\eqref{it:s1}}
Assume that $\Si \in [1,\infty]$. Then $\biv=\infty$, see Lemma  \ref{elm:abs}  and so
\begin{equation*}
\nu_n(y) = \mathcal W_n*\nu(y) + \nu(y).
\end{equation*}
%Then, one the one hand $\nu_n(y) = \mathcal{W}_{\eign}\bar \star \nu (y)$, and, on the other hand,
By Proposition \ref{thm:regularity_nu}\eqref{it:rR},  $\nu \in C^{\Si-1}(\Bbb R)$. Hence,
  for all $0\leq \mathfrak{q}\leq \Si -1$, $\nu^{(\mathfrak{q})}(0) =0$ and since $\mathcal W_n$ is in $C(\R_+)$ by Lemma \ref{lem_smooth_mathcalW}, we have by using Lemma \ref{lem_conv_diff_onederiv} repeatedly,
%  % from Lemma \ref{prop_Wq}\eqref{it:sw},   $\omega^{*k}  \in C^{1}(\R_+) \cap L^1_{loc}(\R_0^+)$ for all $k\geq1$, Lemma \ref{lem_conv_diff_onederiv} yields that for all $k\geq 1$,
%  \begin{equation} \label{eq:nuqc} (\omega^{*k}* \nu)^{(\mathfrak{q})} = \omega^{*k}* \nu^{(\mathfrak{q})}\in C^{0}(\R_+). \end{equation}
%Moreover,  by Lemma \ref{prop_Wq}, the series $\mathcal{W}_{\eign}\bar \star \nu^{(\mathfrak{q})}(y)$ converges locally uniformly on $\R_+$ and one can easily  prove, by a simple induction argument, that, for $0 \leq \mathfrak{q} \leq \Si-1$ and $y\geq 0$,
\begin{eqnarray} \label{eq:dernu}
\nu_n^{(\mathfrak{q})}(y) =  \mathcal{W}_{n}  \star \nu^{(\mathfrak{q})}(y) +  \nu^{(\mathfrak{q})}(y),
\end{eqnarray}
and, in particular, 	$\nu_n\in C^{\Si-1}(\mathbb R)$. To prove the uniform bound, note that by \eqref{eq:dernu} and \eqref{lapl_absmathcalWn},
\begin{equation*}
\begin{split}
\vert \nu_n^{(\mathfrak{q})}(y) \vert \leq & \left( e^{\lambda y} \int_0^\infty e^{-\lambda r} |\mathcal W_n(r)| d r  +1 \right) \sup_{0\leq r\leq y }\vert\nu^{(\mathfrak{q})}(r)\vert \\
\leq & e^{\lambda y} (2-A(\lambda))^n \sup_{0\leq r\leq y }\vert\nu^{(\mathfrak{q})}(r)\vert.
\end{split}
\end{equation*}
Hence the result follows as $\nu^{(\mathfrak{q})} \in C(\R)$.
%
%which combines with \eqref{eq:nuqc} gives that $\nu_n^{(\mathfrak{\mathfrak{q}})} \in C(\R_+)$. To prove the uniform bound, one first observes, from \eqref{eq:dernu},  that, for any $\lambda >0$, $y\geq 0$ and  $0 \leq \mathfrak{q} \leq \Si-1$,		\begin{eqnarray*}
%			\vert \nu_n^{(\mathfrak{q})}(y)\vert  \leq \sum_{k=0}^\infty \vert	\omega^{*k}* \nu^{(\mathfrak{q})}(y)\vert \: \frac{\eign^k}{k!}
%					& \leq & e^{\lambda y} \sup_{0\leq r\leq y }\vert \nu^{(\mathfrak{q})}(r)\vert \sum_{k=0}^\infty \int_0^y e^{-\lambda r}	\omega^{*k}(r) dr \:  \frac{\eign^k}{k!} \\
%& \leq & e^{\lambda y} \sup_{0\leq r\leq y }\vert \nu^{(\mathfrak{q})}(r)\vert \sum_{k=0}^\infty  \:F^{k}(\lambda)   \frac{\eign^k}{k!}
%							\end{eqnarray*}
%where we used that $\omega^{*k}$ is clearly non-negative on $\R_+$. Since, for any $n\geq0$ and $\lambda>0$,
%\begin{equation}
%A(\lambda)^{-n} = e^{\eign  F(\lambda)}=\sum_{k=0}^\infty  \:F^{k}(\lambda)   \frac{\eign^k}{k!};
%\end{equation}
%we extract, for any $y\geq0$, the bound
%\begin{equation}\label{eq:bound_on_nuq}
%	\vert \nu_n^{(\mathfrak{q})}(y) \vert \leq e^{\lambda y} A(\lambda)^{-n} \sup_{0\leq r\leq y }\vert\nu^{(\mathfrak{q})}(r)\vert.
%\end{equation}
%Hence the result follows as $\nu^{(\mathfrak{q})} \in C(\R_+_0)$.

%\subsubsection{Proof of  Proposition \ref{prop:nu_n}\eqref{it:si}}

\subsubsection{Proof of Proposition \ref{prop:nu_n}\eqref{it:si1}}
%\noindent We now proceed with the proof of Proposition \ref{prop:nu_n}\eqref{it:si}.
Recall that, for any $n\geq 1$ and $y>0$,
\begin{equation}\label{conv_eq_nu_n}
\nu_n(y) = \int_0^y \mathcal W_{n}(y-r) \nu(r) d r + \nu(y) + e^{-\overline{\Phi}_{\nu}} \mathcal W_{ n}(y).
\end{equation}
Assume first that $1\leq \Si<\infty$. Then, from Lemma \ref{elm:abs}, $\overline{\Phi}_\nu=\infty$ and thus  by Proposition \ref{thm:regularity_nu}\eqref{it:rR} and \eqref{it:rp} and Lemma \ref{lem_conv_diff_onederiv}, $\nu_n\in C^{\Si}(\R_+)$ with, for any $y>0$,
\begin{equation}\label{conv_eq_nu_n_pthorder}
\nu_n^{(\Si)}(y) = \int_0^y \mathcal W_{n}(y-r) \nu^{(\Si)}(r) d r + \nu^{(\Si)}(y).
\end{equation}
Then,  we get that, for any $y>\epsilon>0$, $\lambda>0$ and $n\geq 1$,
\begin{eqnarray}
|\nu_n^{(\Si)}(y)| &\leq & \sup_{r \in [y-\epsilon,y]}|\mathcal W_{ n}(r)| \int_0^{\epsilon}  |\nu^{(\Si)}(r)| dr +    \sup_{r \in [\epsilon,y]} |\nu^{(\Si)}(r)| \int_\epsilon^y |\mathcal W_{ n}(r)| d r +  |\nu^{(\Si)}(y)|  \nonumber \\
&\leq &   C_\epsilon n (2-A(\lambda))^n + C e^{\lambda y}  (2-A(\lambda))^n %+ |\nu^{(\Si)}(y)| \label{eq:boundproof}
\end{eqnarray}
where for the second inequality we used  \eqref{bound_mathcalW_pthderiv} with $l=0$, \eqref{lapl_absmathcalWn} and  the fact that $\nu^{(\Si)} \in L^1(\R_+)$, which is given in Proposition \ref{thm:regularity_nu}. %Finally, we conclude by recalling that for any $\lambda>0$ and $n\geq 1$, $n A^{-n}(\lambda)\geq 1 $  and  $\nu^{(\Si)} \in C(\R_+)$, see Proposition \ref{thm:regularity_nu} again.
Hence  \eqref{bound_mathcalW_pthderiv2} follows for $\Si\geq 1$.

Second, if $\Si=0$  and $\overline{\kappa}(0^+)<\infty$, then  by Proposition \ref{thm:regularity_nu}\eqref{it:rp}, the fact that $\mathcal W_{n}$ is in $C(\R_+)$ and an application of the dominated convergence Theorem to \eqref{conv_eq_nu_n}, $\nu_n \in C(\R_+)$. The   bound  \eqref{bound_mathcalW_pthderiv2}  is derived by means of similar arguments as in the previous case  but using  the identity  \eqref{conv_eq_nu_n}. This completes the proof of  \eqref{it:si1}.

\subsubsection{Proof of Proposition \ref{prop:nu_n}\eqref{it:siq}}
 Assume $\kappa \in C^{1}(\R_+)$, $\kappa'\in L^1_{loc}(\R_+)$  and $\underline{\kappa}(0^+)=\overline{\kappa}(0^+)$ and let $\bar{\mathfrak{q}}\geq 1$ be as in Theorem \ref{thm:regularity}\eqref{it:defq}.
    Then invoking Proposition \ref{thm:regularity_nu}\eqref{it:r2} and Lemma \ref{lem_smooth_mathcalW} and by applying Lemma \ref{lem_conv_diff_split} to the identity \eqref{conv_eq_nu_n} in the case where  $\underline{\kappa}=0$ and to \eqref{conv_eq_nu_n_pthorder} in the case where  $\underline{\kappa}\geq 1$, we get that $\nu_n\in C^{\Si+\bar{\mathfrak{q}}}(\R_+)$ with, for any $y>0$,
\begin{eqnarray*}
\nu^{(\Si+\bar{\mathfrak{q}})}_n (y) &= & \int_0^{\frac{y}{2}}   \mathcal W^{(\bar{\mathfrak{q}})}_{ n}(y-r) \nu^{(\Si)}(r) d r + \int_0^{\frac{y}{2}} \nu^{(\Si+\bar{\mathfrak{q}})}(y-r) \mathcal W_{ n}(r) d r +  \nu^{(\Si+\bar{\mathfrak{q}})}(y)  \\
& +&  e^{-\overline{\Phi}_{\nu}} \mathcal W^{(q)}_{ n}(y)+ \tfrac 12\sum_{j=0}^{\bar{\mathfrak{q}}-1} \left(  \nu^{(\Si+j)}(\tfrac{y}{2}) \mathcal W_{ n} (\tfrac{y}{2}) +    \nu^{(\Si)}(\tfrac{y}{2})  \mathcal W_{ n}^{(j)} (\tfrac{y}{2})  \right)^{(\bar{\mathfrak{q}}-1-j)}.
\end{eqnarray*}
The bound \eqref{bound_mathcalW_pthderiv3}  is obtained  from this identity in the same way as the previous bound, the tedious details being left to the reader. This completes the proof of Proposition \ref{prop:nu_n}.

We end this part  with the following results which provide  additional but more specific properties regarding the set of eigenmeasures and complement the results on the set of eigenfunctions stated in Proposition \ref{prop:eigen}. Notions introduced below are classical and their definitions can be found for instance in the textbooks \cite{Dunford_II} and \cite{Young}.
\begin{prop}
Let $(\psi, \phi)\in {\Ne}\times \Be$. Then, for all $n,m \in \mathbb{N}$, we have
\begin{equation} \label{eq:bio}  \mathcal V_m \mathcal{L}_n  = {\delta}_{n,m}\end{equation}
 where $\delta_{n,m}$ is the Kronecker symbol. Assume now that $\phi\neq 0$. Moreover, for all $n\geq 0$, the measure ${\mathcal{V}}_n$ is absolutely continuous with respect to $\mathcal{V}$ and we write  $\overline{\mathcal{V}}_n$ for the corresponding Radon-Nykodim derivative. %If for some $n\in \mathbb{N}$, $\overline{\mathcal{V}}_n \notin L^{2}(\mathcal{V})$ then \[ e^{-\lambda_n t} \subseteq {\rm S}_r(P^*_t)=\{\lambda \in \mathbb{C};\: P^*_t - \lambda I \textrm{ is one-to-one and its range is not dense in } L^{2}(\mathcal{V})\}, \] the residual spectrum of $P^*_t$,
 If for some $n\in \mathbb{N}$, $\overline{\mathcal{V}}_n \in L^{2}(\mathcal{V})$ then $e^{-\lambda_n t} \in {\rm S}_p(P^*_t)$, where $P^*_t$ is the adjoint of $P_t$ in $L^{2}(\mathcal{V})$. Moreover if $(\overline{\mathcal{V}}_n)_{n\geq0} \in L^{2}(\mathcal{V})$   then the geometric and algebraic multiplicity of $e^{-\lambda_n t}$ is $1$ for all $n\geq 0$. Finally, in this case, $(\mathcal{L}_n,\overline{\mathcal{V}}_n)_{n\geq0}$  is a biorthogonal sequence in $L^{2}(\mathcal{V})$.
\end{prop}
\begin{remark}
Note that showing that  $\overline{\mathcal{V}}_n \in L^{2}(\mathcal{V})$ for some $n\in \mathbb{N}$ seems to be a difficult problem as one has to get precise asymptotic estimate for small and large values of the argument of the functions  $\nu_n $ and $\nu$. However, in the example \ref{example_specrep}, one easily gets  that  $\overline{\mathcal{V}}_n(y)= \sum_{j=0}^n \binom n j (-1)^j \frac{y^{\alpha j}}{\Gamma(\alpha j +1)} \in L^{2}(\mathcal{V})$ where we recall that in this case $\mathcal{V}(dy)=e^{-y}dy$.
\end{remark}
\begin{proof}
%Since, for any $n \in \N$, $|\mathcal{L}_n(x)|\stackrel{\infty}{=} {\rm{O}}\left(e^{-\lambda x}\right)$ for any $-R_A<\lambda<0$, we deduce,  from \eqref{eq:eigenmeasure_P} combined with \eqref{eq:pol_eig}, that, for all $n,m \in \mathbb{N}$  and $t$ large enough, \begin{equation}
%	e^{-\lambda_nt} \mathcal V_m \mathcal{L}_n = e^{-\lambda_mt}\mathcal V_m \mathcal{L}_n.
	%\end{equation}
%	Since for any $n\neq m$, $\lambda_n \neq \lambda_m$ we get, in this case, that $\mathcal V_m \mathcal{L}_n =0$. Otherwise,
 First, observe, from \eqref{eq:pol_cont},  that, for all $n,m \in \mathbb{N}$, \[ \mathcal V_m  \mathcal{L}_n = \int_0^{\infty}\frac{1}{2 \pi i}\oint \frac{\Gx(z)}{z^{n+1}}dz\mathcal V_m(dx).\] Next, using \eqref{eq:est_cont}, we have, choosing $R$ such that $0<\overline{B}<R_A$ with $\overline B=\max_{|z|= R}|B(z)|$,
\[ \int_0^{\infty}\left|\oint \frac{\Gx(z)}{z^{n+1}}dz \right| |{\mathcal V}_m|(dx) \leq C_R \int_0^{\infty} e^{\overline{B}x} |{\mathcal V}_m|(dx) <\infty \]
where $C_R>0$ and the last inequality follows from \eqref{bound_lapl_abs_Vn}. Thus,  an application of Fubini Theorem yields
\begin{eqnarray*}
\mathcal V_m  \mathcal{L}_n &=&  \frac{1}{2 \pi i}\oint e^{\piv(B(z))} \mathcal V_m e_{B(z)}\frac{dz}{z^{n+1}} \\
&=& \frac{1}{2 \pi i}\oint  \frac{1}{z^{n-m+1}} dz ={\delta}_{n,m}.
\end{eqnarray*}
Next, let $f \in L^{2}(\mathcal{V})$ then, for any $n\in \mathbb{N}$,  using \eqref{eq:eigenmeasure_P},
\[ \langle  f , P^*_t\overline{\mathcal{V}}_n\rangle_{\mathcal{V}}= \langle P_t f , \overline{\mathcal{V}}_n\rangle_{\mathcal{V}} = e^{-nt}\langle  f , \overline{\mathcal{V}}_n\rangle_{\mathcal{V}}, \]
which shows that $e^{-nt}\in {\rm S}_p(P^*_t)$. The multiplicity of the eigenvalues is proved in \cite[Proposition 2.27]{PS14} whereas the last statement follows from \eqref{eq:bio}.
%As $\mathcal L_n(x)$ are the coefficients in the Taylor series at $0$ of $z\mapsto G_x(z)$, we have by the identity \eqref{eq:lt_nu_n_proof} below,
%\begin{equation*}
%\begin{split}
%\mathcal V_m \mathcal{L}_n = \int_0^\infty \mathcal L_n(x) \mathcal V_m(d x) = & \int_0^\infty \left. \frac{1}{n!}  \frac{d^n}{d v^n}\frac{1}{e^{-\Phi_\nu(B(v))}}e^{-x B(v)} \right|_{v=0} \mathcal V_m(d x) \\
%= & \left. \frac{1}{n!} \frac{d^n}{d v^n} \frac{1}{e^{-\Phi_\nu(B(v))}}\int_0^\infty e^{-x B(v)} \mathcal V_m(d x) \right|_{v=0} \\
%= & \left. \frac{1}{n!} \frac{d^n}{d v^n} \left( A(B(v)) \right)^m \right|_{v=0} \\
%= & \left. \frac{1}{n!} \frac{d^n}{d v^n} v^m \right|_{v=0} \\
%= & \mathbf 1_{\{n=m\}}.
%\end{split}
%\end{equation*}
% Note that the switching of the derivatives and the integral in the second equality can be justified by the dominated convergence theorem because  $A$ and $\Phi_\nu$ are analytic at zero and therefore   $\int_0^\infty e^{\lambda x}\omega( x)d x<\infty$, $\int_0^\infty e^{\lambda x} \nu(x)d x<\infty$ and consequently $\int_0^\infty e^{\lambda x}|\mathcal V_m|(d x)<\infty$ for some $\lambda>0$, where $|\mathcal V_m|$ is the total variation measure of $\mathcal V_m$. 	
%	\added[id=RL]{Since we use here an identity which is proved later on, it might be better to switch the order of Sections 4 and 5.}
	
\end{proof}

\section{Proof of Theorem \ref{thm:spectral-expansion}, Proposition \ref{cor:exp-ergodicity}  and Theorem \ref{thm:regularity}} \label{sec:proof_mr}
\subsection{Proof of Theorem \ref{thm:spectral-expansion} and Theorem \ref{thm:regularity}\eqref{it:Cinfty}, \eqref{it:CinftyC0} and \eqref{it:sh}}\label{sec:proof_one}

The proof %of Theorem \ref{thm:spectral-expansion}
is split into several steps. We first establish that  the CBI-semigroup $P$ coincide  on appropriate linear spaces with the following two  (linear spectral) operators defined on $B_b(\Rpo)$, for any  $x\geq 0$, by
\begin{equation}\label{eq:def-SandbarS}
S_tf(x) = \sum_{n=0}^\infty e^{-\eign  t} \mathcal{V}_nf \: {\mathcal L}_n(x)\quad \mbox{ and } \quad \overline{S_t}f (x) = \int_0^{\infty} f(y) S_t\delta_y(x),
\end{equation}
where we use the notation
$
S_t\delta_y(x) = \sum_{n=0}^\infty e^{-\eign  t} \mathcal{V}_n(dy) \: {\mathcal L}_n(x)
$. First we show that $S_tf=P_tf$ on $\Lambda$ and then that $\overline{S_t}f= S_tf$ on $\mathcal{D}_t$.

\begin{lemma}\label{lem:StequalsPt} For all  $t> T_0$, $ x \geq 0$ and  $f\in \Lambda$, we have
$S_tf(x) = P_tf(x)$. Consequently, $(S_t)_{t > T_0}$ is a densely defined continuous semigroup on $C_0(\Rpo) $, endowed with the uniform topology $\Vert .\Vert_{\infty}$, i.e.~for any $f\in \Lambda$,
$\Vert S_tf\Vert_{\infty} \leq   \Vert f\Vert_{\infty}$, with $P_t$   its unique contraction semigroup extension on the closure of $\Lambda$, that is on $\overline{\Lambda}=C_0(\Rpo) $.
\end{lemma}
\begin{proof}
Using successively the identities
\eqref{eq:lt_nu_n_proof}, \eqref{eq:def-Ln_statement} and  \eqref{eq:lt_cbi},  we have, for any $\lambda>0$,
\begin{eqnarray*}
S_t e_{\lambda}	(x)  &=&  \sum_{n=0}^\infty e^{- \eign t}\mathcal V_n e_\lambda {\mathcal L}_n(x)   =
\sum_{n=0}^\infty e^{- \eign t} e^{-\piv(\lambda)} A(\lambda)^n{\mathcal L}_n(x)  \\
&=&  e^{-\piv(\lambda)}\Gx(A(\lambda)e^{-mt}) \\ &=& P_te_\lambda (x)
\end{eqnarray*}
where we note that \eqref{eq:def-Ln_statement} is valid since $A(\lambda) \in (0,1)$ on $\R_+ $ and hence  $0<A(\lambda)e^{- mt} < e^{- m T_0} = R_0$. This completes the proof as plainly  $S_t$ is linear on $\Lambda$ and $P$ is a Feller semigroup on $\Rpo$.
\end{proof}

\begin{lemma} \label{lem:StequalsbarSt}  For all  $t> T_0$, $ x \geq 0$ and  $f\in \mathcal{D}_t$, we have $\overline{S_t}f (x) = S_tf(x)$. Note that $\Lambda_{t}={\rm Span}(e_{\lambda},\lambda > -\bar \lambda_t) \subset \mathcal{D}_t$, where $\bar \lambda_t$ is defined in \eqref{eq:defBtbar_and_H}.
\end{lemma}

\begin{proof}
Let  $f\in \mathcal{D}_t$, which by  definition  in \eqref{eq:defBtbar_and_Ht},   implies that there exists $C >0$, chosen here for sake of clarity greater than $f(0)$,   and $\lambda<\bar{\lambda}_t$ such that $\vert f(y) \vert \leq Ce^{\lambda y}$ for a.e.~$y\geq 0$. %where we simply write here $\lambda = -\bar \lambda_t$.
Thus,  we get
 that
$
 		\vert \mathcal{V}_n \vert \vert  f \vert \leq   C \vert \mathcal{V}_n \vert e_{-\lambda},
$ where  we used the fact that $\mathcal{V}_n$ is absolutely continuous on $\R^+$.
From \eqref{bound_lapl_abs_Vn} and the definition of $\bar{\lambda}_t$ we get  that
%\begin{equation}\label{bound_lapl_abs_Vn}
$\vert \mathcal{V}_n \vert  e_{-\lambda}
\leq e^{-\piv(-\lambda)}(2-A(-\lambda))^{n}.$
%\end{equation}
Hence $ e^{-\eign t} \vert\mathcal{V}_n \vert\vert f \vert
\leq   C	e^{-\piv(-\lambda)} \left(e^{-mt}(2-A(-\lambda)) \right)^n$.
From the definition again of $\bar\lambda_t$ and since from Lemma \ref{Analicity_of_A_B_F} $A$ is increasing on $(-R_A,\infty)$, we   have in both cases $e^{-mt}(2-A(-\lambda))<e^{-mt}(2-A(-\bar\lambda_t))\leq  e^{-m T_0} = R_0
$. Thus, we obtain that there exists $R\in (0,R_0)$ such that,
 for any   $n\geq 0$,
 \begin{equation}
	\label{eq:bound_Nunf}
	e^{-\eign t} \vert  \mathcal{V}_n\vert  \vert  f \vert \leq C R^n.
	\end{equation}
Hence, the representation \eqref{eq:uniform_convergence} for $\mathfrak{p}=0$ in Proposition \ref{prop:Analicity_of_G} yields
\[
\int_{0}^{\infty}\vert f(y) \vert \sum_{n=0}^\infty e^{-\eign t}\vert \mathcal{V}_n\vert (dy) \: \vert \mathcal{L}_n(x) \vert
= \sum_{n=0}^\infty e^{-\eign t}\vert \mathcal{V}_n\vert \vert f \vert \: \vert \mathcal{L}_n(x) \vert  <\infty,
\]
 where we recall that the series are locally uniformly convergent.
By Fubini Theorem, this shows that the linear operator $\overline{S_t}  $ is well defined and satisfies $ \overline{S_t}f = S_tf$ on $\mathcal{D}_t$. The statement  $\Lambda_{t} \subset \mathcal{D}_t$ is obvious.
\end{proof}

\noindent From Lemma \ref{lem:StequalsPt} and  Lemma \ref{lem:StequalsbarSt}, we have gained that, for any $t>T_0$,  $P_t$ and $\overline{S_t}$ share the same Laplace transform on $(-\bar \lambda_t, \infty)$. This is sufficient to claim, by injectivity of the Laplace transform, that for all $t>T_0$, $x\geq 0$,
\begin{equation}\label{eq:c13}
P_t(x,dy)= S_t\delta_y(x) = \sum_{n=0}^\infty e^{-\eign  t}  \: {\mathcal L}_n(x) \: \mathcal{V}_n(dy)
\end{equation}
and to  get that
\begin{equation}\label{condforrudin}
 P_tf(x)=S_tf(x)= \sum_{n=0}^{\infty}e^{-\eign  t} \mathcal{V}_n f \: {\mathcal L}_n(x)  \mbox{ on } \mathcal{D}_t\cup \Lambda.
 \end{equation}
Next, for any $f \in \mathcal{D}_t\cup \Lambda$, $x\geq0$ and $t>T_0$,   and integers  $\mathfrak{m},\mathfrak{p}$,
  the equality \eqref{eq:lt_nu_n_proof} if $f \in \Lambda$ or the bound \eqref{eq:bound_Nunf} if $f \in \mathcal{D}_t$ yield that in both cases there exists  $R\in (0, R_0)$ such that for  $n$ large enough $|(-\eign )^{\mathfrak{m}} e^{-\eign  t} \mathcal{V}_n f | \leq R^n$. Thus, there exists $C>0$ such that
\begin{eqnarray*}
\sum_{n=\mathfrak{p}}^{\infty}\left|(-\eign )^{\mathfrak{m}} e^{-\eign  t} \mathcal{V}_n f \: {\mathcal L}_n^{(\mathfrak{p})}(x)\right| &\leq & C \sum_{n=\mathfrak{p}}^{\infty} R^n \: \left|{\mathcal L}_n^{(\mathfrak{p})}(x)\right|.
\end{eqnarray*}
Thus from  \eqref{eq:uniform_convergence} in Proposition \ref{prop:Analicity_of_G}, the series on the right-hand side is locally uniformly convergent in $(t,x)$. Since we already showed that \eqref{condforrudin} holds,  this proves by a classical result in analysis, see e.g. Theorem 7.17 in \cite{Rudin1976}, the claim \eqref{eq:derivatives_of_the_semigroup} in Theorem \ref{thm:spectral-expansion} and shows that $(t,x) \mapsto P_t f(x) \in C^{\infty^2}((T_0,\infty)\times \Rpo)$ for such $f$. Theorem \ref{thm:regularity}\eqref{it:Cinfty} (resp.~Theorem \ref{thm:regularity}\eqref{it:CinftyC0}) is an immediate consequence of this property combined with the fact $P$ is a Markov operator (resp.~that $\mathcal{D}_t\cap C_0(\Rpo) \subset C_0(\Rpo)$ and $P$ is a Feller semigroup).
Recalling that $\mathcal V_n (dy) = e^{-\biv} \delta_0(dy) + \nu_n(y)dy$, the claim \eqref{eq:transition_kernel} is in fact the identity \eqref{eq:c13}. Consequently,  for any $x\geq0$ and $t>T_0\geq0$, $P_t(x,dy)$ is absolutely continuous if and only if
 $e^{-\biv }e^{\piv(B(e^{-m t}))}e^{-xB(e^{-m t})}=0$ which is equivalent to $\biv = \infty$, as $B(e^{-m t}) =\infty$ if and only if $t=0$, which is impossible. The claims \eqref{eq:deck}
 and \eqref{absolute_continuity_condition0} follow.

 Next, for any $t>T_0$, $x \in [0,\rm{K}], \rm{K}>0$, $\lambda<\min\{\bar{\lambda}_t,0\}$ and $y \in [K^{-1},K], K>0$ and for any integers $\mathfrak{m},\mathfrak{p}$, and, any integer $\mathfrak{q}$ such  that $\mathfrak{q}\leq \Si -1$ if $\Si\geq 1$, or $\mathfrak{q}\leq \Si +\bar{\mathfrak{q}}$ with  ${\bar{\mathfrak{q}}}$ as in Theorem \ref{thm:regularity}\eqref{it:defq}, we get
 \begin{eqnarray*}
\sum_{n=\mathfrak{p}}^{\infty}\left|(-\eign )^{\mathfrak{m}} e^{-\eign  t} \nu^{(\mathfrak{q})}_n(y) \: {\mathcal L}_n^{(\mathfrak{p})}(x)\right| &\leq & C \sum_{n=\mathfrak{p}}^{\infty} \eign ^{\mathfrak{m}} e^{-\eign  t} n^a (2-A(-\lambda))^n\: \left|{\mathcal L}_n^{(\mathfrak{p})}(x)\right|
\end{eqnarray*}
where $C>0$  and to estimate $|\nu^{(\mathfrak{q})}_n(y)|$ we used   the  bounds of Proposition \ref{prop:nu_n} with $a=\max(\mathfrak{q}+1-\Si,0).$ %either $a=0$ or $a=\mathfrak{q}+1-\Si$ depending on the possible values of $\mathfrak{q}$.
We conclude that the series is locally uniformly convergent in $(t,x,y)$ by invoking again \eqref{eq:uniform_convergence}  after noting that  $\eign ^{\mathfrak{m}} e^{-\eign  t} n^a (2-A(- {\lambda} ))^n<R^n$ for some $R<R_0$ and all $n$ large enough. Combined with \eqref{eq:c13}, this provides (via  Theorem 7.17 in \cite{Rudin1976}) the expression \eqref{eq:derivatives_of_the_density_in_y} and the proof of Theorem \ref{thm:regularity}\eqref{it:sh}\eqref{it:4a}-\eqref{it:defq}. This completes the proof of Theorem   \ref{thm:spectral-expansion} and Theorem \ref{thm:regularity} \eqref{it:Cinfty}, \eqref{it:CinftyC0} and \eqref{it:sh}.

\subsection{Proof of Proposition  \ref{cor:exp-ergodicity}}
Note that the first claim regarding the existence of an invariant probability measure was proved in Proposition \ref{prop:invariant_measure}.
Next, let $t> \underline{T}=T_0 + \frac{1}{m}\ln (2-A(-R_A))$ and $E\in {\mathcal B}(\Rpo)$. Then, from the  definition \eqref{eq:defBtbar_and_H},  we have   ${\overline \lambda}_t = R_A\wedge R_\phi$  and $\mathbb I_{E} \in {\mathcal D}_t$ as clearly $\mathbb I_{E} \leq e_{-\lambda} $  for any $\lambda \in (0, R_A\wedge R_\phi)$. Therefore, noting that $\lambda_0=0$, ${\mathcal L}_0(x) =1 $ and ${\mathcal V}_0 = {\mathcal V}$, we get, from \eqref{eq:derivatives_of_the_semigroup} in Theorem \ref{thm:spectral-expansion}, that, for any $x\geq0$,
\begin{eqnarray*}
\left \vert P_t(x,E)-{\mathcal V}(E) \right \vert
= \left \vert \sum_{n=1}^\infty e^{- n m  t } {\mathcal L}_n(x) {\mathcal V}_n(E)\right \vert
\leq    \sum_{n=1}^\infty  e^{-nm t } \left \vert  {\mathcal L}_n(x)\right \vert \left \vert  {\mathcal V}_n(E)\right \vert.
\end{eqnarray*}
Now since ${\mathbb I}_{E} \leq e_{-\lambda}$, we have
$
\left \vert  {\mathcal V}_n(E)\right \vert
\leq \left \vert  {\mathcal V}_n\right \vert (E) \leq
\left \vert  {\mathcal V}_n\right \vert e_{-\lambda} \leq e^{-\piv(-\lambda)}(2-A(-\lambda))^n
$
where the last inequality follows from \eqref{bound_lapl_abs_Vn} in Proposition \ref{prop:nu_n} which is valid since   $-\lambda > -R_A\wedge R_\phi$.  Thus
\begin{eqnarray*}
\left \vert P_t(x,E)-{\mathcal V}(E) \right \vert
\leq  e^{-\piv(-\lambda)}  \sum_{n=1}^\infty \left \vert  {\mathcal L}_n(x)\right \vert \left(e^{-mt}(2-A(-\lambda))\right)^n.
\end{eqnarray*}
Now,  from the fact that $A$ is increasing on $(-R_A, \infty)$ one sees that the choices of $t > \underline{T}= T_0 + \frac{1}{m}\ln(2-A(-R_A))$ and of $\lambda < R_A\wedge R_\phi$ ensure that  $0<R=e^{-m \underline{T}}(2-A(-\lambda)) = e^{-m T_0}\frac{2-A(-\lambda)}{2-A(-R_A)} < e^{-m T_0}=R_0$. Thus by \eqref{eq:bpol} in Proposition \ref{prop:Analicity_of_G}, there exist $C=C(R)>0$ and $\overline{B}=\max_{|z|= R}|B(z)|\in \R_+$ such that
\begin{eqnarray*}
\left \vert P_t(x,E)-{\mathcal V}(E) \right \vert
& \leq & e^{-\piv(-\lambda)}  Ce^{\overline{B}x}
\sum_{n=1}^\infty \left(\frac{ e^{-mt}(2-A(-\lambda))}{R}\right)^n \\
& = &  e^{-\piv(-\lambda)} Ce^{\overline{B}x}
\sum_{n=1}^\infty e^{-n m(t-\underline{T})},
\end{eqnarray*}
which plainly proves the Corollary.
\subsection{Proof of Theorem \ref{thm:regularity}\eqref{it:sf}}
%We end this section by showing
%\begin{cor} Let $P$ be a CBI$( \psi,   \phi)$ semigroup with $(\psi,\phi) \in \Ne \times \Be$.  Then $P$ is (eventually) strong Feller in the sense that,   for any $t>T_0$, $ P_tf \in C_b(\Rpo)$ for all $f \in B_b(\Rpo)$,
%\end{cor}

 %\begin{proof}
 Let us now prove  that $P$ is eventually strong Feller. We mention that  Schilling and Wang in \cite{Schilling-Wang} provide interesting sufficient  conditions on the transition kernel of Feller semigroups which imply  the strong Feller property. However, we have  not been able to use their criteria  and instead we  prove this property directly using Theorem \ref{thm:regularity}\eqref{it:Cinfty}. To this end, let $t >T_0$ and $f\in B_b(\Rpo)$. There exists a non-decreasing sequence $(g_n)_{n\geq 0}$ of non-negative functions in $C_{c}(\Rpo)$
converging pointwise to $g \equiv 1$ and by the monotone convergence Theorem we have $\lim_{n \to \infty} P_tg_n = P_tg$.  As $g_n \in C_{c}(\Rpo)$, the space of continuous functions with compact support, by Theorem \ref{thm:regularity}\eqref{it:Cinfty}, we have $P_tg_n \in C(\Rpo)$. Moreover, by \eqref{eq:CBI-initial}, $P_t g \in C(\Rpo)$. Then Dini  Theorem shows that the convergence is uniform on any compact set in $\Rpo$. Now let $x \in \Rpo$ and $\epsilon >0$. Consider a compact set $K$ of the form $[0, \eta]$ if $x =0$,  $[x-\eta, x+\eta]$ otherwise, then there exists $n_0=n_0(\epsilon) \geq 0$  such that
\[
\sup _{x\in K}\vert  P_t  (g-g_{n_0}) (x)  \vert < \frac{\epsilon}{3\Vert f \Vert_\infty}. \]
Next, since $f g_{n_0} \in B_{b}(\Rpo)$ with compact support,  Theorem \ref{thm:regularity}\eqref{it:Cinfty} entails  that $P_tf g_{n_0} \in C^\infty(\Rpo)$. Therefore,  there exists  $\eta_0=\eta_0(n_0,\epsilon) \in(0,\eta)$ such that, for all $ y\in  U_{0}  = (0, \eta_0)$ if $x =0$ and  $y \in U_{x} =(x-\eta_0, x+\eta_0)$ otherwise, \[\vert P_tfg_{n_0}(x) -  P_tfg_{n_0}(y)\vert < \frac{\epsilon}{3}.\]
Since  $U_{x} \subset K$, we get, using the previous estimates, that for all $y \in U_{x}$,
\begin{eqnarray*}
\vert P_tf(x) -P_tf(y) \vert
  &\leq &
 \vert  P_t f (g-g_{n_0}) (x) \vert + \vert  P_t f (g-g_{n_0}) (y) \vert + \vert P_tfg_{n_0}(x) -  P_tfg_{n_0}(y)\vert \\ &\leq &
 \vert  || f ||_{\infty}\left(\vert P_t  (g-g_{n_0}) (x) \vert + \vert  P_t  (g-g_{n_0}) (y) \vert \right) + \vert P_tfg_{n_0}(x) -  P_tfg_{n_0}(y)\vert \\
 & < & \epsilon.
\end{eqnarray*}
Hence,  from the contraction property of $P_t$, we conclude that for any $t>T_0$ and  $f \in B_b(\Rpo)$,  $P_tf \in C_b(\Rpo)$ which is the (eventually) strong Feller property.
%\end{proof}

\section{Examples} \label{sec:ex}
We end this paper by detailing some examples of CBI-semigroups which illustrate the variety of smoothness properties that the absolutely continuous part of their transition kernel enjoy. The last example also reveals that our results are sharp in the sense that some instances of CBI-semigroups do not have a better regularity property on $\R_+$ than the one stated in Theorem \ref{thm:regularity}. %We remark that  we were not able to specify the value of $T_0$ for the examples given below, which seems inherently difficult to do, see Remark \ref{remark_determ_T0}. %In order to get an idea of the value of $T_0=-\ln(R_0)/m$ in specific examples, one can instead proceed by numerically computing $R_0$, the radius of convergence of the power series \eqref{eq:def-Ln_statement}. Note that in \cite{CLP-F} we have given an algorithm   for computing the polynomial eigenfunctions $\mathcal L_n(x)$.

 \subsection{Handa  CBI-semigroups}
In \cite{CBI-Handa},  Handa showed that that   every generalized
gamma convolution distribution is an invariant measure for a CBI-semigroup and then he studied the sector property of this class of CBI-semigroups and we refer to the aforementioned paper for definition. More specifically, let us consider the mechanisms
\begin{equation}\label{eq:mechanisms-Handa}
\psi(u) = \sigma^2 u^2 +mu +\int_0^\infty\left(e^{-ur}-1+ur\right)\Pi(dr)  \quad \mbox{and } \quad \phi(u) =u
 \end{equation}
 where $\sigma^2\geq0$, and
  \begin{equation}\label{eq:def_lm}
  \Pi(dr) = \int_0^\infty v^2e^{-vr}M(dv)  dr,
   \end{equation}
   for some measure $M$ on $\R_+$ such that $\int_0^\infty \frac{M(dv)}{1+v} <\infty$ and which is associated to   a  Thorin measure $\tho$
by the following relation \[
 \int_0^\infty\frac{\tho(dv)}{u+v} = \frac{1}{\frac{u}{\overline{\tho}_0} + \frac{1}{\overline{\tho}_{1}} + \int_0^\infty\frac{u}{u+v}M(dv)}
\]
 where for $k=0,1$, $\overline{\tho}_{k} = \int_0^\infty v^{-k} \tho(dv)$. We recall that a Thorin measure is a measure $\tau$ on $\R_+$ satisfying
 \begin{equation}
 \int_0^{\frac12} \vert \log v\vert \tho(dv)+ \int_{\frac12}^{\infty}\frac{\tho(dv)}{v}<\infty.
 \end{equation}
 Then according to \cite[Theorem 4.3]{CBI-Handa}, we have that the L\'evy measure of $\Phi_{\nu}$, the Laplace exponent of the invariant measure, is $\frac{\kappa(r)}{r}dr,r>0,$ where
 \begin{equation}\label{eq:kappa-Handa}
 \kappa(r)= \int_0^\infty e^{-vr}\tho(dv)=W^{(1)}(r)
 \end{equation}
 is completely monotone and hence the invariant measure is a generalized gamma convolution distribution.
 Note that \eqref{eq:kappa-Handa}  combined with \eqref{eq:pot} yields that
 \[
  \phi_p(u)=\frac{u}{\overline{\tho}_0} + \frac{1}{\overline{\tho}_{-1}} + \int_0^\infty\frac{u}{u+v}M(dv).
\]
Thus, if $M\equiv 0$ on $[0,R)$ for some $R>0$ then it is easy to check  that    $\phi_p \in \mathcal{H}(R)$ and also $\psi \in \mathcal{H}(R)$. Moreover, the condition \eqref{eq:standing_assumption} is according to Lemma \ref{lem:mainco} satisfied if $\sigma^2>0$, i.e.~$\overline{\tho}_0<\infty$, or,  for instance if there exists $g$ positive and non-increasing such that  $\frac{M(dv)}{dv}=g(v)\stackrel{\infty}{\sim} C v^{-\alpha}$, $\alpha \in (1,2)$, as by classical arguments, $\overline{\overline{\Pi}}(r) \stackrel{0}{\sim} C_{\alpha} r^{1-\alpha}$.
From \eqref{eq:kappa-Handa} we see that, as the Laplace transform of a measure on $\R_+$, $\kappa$ is completely monotone. Hence  $\kappa \in C^\infty(\R_+)$ and
\begin{equation}\label{eq:kappa0-Handa}
\underline{\kappa}(0^+)=\overline{\kappa}(0^+) =\tho(\R_+) = W^{(1)}(0^+)=\frac{1}{\sigma^2}  \in (0,\infty],
\end{equation}
where for the last equality we have used Proposition \ref{lem:kap}\eqref{item:lem_kap_sigma}.
Thus, we have that  $\underline{\kappa}(0^+)>0$, which by Lemma \ref{elm:abs} implies that $P_t(x,dy)= p_t(x,y)dy$.
First, if $\tho(\R_+) =\infty$ then, by \eqref{eq:kappa0-Handa},  $\Si =\infty$ and from  Theorem \ref{thm:regularity} \eqref{it:4a}, we have that $(t,x,y) \mapsto p_t(x,y) \in C^{\infty^3}((T_0, \infty)\times\Rpo\times \R)$.
Assume now that $\tho(\R_+) <\infty$. Then, by \eqref{eq:kappa-Handa},   for any $a>0$,
 \[
   \int_0^a \vert \kappa^{(1)}(r)\vert  dr
  =  \int_0^a \int_0^\infty  ve^{-vr}\tho(dv)dr
  = \int_0^\infty (1-e^{-av})\tho(dv) < \tho(\R_+) <\infty,
  \]
  so that $\kappa^{(1)} \in L^1_{loc}(\Rpo)$. Moreover, since by \eqref{eq:kappa-Handa}  $\kappa, W \in C^\infty(\R_+)$, we have that $\bar{\mathfrak{q}}=\infty$ in Theorem \ref{thm:regularity} \eqref{it:defq} and hence $(t,x,y) \mapsto p_t(x,y) \in C^{\infty^3}((T_0, \infty)\times\Rpo\times \R_+)$.

\subsection{Tempered stable}
Let for any $\alpha\in(1,2]$ and $\beta\in(-1,1]$,
\begin{equation*}
\psi(u) =  a \left( (u+\eta_\psi)^\alpha + \gamma u -\eta_\psi^\alpha   \right) \textrm{ and }
\phi( u) = c a \left|( u+\eta_\phi)^\beta    -\eta_\phi^\beta   \right|,
\end{equation*}
where  $a,c,\eta_\psi,\eta_\phi>0$ and $\gamma>-\alpha\eta_\psi^{\alpha-1}$. Then $\psi$ and $\phi$ are of the form \eqref{eq:def_psi} and \eqref{eq:def_Bernstein} and satisfy  \eqref{eq:standing_assumption} and \eqref{Analicity_of_the_mechanismsp} with $m=a ( \alpha \eta_\psi^{\alpha -1} + \gamma )$ and
\begin{equation*}
\begin{split}
& \sigma^2=
\begin{cases}
0 & \text{if $\alpha<2$}, \\
a & \text{if $\alpha=2$},
\end{cases}
\quad
\Pi(d r)=
\begin{cases}
\frac{a\alpha(\alpha-1)}{\Gamma(2-\alpha)} e^{-\eta_\psi r} r^{-\alpha-1}d r & \text{if $\alpha<2$}, \\
0 & \text{if $\alpha=2$},
\end{cases}
%\quad
%m=
%\begin{cases}
%a ( \alpha \eta_\psi^{\alpha -1} + \gamma )% & \text{if $\alpha<2$}, \\
%a(2\eta_\psi+\gamma) & \text{if $\alpha=2$},
%\end{cases}
\\
& b=
\begin{cases}
0 & \text{if $\beta<1$}, \\
ca & \text{if $\beta=1$},
\end{cases}
\quad
\mu(d r)=
\begin{cases}
\frac{c a}{\Gamma(-\beta)} e^{-\eta_\phi r} r^{-\beta-1}d r & \text{if $\beta< 0$}, \\
\frac{c a\beta}{\Gamma(1-\beta)} e^{-\eta_\phi r} r^{-\beta-1}d r & \text{if $0<\beta<1$}, \\
0 & \text{if $\beta\in\{0,1\}$},
\end{cases}
\end{split}
\end{equation*}
where for $a>0$, $\Gamma(a)=\int_0^\infty e^{-x} x^{a-1} d x$ is the gamma function. Note that $\bar{\mu}(0^+)=\infty$ if $0<\beta<1$. We fix $t>T_0$ and $x\geq 0$ and investigate the regularity of $y\mapsto p_t(x,y)$.

 If $\alpha<\beta+1$, then by combining Proposition \ref{lem:kap}\eqref{item:lem_kap_regvar} and Theorem \ref{thm:regularity}, we have $y\mapsto p_t(x,y)\in C^{\infty}(\mathbb R)$. Suppose now $\alpha>\beta+1$. Then as $\overline{\Pi}$ is completely monotone and $\overline{\mu}\in C^\infty(\mathbb R^+)$ it follows from Remark \ref{remark_smoothness_kappa} and Proposition \ref{lem:diff_kappa}\eqref{it:sigma_nul_W} (if $\alpha<2$ and $\beta>0$), Proposition \ref{lem:diff_kappa}\eqref{it:mu_finite} (if $\alpha<2$ and $\beta\leq 0$) or Proposition \ref{lem:diff_kappa}\eqref{it:sigma_pos} (if $\alpha=2$) that $W,\kappa\in C^\infty(\mathbb R^+)$ and $\kappa^{(1)}\in L^1_{loc}(\mathbb R^+)$ and thus by Theorem \ref{thm:regularity}, $y\mapsto p_t(x,y)\in C^{\infty}(\mathbb R^+)$, provided $\overline{\kappa}(0^+)=\Si(0^+)<\infty$. To show the latter, note that by Proposition \ref{lem:kap}\eqref{item:lem_kap_sigma}, $\overline{\kappa}(0^+)=0$ if $\alpha=2$ and by Proposition \ref{lem:kap}\eqref{item:lem_kap_regvar},  $\overline{\kappa}(0^+)=0$ if $\alpha<2$ and $\beta\leq 0$. If $\alpha<2$ and $\beta>0$, let $\ell( u)=\frac{ u^\alpha}{\psi( u)}$. Then  $\lim_{ u\to\infty}\ell( u)=1/a$ and thus in particular it is a slowly varying function at infinity.
Since
 $$\int_0^\infty   e^{- u y} W^{(1)}(y) d y = \frac{ u}{\psi( u)}=  u^{1-\alpha}\ell( u),$$
 and $W^{(1)}$ is non-increasing (since it is completely monotone), we have by \cite[Theorem 1.7.1' and Theorem 1.7.2b]{BinghamGoldieTeugels87} that
 $$\lim_{y\downarrow 0} \frac{W^{(1)}(y)}{ y^{\alpha-2}} =\frac{1}{a\Gamma(\alpha-1)}.$$
Let $\epsilon>0$ be arbitrary and choose $\delta>0$ be such that $\frac{a\Gamma(\alpha-1) W^{(1)}(y)}{ y^{\alpha-2}} \in [1-\epsilon,1+\epsilon]$ and $e^{-\eta_\phi y}\in [1-\epsilon,1]$ for all $0<y<\delta$. Then, writing $C_{\epsilon}= \frac{c\beta(1+\epsilon)}{\Gamma(\alpha-1)\Gamma(1-\beta)}$, we have, for $0<y<\delta$,
\begin{eqnarray}
\kappa(y) &= & \int_0^y W^{(1)}(y-r)\bar{\mu}(r) dr  \label{upperbound_kappa0}
 \leq  \int_0^y (y-r)^{\alpha-2} \int_r^\infty e^{-\eta_\phi u} u^{-\beta-1}d u d r  \\
&= & C_{\epsilon}\int_0^y (y-r)^{\alpha-2} \left( \int_r^y e^{-\eta_\phi u} u^{-\beta-1}d u  + \int_y^\infty e^{-\eta_\phi u} u^{-\beta-1}d u \right) d r  \nonumber \\
& \leq &
C_{\epsilon}\left(\int_0^y (y-r)^{\alpha-2} \int_r^y  u^{-\beta-1}d u d r  +    \int_0^y (y-r)^{\alpha-2} \int_y^\infty e^{-\eta_\phi u} u^{-\beta-1}d u d r  \right) \nonumber \\
&= &
C_{\epsilon} \left( \frac 1\beta \int_0^y (y-r)^{\alpha-2}  \left( r^{-\beta}-y^{-\beta} \right) d r     +    \frac{y^{\alpha-1}}{\alpha-1}  \int_y^\infty e^{-\eta_\phi u} u^{-\beta-1}d u  \right) \nonumber \\
&= &
C_{\epsilon} \left( \left( \frac{\Gamma(\alpha-1)\Gamma(1-\beta)}{\beta\Gamma(\alpha-\beta)} - \frac 1{(\alpha-1)\beta} \right) y^{\alpha-\beta-1}       +    \frac{y^{\alpha-1}}{\alpha-1}  \int_y^\infty e^{-\eta_\phi u} u^{-\beta-1}d u  \right),\nonumber
\end{eqnarray}
where for the last equality we used, for $\eta_1<1$ and $\eta_2<1$, the identity
\begin{equation*}
\int_0^y (y-r)^{-\eta_1}r^{-\eta_2} d r = \frac{\Gamma(1-\eta_1)\Gamma(1-\eta_2)}{\Gamma(2-\eta_1-\eta_2)} y^{1-\eta_1-\eta_2},
 \end{equation*}
 which can be easily proved via Laplace transforms. Similarly, writing $  \overline{C}_{-\epsilon} =\frac{\Gamma(\alpha-1)\Gamma(1-\beta)}{\beta\Gamma(\alpha-\beta)} - \frac 1{(\alpha-1)\beta} (1-\epsilon)$, we have the lower bound
\begin{eqnarray}\label{lowerbound_kappa0}
\kappa(y) &\geq & C_{-\epsilon} \left( \overline{C}_{-\epsilon} y^{\alpha-\beta-1}+   \frac{y^{\alpha-1}}{\alpha-1}  \int_y^\infty e^{-\eta_\phi u} u^{-\beta-1}d u  \right).
\end{eqnarray}
Next, by means of l'H\^opital's rule, we observe that
\begin{equation}\label{hopital_kappa0}
\frac{y^{\alpha-1}}{\alpha-1}  \int_y^\infty e^{-\eta_\phi u} u^{-\beta-1}d u
 \stackrel{0}{\sim}  \frac{e^{-\eta_\phi y} y^{-\beta-1}}{(\alpha-1)^2 y^{-\alpha} } \stackrel{0}{\sim} \frac{y^{\alpha-\beta-1}}{(\alpha-1)^2}.
\end{equation}
%\added[id=RL]{Explain notation $\stackrel{0}{\sim}$?}
%
% and so choosing $\delta>0$ such that $\alpha-\beta-2\delta>1$, we see that there  exist $K>0$ and $\bar y>0$ such that
%\begin{equation*}
%\kappa(y) \leq K y^{\alpha-\beta-2\delta-1}, \quad 0<y<\bar y.
%\end{equation*}
Recalling that we are considering the case $\alpha>\beta+1$, it follows from \eqref{upperbound_kappa0} and \eqref{hopital_kappa0} that  $\overline\kappa(0^+)=0$ also when $\alpha<2$ and $\beta>0$. Thus, we conclude  that $y\mapsto p_t(x,y)\in C^{\infty}(\mathbb R+)$   if $\alpha>\beta+1$. Now we consider the remaining case where $\alpha=\beta+1$. If $\alpha=2$, then by Lemma \ref{lem:kap}\eqref{item:lem_kap_sigma}, $\overline{\kappa}(0^+)=\underline{\kappa}(0^+) = c$. If $\alpha<2$, then by \eqref{upperbound_kappa0}-\eqref{hopital_kappa0}, we have for any $\epsilon>0$,
\begin{equation*}
\overline\kappa(0^+)\leq c(1+\epsilon), \quad \underline{\kappa}(0^+) \geq c(1-\epsilon)^2  + \epsilon  \frac{c\beta(1-\epsilon  )}{\Gamma(\alpha-1)\Gamma(1-\beta)(\alpha-1)^2}. %\left(    (1-\epsilon) - 1 \right)
\end{equation*}
Hence $\overline{\kappa}(0^+)=\Si(0^+)=c$.  Then, by Theorem \ref{thm:regularity},  $ y\mapsto  p_t(x,y)\in  C^{k-1}(\mathbb R)\cap C^k(\mathbb R^+)$ for all integers $k\in [1,c)$. To deal with additional  regularity on $\mathbb R^+$, note that if $\alpha=2$, Theorem \ref{thm:regularity}\eqref{it:b2}   applies via Lemma \ref{lem:diff_kappa}\eqref{it:sigma_pos} and Remark \ref{remark_smoothness_kappa} and thus  $y\mapsto p_t(x,y)\in C^{\infty}(\mathbb R^+)$. However, if $\alpha<2$ it is not obvious to us how to prove that $\kappa^{(1)}\in L^1_{loc}(\mathbb R^+)$ since the integral condition in Lemma \ref{lem:diff_kappa}\eqref{it:sigma_nul_W} is not satisfied.
  Therefore we restrict ourselves to the special case where $\eta_\psi=\eta_\phi$ and $\gamma=-\eta_\psi^{\alpha-1}$, which has also been considered in \cite{CLP-F} and \cite{Ogura-70}. In that case, $\frac{\phi( u)}{\psi( u)} =\frac{c}{ u+\eta_\psi}$ and so by Laplace inversion, $\kappa(y)=c e^{-\eta_\psi y}$. Then, recalling that $W$ is in $C^{\infty}(\mathbb R^+)$, Theorem \ref{thm:regularity} \eqref{it:defq}  yields $y\mapsto p_t(x,y)\in C^{\infty}(\mathbb R^+)$.

\subsection{Example of non-smoothness}
 This example provides an instance when  the absolutely continuous part of the transition density  is not in $C^\infty(\R_+)$ in general.
Let us consider the case where $\sigma=1$, $m=1$ and $\Pi(d r) =\delta_1(d r)$ and $b=0$, $\mu\equiv 0$, that is, $P$ is a CB semigroup. Note that from Proposition \ref{prop:transf}(ii), one may construct a CBI-semigroup whose transition kernel has the same  smoothness  properties as the kernel of the CB semigroup. We fix $t>T_0$ and $x\geq 0$. Since $W\in C^2(\R_+)$ when $\sigma>0$, see in \cite[Theorem 1]{Chan-Kyp-Savov}, we know  by Theorems \ref{thm:regularity} and \ref{thm:spectral-expansion} and Lemma \ref{elm:abs} that in this case $P_t(x, d y) = \delta_0(d y)+p_t(x, y)d y$, where
\begin{equation*}
y\mapsto p_t(x,   y) = \sum_{n=0}^\infty \mathcal L_n(x) e^{- \lambda_n t} \mathcal W_{ n}(y) \in C^2(\R_+),
\end{equation*}
 with, with the obvious notation,
\begin{equation}\label{example_CB_1}
\begin{split}
p_t^{(2)}(x,y) = & \sum_{n=0}^\infty \mathcal L_n(x) e^{-\lambda_n t} \mathcal W^{(2)}_{ n}(y) \\
= & \sum_{n=0}^\infty \mathcal L_n(x) e^{-\lambda_n t} \left( \mathcal W^{(2)}_{ n}(y) + \eign \frac{W^{(2)}(y)}{y} \right) -  \sum_{n=0}^\infty \mathcal L_n(x) e^{-\lambda_n t}  \eign \frac{W^{(2)}(y)}{y}.
\end{split}
\end{equation}
We are going to show that $y\mapsto p_t(x,y)\notin C^3(\R_+)$.  By \eqref{scale_2ndderiv_sigma}, we have, for $y>0$,
\begin{equation}\label{example_CB_2}
W^{(2)}(y) = \overline{\overline{\Pi}}(y) + \sum_{n=2}^\infty   (- 1)^n \overline{\overline{\Pi}}{}^{*n}(y).
\end{equation}
As $\overline{\overline{\Pi}}(y) = (1-y)\mathbb{I}_{\{0<y\leq 1\}}$ is an absolutely continuous function with bounded density, we get  by Lemma \ref{lem_conv_diff_onederiv}, that $\overline{\overline{\Pi}}{}^{*n} \in C^{1}(\R_+)$ for $n\geq 2$. It is then an easy exercise to show that the infinite sum on the right hand side of \eqref{example_CB_2} is in $C^1(\R_+)$. Therefore, $W^{(2)}$ is absolutely continuous on $\R_+$ (with a bounded density) and differentiable on $\R_+\backslash\{1\}$ but not  at $y=1$, since $\overline{\overline{\Pi}}(1^-)=-1$ and $\overline{\overline{\Pi}}(1^+)=0$. By \eqref{pderiv_mathcalW},
\begin{multline*}
- \left( \mathcal W^{(2)}_{ n}(y) + \lambda_n \frac{W^{(2)}(y)}{y} \right)
= 2\frac{\mathcal W_{ n}(y)}{y} + \frac{\lambda_n}{y} \Bigg(  \int_0^{\frac 12 y}   W^{(2)}(y-r)\mathcal W_{ n}(r)  d z \\ + \int_0^{\frac 12 y}  \mathcal W_{ n}^{(1)} (y-r)  W^{(1)}(r) d r
 +  W^{(1)}(\tfrac{y}{2}) \mathcal W_{ n} (\tfrac{y}{2})   \Bigg)
\end{multline*}
and thus by invoking Lemma  \ref{lem_conv_diff_onederiv} for the first integral, $y\mapsto y\mathcal W^{(2)}_{n}(y) + \lambda_n W^{(2)}(y) \in C^{1}(\R_+)$. Following the proof of Theorem \ref{thm:regularity}\eqref{it:defq} in Section \ref{sec:proof_one}, we deduce that the first infinite sum on the right hand side of \eqref{example_CB_1} is in $C^1(\R_+)$. But the second infinite sum is not differentiable at $y=1$ and  thus $y\mapsto p_t(x,y)\notin C^3(\R_+)$.

\bibliographystyle{plain}
%\bibliography{./ref}

\end{document}